\documentclass[11pt]{article}

\usepackage{fullpage}
\usepackage{graphicx}

\usepackage{multirow}

\usepackage{caption}
\usepackage{subcaption}

\usepackage{complexity}
\usepackage[linesnumbered,lined,boxed,noend]{algorithm2e}

\usepackage{amsmath,amssymb}
\usepackage[amsmath,thmmarks]{ntheorem}

\usepackage[blocks]{authblk}

\usepackage{hyperref}

\usepackage{xcolor}
\definecolor{dark-red}{rgb}{0.6,0.15,0.15}
\definecolor{dark-blue}{rgb}{0.15,0.15,0.4}
\definecolor{medium-blue}{rgb}{0,0,0.5}
\definecolor{gray}{rgb}{0.5,0.5,0.5}

\hypersetup{
    colorlinks, linkcolor={dark-blue},
    citecolor={medium-blue}, urlcolor={dark-blue}
}

\usepackage{enumitem}
\usepackage{enumitem}
\setlist[enumerate,1]{label={(\roman*)}, ref={(\roman*)}}
\setlist[enumerate,2]{label={(\alph*)}, ref={(\roman{enumi}.\alph*)}}
\setlist[enumerate,3]{label={(\arabic*)}, ref={(\roman{enumi}.\alph{enumii}.\arabic*)}}

\newcommand\plainparagraph[1]{
	\medskip
	\noindent\textbf{#1}%
}

\usepackage{todonotes}
\newcommand{\noop}[1]{}

\usepackage{todonotes}
\newcommand\yes{\textsc{Yes}}
\newcommand\no{\textsc{No}}

\usepackage{boxedminipage}
\usepackage{framed}
\usepackage{xspace}
\usepackage{xifthen}
\usepackage{tabularx}
\usepackage{tikz} 
\usetikzlibrary{calc}

\newlength{\RoundedBoxWidth}
\newsavebox{\GrayRoundedBox}
\newenvironment{GrayBox}[1]%
   {\setlength{\RoundedBoxWidth}{.93\textwidth}
    \def\boxheading{#1}
    \begin{lrbox}{\GrayRoundedBox}
       \begin{minipage}{\RoundedBoxWidth}}%
   {   \end{minipage}
    \end{lrbox}
    \begin{center}
    \begin{tikzpicture}%
       \node(Text)[draw=black!20,fill=white,rounded corners,%
             inner sep=2ex,text width=\RoundedBoxWidth]%
             {\usebox{\GrayRoundedBox}};
        \coordinate(x) at (current bounding box.north west);
        \node [draw=white,rectangle,inner sep=3pt,anchor=north west,fill=white] 
        at ($(x)+(6pt,.75em)$) {\boxheading};
    \end{tikzpicture}
    \end{center}}     

\newenvironment{defproblemx}[2][]{\noindent\ignorespaces%
                                \FrameSep=6pt%
                                \parindent=0pt%
                \vspace*{-.5em}
                \ifthenelse{\isempty{#1}}{%
                  \begin{GrayBox}{\textsc{#2}}%
                }{%
                  \begin{GrayBox}{\textsc{#2} parameterized by~{#1}}%
                }
                \begin{tabular*}{\textwidth}{@{\hspace{.1em}} >{\itshape} p{1.8cm} p{0.8\textwidth} @{}}%
            }{
                \end{tabular*}%
                \end{GrayBox}%
                \ignorespacesafterend
            }  

\newcommand{\fancyparproblemdef}[4]{
  \begin{defproblemx}[#3]{#1}
    Input:  & #2 \\
    Question: & #4
  \end{defproblemx}
}%

\newcommand{\fancyproblemdef}[3]{
  \begin{defproblemx}{#1}
    Input:  & #2 \\
    Question: & #3
  \end{defproblemx}
}%

\newcommand*{\defeq}{\mathrel{\vcenter{\baselineskip0.5ex \lineskiplimit0pt
                     \hbox{\scriptsize.}\hbox{\scriptsize.}}}%
                     =}

\newcommand*{\eqdef}{=\mathrel{\vcenter{\baselineskip0.5ex \lineskiplimit0pt
                     \hbox{\scriptsize.}\hbox{\scriptsize.}}}%
                     }

\newcommand\card[1]{\lvert#1\rvert}

\newcommand\abs[1]{\card{#1}}

\newcommand\bN{\mathbb{N}}

\theoremseparator{.}
\theorembodyfont{\itshape}
\theoremheaderfont{\normalfont\bfseries}
\newtheorem{theorem}{Theorem}[section]
\newtheorem{lemma}[theorem]{Lemma}
\newtheorem{proposition}[theorem]{Proposition}
\newtheorem{corollary-bf}[theorem]{Corollary}

\theorembodyfont{\normalfont}
\newtheorem{definition}[theorem]{Definition}
\newtheorem{reduction}{Reduction}

\theoremseparator{.}
\theoremheaderfont{\itshape}
\theorembodyfont{\itshape}
\newtheorem{corollary-wp}[theorem]{Corollary}

\newtheorem{nestedclaim}{Claim}[theorem]

\theoremsymbol{}

\newtheorem{corollary}[theorem]{Corollary}

\theorembodyfont{\normalfont}
\newtheorem{remark}[theorem]{Remark}
\newtheorem{observation}[theorem]{Observation}

\theoremsymbol{$\bullet$}
\theoremheaderfont{\normalfont}
\newtheorem{case}{Case}

\theoremsymbol{$\circ$}

\theoremsymbol{}

\newcommand\claimqed{$\lrcorner$}
\theoremsymbol{}

\theoremstyle{nonumberplain}
\theoremheaderfont{\itshape}
\theorembodyfont{\normalfont}
\theoremsymbol{$\square$}
\newtheorem{proof}{Proof}
\theoremsymbol{\claimqed}
\newtheorem{claimproof}{Proof}

\theoremsymbol{}
\newtheorem{claim-nn}{Claim}
\newtheorem{corollary-nn}{Corollary}

\theoremheaderfont{\bfseries}
\newtheorem{theorem-nn}{Theorem}
\newtheorem{lemma-nn}{Lemma}
\newtheorem{proposition-nn}{Proposition}

\newenvironment{re-proof}[1]
	{\begin{proof} 
	\begingroup\renewcommand\thetheorem{\getrefnumber{#1}}}
	{\endgroup
	\end{proof}}

\newenvironment{re-theorem}[1]
	{\begin{theorem-nn}[\ref{#1}]
		\setcounter{nestedclaim}{0}}
	{\end{theorem-nn}}
	
\newenvironment{re-lemma}[1]
	{\begin{lemma-nn}[\ref{#1}]
		\setcounter{nestedclaim}{0}}
	{\end{lemma-nn}}
	
\newenvironment{re-proposition}[1]
	{\begin{proposition-nn}[\ref{#1}]
		\setcounter{nestedclaim}{0}}
	{\end{proposition-nn}}
	
\newenvironment{re-claim}[1]
	{\begin{claim-nn}[\ref{#1}]}
	{\end{claim-nn}}

\usepackage[nameinlink]{cleveref}

\crefname{definition}{Definition}{Definitions}
\crefname{theorem}{Theorem}{Theorems}
\crefname{lemma}{Lemma}{Lemmas}
\crefname{corollary}{Corollary}{Corollaries}
\crefname{corollary-wp}{Corollary}{Corollaries}
\crefname{corollary-bf}{Corollary}{Corollaries}
\crefname{observation}{Observation}{Observations}
\crefname{nestedobservation}{Observation}{Observations}
\crefname{fact}{Fact}{Facts}
\crefname{remark}{Remark}{Remarks}
\crefname{conjecture}{Conjecture}{Conjectures}
\crefname{proposition}{Proposition}{Propositions}
\crefname{figure}{Figure}{Figures}
\crefname{table}{Table}{Tables}
\crefname{section}{Section}{Sections}
\crefname{subsection}{Subsection}{Subsections}
\crefname{subsubsection}{Subsection}{Subsections}
\crefname{algorithm}{Algorithm}{Algorithms}
\crefname{example}{Example}{Examples}
\crefname{note}{Note}{Notes}
\crefname{claim}{Claim}{Claims}
\crefname{nestedclaim}{Claim}{Claims}
\crefname{enumi}{}{}
\crefname{equation}{}{}
\crefname{reduction}{Reduction}{Reductions}
\crefname{subreduction}{Reduction}{Reductions}
\crefname{property}{Property}{Properties}
\crefname{case}{Case}{Cases}
\crefname{subcase}{Case}{Cases}
\crefname{subsubcase}{Case}{Cases}
\crefname{process}{Process}{Processes}

\newcommand\calA{\mathcal{A}}
\newcommand\calB{\mathcal{B}}
\newcommand\calF{\mathcal{F}}

\newcommand\calM{\mathcal{M}}
\newcommand\calO{\mathcal{O}}
\newcommand\cO{\calO}
\newcommand\calP{\mathcal{P}}
\newcommand\calS{\mathcal{S}}
\newcommand\calT{\mathcal{T}}

\newcommand\bd{\mathrm{bd}}
\newcommand\bdd[2]{\bd(#1, #2)}

\newcommand\labeling{\ensuremath{\lambda}}

\newcommand\terminals{\mathcal{X}}
\newcommand\domains{\mathcal{U}}

\newcommand\probDP{\textsc{Disjoint Paths}}
\newcommand\probTDP{\textsc{Totally Disjoint Paths}}
\newcommand\SRDP{\textsc{Set-Restricted Disjoint Paths}}
\newcommand\SRTDP{\textsc{Set-Restricted Totally Disjoint Paths}}
\newcommand\SRDCS{\textsc{Set-Restricted Disjoint Connected Subgraphs}}

\newcommand\mathsc[1]{\mbox{\textsc{#1}}}

\newcommand\obswp{\mathbb{O}}

\usepackage{enumitem,tabularx}
\usepackage{tikz}
\usetikzlibrary{graphs,graphs.standard}
\usetikzlibrary{calc}
\usetikzlibrary{decorations.pathreplacing}

\tikzset{
	auto,
	node distance = 1cm,
	vtx/.style={
		circle,
		fill=black,
		draw,
		thick,
		minimum size=.11cm,
		inner sep=0pt
	},
	bag/.style={
		shape=circle,
		draw,
		minimum width=2.25cm,
	},
	smallbag/.style={
		shape=circle,
		draw,
		minimum width=1.5cm,
	},
	edge/.style = {draw},
	nonedge/.style = {draw, dotted},
	diredge/.style = { ->},
	edgelabel/.style = {scale=0.75, midway},
	highlighta/.style = {color=lightgray},
	highlightb/.style = {color=lightgray},
	figlabel/.style = {scale=1.2}
}

\author[1,2]{Jungho Ahn} 
\author[3]{Lars Jaffke}
\author[4,2]{O-joung Kwon}
\author[3]{Paloma T.\ Lima}

\affil[1]{Department of Mathematical Sciences, KAIST, Daejeon,~South~Korea}
\affil[2]{Discrete Mathematics Group, Institute for Basic Science (IBS), Daejeon, South~Korea}
\affil[3]{Department of Informatics, University of Bergen, Bergen, Norway}
\affil[4]{Department of Mathematics, Incheon National University, Incheon,~South~Korea}
\affil[ ]{\small \textit{Email addresses:}  \texttt{junghoahn@kaist.ac.kr},
  \texttt{lars.jaffke@uib.no},
  \texttt{ojoungkwon@gmail.com},
  \texttt{Paloma.Lima@uib.no }}

\title{Well-partitioned chordal graphs: obstruction set and disjoint paths%
	\thanks{J.\,A.\ and O.\,K.\ are supported by IBS-R029-C1.
		O.\,K.\ is also supported by the National Research Foundation of Korea (NRF) grant funded by the Ministry of Education (No.\ NRF-2018R1D1A1B07050294). 
		L.\,J.\ is supported by the Bergen Research Foundation (BFS), grant number 810564.
		P.\,T.\,L.\ is supported by the Norwegian Research Council via the project ``CLASSIS'', grant number 249994.}} 

\begin{document}

\maketitle

\begin{abstract}
	We introduce a new subclass of chordal graphs that generalizes split graphs, 
	which we call well-partitioned chordal graphs.
	Split graphs are graphs that admit a partition of the vertex set into cliques that can 
	be arranged in a star structure, the leaves of which are of size one.
	Well-partitioned chordal graphs are a generalization of this concept in the following two ways.
	First, the cliques in the partition can be arranged in a tree structure,
	and second, each clique is of arbitrary size.
	We provide a characterization of well-partitioned chordal graphs by forbidden induced subgraphs,
	and give a polynomial-time algorithm that given any graph, 
	either finds an obstruction, 
	or outputs a partition of its vertex set that asserts that the graph is well-partitioned chordal.
	We demonstrate the algorithmic use of this graph class by showing that two variants of the problem of finding
	pairwise disjoint paths between $k$ given pairs of vertices 
	is in \FPT{} parameterized by $k$ on well-partitioned chordal graphs, 
	while on chordal graphs, these problems are only known to be in \XP.
	From the other end, we observe that there are problems that are polynomial-time solvable on split graphs,
	but become \NP-complete on well-partitioned chordal graphs.
\end{abstract}

\section{Introduction}
A central methodology in the study of the complexity of computationally hard graph problems
is to impose additional structure on the input graphs, 
and determine if the additional structure can be exploited in the design of an efficient algorithm.
Typically, one restricts the input to be contained in a \emph{graph class}, 
which is a set of graphs that share a common structural property.
For example, the class of \emph{forests} is the class of graphs that do not contain a cycle.
Following the establishment of the theory of \NP-hardness, 
numerous problems were investigated in specific classes of graphs;
either providing a polynomial-time algorithm for a problem $\Pi$ on a specific graph class, 
while $\Pi$ is \NP-hard in a more general setting,
or showing that $\Pi$ remains \NP-hard on a graph class.
We refer to the textbooks~\cite{GRAPHCLASSES,GOLUMBIC} for a detailled introduction to the subject.
A key question in this field is to find for a given problem $\Pi$ that is hard on a graph class $\calA$,
a subclass $\calB \subsetneq \calA$ such that $\Pi$ is efficiently solvable on $\calB$.
Naturally, the goal is to narrow down the gap $\calA \setminus \calB$ as much as possible,
and several notions of hardness/efficiency can be applied.
For instance, we can require our target problem to be \NP-hard on $\calA$ and polynomial-time solvable on $\calB$;
or, from the viewpoint of parameterized complexity~\cite{CYGANETAL,DOWNEYFELLOWS},
we require a target parameterized problem $\Pi$ to be \W[1]-hard on $\calA$, while $\Pi$ is in \FPT{} on $\calB$,
or a separation in the kernelization complexity~\cite{KERNELIZATION} of $\Pi$ between $\calA$ and $\calB$.

Chordal graphs
are arguably one of the main characters in the algorithmic study of graph classes.
They find applications for instance in computational biology~\cite{SempleS2003}, optimization~\cite{Lieven2015},
and sparse matrix computations~\cite{GTSPARSEMATRIX}.
Split graphs
are an important subclass of chordal graphs.
The complexities of computational problems on chordal and split graphs often coincide,
see e.g.~\cite{DOMSETSPLIT,MAXCUTSPLIT,Paola1999,HAMCYCSPLIT};
however, this is not always the case. 
For instance, several variants of graph (vertex) coloring problems are polynomial-time solvable on split graphs
and \NP-hard on chordal graphs, see the works of Havet et al.~\cite{HSS12}, and of Silva~\cite{Sil19}.
Also, the \textsc{Sparest $k$-subgraph}~\cite{Watrigant2016} and \textsc{Densest $k$-subgraph}~\cite{CorneilP1984} 
problems are polynomial-time solvable on split graphs and \NP-hard on chordal graphs.
Other problems, for instance the \textsc{Tree $3$-Spanner} problem~\cite{Brandstadt04},
are easy on split graphs, while their complexity on chordal graphs is still unresolved.

In this work, we introduce the class of \emph{well-partitioned chordal graphs}, 
a subclass of chordal graphs that generalizes split graphs,
which can be used as a tool for narrowing down complexity gaps for problems that are hard on chordal graphs, 
and easy on split graphs.
The definition of well-partitioned chordal graphs is mainly motivated by a property of split graphs: 
the vertex set of a split graph can be partitioned into sets that can be viewed as a central clique of arbitrary size 
and cliques of size one that have neighbors only in the central clique.
Thus, this partition has the structure of a star.
Well-partitioned chordal graphs relax these ideas in two ways: 
by allowing the parts of the partition to be arranged in a tree structure instead of a star, 
and by allowing the cliques in each part to have arbitrary size.
The interaction between adjacent parts $P$ and $Q$ remains simple: it induces a complete bipartite graph 
between a subset of $P$, and a subset of $Q$.
Such a tree structure is called a partition tree, and we give an example of a well-partitioned chordal graph in Figure~\ref{fig:ex-wpc}.
We formally define this class in Section~\ref{sec3}.
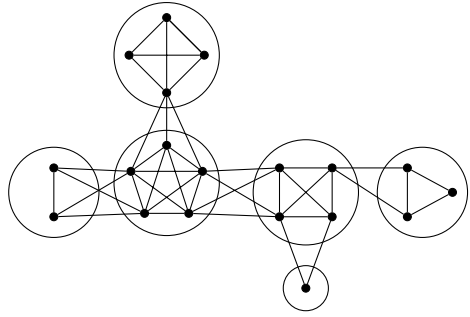
\begin{figure}
	\centering
    \tikzstyle{v}=[circle, draw, solid, fill=black, inner sep=0pt, minimum width=3pt]
	\begin{tikzpicture}[scale=1]
		\foreach \x in {0,1,2,3,4} {
    			\draw (\x*72+90:.5) node [v](v\x){};
    		}
    		\draw (v0)--(v1)--(v2)--(v3)--(v4)--(v0);
		\draw (v0)--(v2)--(v4)--(v1)--(v3)--(v0);
		\draw (0,1.2) node[v](v6){};
		\draw (0.5,1.7) node[v](v7){};
		\draw (0,2.2) node[v](v8){};
		\draw (-0.5,1.7) node[v](v9){};
		\draw (v1)--(v6)--(v0);
		\draw (v6)--(v4);
		\draw (v6)--(v7)--(v8)--(v9)--(v6);
		\draw (v6)--(v8)--(v7)--(v9);
		\draw (-1.5,0.2) node[v](v10){};
		\draw (-1.5,-0.45) node[v](v11){};
		\draw (v10)--(v11)--(v1)--(v10);
		\draw (v10)--(v2)--(v11);
		\draw (1.5,0.2) node[v](w1){};
		\draw (1.5,-0.45) node[v](w2){};
		\draw (2.2,0.2) node[v](w3){};
		\draw (2.2,-0.45) node[v](w4){};
		\draw (1.85,-1.4) node[v](w5){};
		\draw (w1)--(v4)--(w2)--(v3)--(w1);
		\draw (w1)--(w2)--(w3)--(w4)--(w1)--(w3);
		\draw (w2)--(w4)--(w5)--(w2);
		\draw (3.2,0.2) node[v](w6){};
		\draw (3.2,-0.45) node[v](w7){};
		\draw (3.8,-0.125) node[v](w8){};
		\draw (w6)--(w7)--(w8)--(w6)--(w3)--(w7);
		
		\draw[black, thin] (0,0) circle (0.7);
		\draw[black, thin] (0,1.7) circle (0.7);
		\draw[black, thin] (-1.5,-0.125) circle (0.6);
		\draw[black, thin] (1.85,-0.125) circle (0.7);
		\draw[black, thin] (3.4,-0.125) circle (0.6);
		\draw[black, thin] (w5) circle (0.3);
	\end{tikzpicture}
	\caption{A well-partitioned chordal graph.}
	\label{fig:ex-wpc}
\end{figure}

That said, it is not difficult to observe that the graphs constructed in the \NP-hardness proofs in the works~\cite{HSS12,Sil19}
are in fact well-partitioned chordal graphs --- we immediately narrowed down the complexity gaps
of these problems from 
$\mathsc{Chordal} \setminus \mathsc{Split} \mbox{ to } \mathsc{Well-partitioned chordal} \setminus \mathsc{Split}.$

\begin{figure}
	\centering
	\scalebox{0.8}{
		\newcommand\drawedge[2]{
			\draw[edge] #1 to #2;
		}
		\newcommand\drawdottededge[2]{
			\draw[nonedge] #1 to #2;
		}
		\def\nameindent{-.75}
		\def\nameindentsmall{-0.5}
		\newcommand\obstruction{O}
		\newcommand\ABar{$\obstruction_1$}
		\newcommand\FourFan{$\obstruction_2$}
		\newcommand\AhnOne{$\obstruction_3$}
		\newcommand\AhnTwo{$\obstruction_4$}
		\newcommand\Holes{$H_k$, $k \ge 4$}
		\begin{tikzpicture}	
			\begin{scope}[]
				\node [] (name) at (0.75, \nameindent) {\ABar};

				\node [vtx] (a) at (0, 0) {};
				\node [vtx] (b) at (1.5, 0) {};
				\node [vtx] (c) at (0, 1) {};
				\node [vtx] (d) at (1.5, 1) {};
				\node [vtx] (e) at (0, 2) {};
				\node [vtx] (f) at (1.5, 2) {};
	
				\drawedge{(a)}{(b)}
				\drawedge{(a)}{(c)}
				\drawedge{(a)}{(d)}
				\drawedge{(b)}{(d)}
				\drawedge{(c)}{(d)}
				\drawedge{(c)}{(e)}
				\drawedge{(c)}{(f)}
				\drawedge{(d)}{(f)}
				\drawedge{(e)}{(f)}
			\end{scope}
	
	
	
			\begin{scope}[xshift=3.5cm]
				\node[] (name) at (0.75, \nameindent) {\FourFan};
		
				\node [vtx] (a) at (0, 0) {};
				\node [vtx] (b) at (1.5, 0) {};
				\node [vtx] (c) at (0, 1) {};
				\node [vtx] (d) at (1.5, 1) {};
				\node [vtx] (e) at (0, 2) {};
				\node [vtx] (f) at (1.5, 2) {};
		
				\drawedge{(a)}{(b)}
				\drawedge{(a)}{(c)}
				\drawedge{(a)}{(d)}
				\drawedge{(b)}{(d)}
				\drawedge{(c)}{(d)}
				\drawedge{(c)}{(e)}
				\drawedge{(d)}{(e)}
				\drawedge{(d)}{(f)}
				\drawedge{(e)}{(f)}
			\end{scope}
	
	
			\begin{scope}[xshift=7cm]
				\node [] (name) at (1.5, \nameindent) {\AhnOne};
		
				\node [vtx] (a) at (0, 0) {};
				\node [vtx] (b) at (1.5, 0) {};
				\node [vtx] (c) at (3, 0) {};
				\node [vtx] (d) at (0.75, 1) {};
				\node [vtx] (e) at (2.25, 1) {};
				\node [vtx] (f) at (1.5, 2) {};
				\node [vtx] (g) at (3, 2) {};
		
				\drawedge{(a)}{(b)}
				\drawedge{(a)}{(d)}
				\drawedge{(a)}{(e)}
				\drawedge{(b)}{(c)}
				\drawedge{(b)}{(d)}
				\drawedge{(b)}{(e)}
				\drawedge{(c)}{(d)}
				\drawedge{(c)}{(e)}
				\drawedge{(d)}{(e)}
				\drawedge{(d)}{(f)}
				\drawedge{(e)}{(f)}
				\drawedge{(e)}{(g)}
				\drawedge{(f)}{(g)}
			\end{scope}
	
			\begin{scope}[xshift=12cm]
				\node [] (name) at (1.75, \nameindent) {\AhnTwo};
		
				\node [vtx] (a) at (1, 0) {};
				\node [vtx] (b) at (2.5, 0) {};
				\node [vtx] (c) at (0, 0.5) {};
				\node [vtx] (d) at (3.5, 0.5) {};
				\node [vtx] (e) at (1, 1) {};
				\node [vtx] (f) at (2.5, 1) {};
				\node [vtx] (g) at (0, 1.5) {};
				\node [vtx] (h) at (1.75, 2) {};
				\node [vtx] (i) at (3.5, 1.5) {};
		
				\drawedge{(a)}{(b)}
				\drawedge{(a)}{(c)}
				\drawedge{(a)}{(e)}
				\drawedge{(a)}{(f)}
				\drawedge{(a)}{(h)}
				\drawedge{(b)}{(d)}
				\drawedge{(b)}{(e)}
				\drawedge{(b)}{(f)}
				\drawedge{(b)}{(h)}
				\drawedge{(c)}{(e)}
				\drawedge{(c)}{(g)}
				\drawedge{(d)}{(f)}
				\drawedge{(d)}{(i)}
				\drawedge{(e)}{(g)}
				\drawedge{(e)}{(f)}
				\drawedge{(e)}{(h)}
				\drawedge{(f)}{(h)}
				\drawedge{(f)}{(i)}
			\end{scope}
	
			\begin{scope}[yshift=4cm, xshift=.75cm]
				\node [] (name) at (3, \nameindent) {$W_{1, t}$, $t \ge 0$};
				\node [vtx] (a) at (0.75, 0) {};
				\node [vtx, above left of = a] (b) {};
				\node [vtx, above right of = b] (c) {};
				\node [vtx, above right of = a] (d) {};
				\node [vtx, above right of = d, node distance=.75cm] (t12) {};
				\node [vtx, below right of = t12, node distance=.75cm] (t13) {};
				\node [vtx, right of = t13] (t31) {};
				\node [vtx, above right of = t31, node distance=.75cm] (t32) {};
				\node [vtx, below right of = t32, node distance=.75cm] (t33) {};
				\node [vtx, above right of = t33] (y) {};
				\node [vtx, below right of = y] (z) {};
				\node [vtx, below left of = z] (w) {};
		
				\drawedge{(a)}{(b)}
				\drawedge{(a)}{(d)}
				\drawedge{(b)}{(c)}
				\drawedge{(b)}{(d)}
				\drawedge{(c)}{(d)}
				\drawedge{(d)}{(t12)}
		
				\drawedge{(d)}{(t13)}

				\drawedge{(t12)}{(t13)}
				
				\drawdottededge{(t13)}{(t31)}
				\drawedge{(t31)}{(t32)}
				\drawedge{(t31)}{(t33)}
				\drawedge{(t32)}{(t33)}
		
				\drawedge{(w)}{(t33)}
				\drawedge{(w)}{(z)}
				\drawedge{(t33)}{(y)}
				\drawedge{(t33)}{(z)}
				\drawedge{(y)}{(z)}
		
				\draw [decorate,decoration={brace,amplitude=10pt,raise=5pt}]
						(t33) -- (d) node [black,midway,yshift=-.5cm] {{\small $t$ triangles}};
		
			\end{scope}
	
			\begin{scope}[yshift=4cm, xshift=9cm]
				\node [] (name) at (3, \nameindent) {$W_{2, t}$, $t \ge 0$};
				\node [vtx] (a) at (0.75, 0) {};
				\node [vtx, above left of = a] (b) {};
				\node [vtx, above right of = b] (c) {};
				\node [vtx, above right of = a] (d) {};
				\node [vtx, above left of = b] (e) {};
				\node [vtx, above right of = e] (f) {};
				\node [vtx, above right of = d, node distance=.75cm] (t12) {};
				\node [vtx, below right of = t12, node distance=.75cm] (t13) {};
				\node [vtx, right of = t13] (t31) {};
				\node [vtx, above right of = t31, node distance=.75cm] (t32) {};
				\node [vtx, below right of = t32, node distance=.75cm] (t33) {};
				\node [vtx, above right of = t33] (y) {};
				\node [vtx, below right of = y] (z) {};
				\node [vtx, below left of = z] (w) {};
		
				\drawedge{(a)}{(b)}
				\drawedge{(a)}{(c)}
				\drawedge{(a)}{(d)}
				\drawedge{(b)}{(c)}
				\drawedge{(b)}{(d)}
				\drawedge{(c)}{(d)}
		
				\drawedge{(b)}{(e)}
				\drawedge{(c)}{(e)}
				\drawedge{(c)}{(f)}
				\drawedge{(e)}{(f)}		
		
		
				\drawedge{(d)}{(t12)}
		
				\drawedge{(d)}{(t13)}

				\drawedge{(t12)}{(t13)}
				
				\drawdottededge{(t13)}{(t31)}
				\drawedge{(t31)}{(t32)}
				\drawedge{(t31)}{(t33)}
				\drawedge{(t32)}{(t33)}
		
				\drawedge{(w)}{(t33)}
				\drawedge{(w)}{(z)}
				\drawedge{(t33)}{(y)}
				\drawedge{(t33)}{(z)}
				\drawedge{(y)}{(z)}
		
				\draw [decorate,decoration={brace,amplitude=10pt,raise=5pt}]
						(t33) -- (d) node [black,midway,yshift=-.5cm] {{\small $t$ triangles}};
		
			\end{scope}
	
			\begin{scope}[yshift=7cm, xshift=.75cm]
				\node [] (name) at (3, \nameindent) {$W_{3, t}$, $t \ge 0$};
				\node [vtx] (a) at (0.75, 0) {};
				\node [vtx, above left of = a] (b) {};
				\node [vtx, above right of = b] (c) {};
				\node [vtx, above right of = a] (d) {};
				\node [vtx, above left of = b] (e) {};
				\node [vtx, above right of = e] (f) {};
				\node [vtx, above right of = d, node distance=.75cm] (t12) {};
				\node [vtx, below right of = t12, node distance=.75cm] (t13) {};
				\node [vtx, right of = t13] (t31) {};
				\node [vtx, above right of = t31, node distance=.75cm] (t32) {};
				\node [vtx, below right of = t32, node distance=.75cm] (t33) {};
				\node [vtx, above right of = t33] (y) {};
				\node [vtx, below right of = y] (z) {};
				\node [vtx, below left of = z] (w) {};
				\node [vtx, above right of = z] (u) {};
				\node [vtx, above right of = y] (v) {};
		
				\drawedge{(a)}{(b)}
				\drawedge{(a)}{(c)}
				\drawedge{(a)}{(d)}
				\drawedge{(b)}{(c)}
				\drawedge{(b)}{(d)}
				\drawedge{(c)}{(d)}
		
				\drawedge{(b)}{(e)}
				\drawedge{(c)}{(e)}
				\drawedge{(c)}{(f)}
				\drawedge{(e)}{(f)}	
		
				\drawedge{(w)}{(y)}
				\drawedge{(y)}{(u)}
				\drawedge{(y)}{(v)}
				\drawedge{(z)}{(u)}
				\drawedge{(u)}{(v)}
		
		
				\drawedge{(d)}{(t12)}
		
				\drawedge{(d)}{(t13)}

				\drawedge{(t12)}{(t13)}
				
				\drawdottededge{(t13)}{(t31)}
				\drawedge{(t31)}{(t32)}
				\drawedge{(t31)}{(t33)}
				\drawedge{(t32)}{(t33)}
		
				\drawedge{(w)}{(t33)}
				\drawedge{(w)}{(z)}
				\drawedge{(t33)}{(y)}
				\drawedge{(t33)}{(z)}
				\drawedge{(y)}{(z)}
		
				\draw [decorate,decoration={brace,amplitude=10pt,raise=5pt}]
						(t33) -- (d) node [black,midway,yshift=-.5cm] {{\small $t$ triangles}};
		
			\end{scope}
	
			\begin{scope}[yshift=7cm, xshift=11.25cm]
				\node [] (name) at (1, \nameindentsmall) {\Holes};
				\node [vtx] (b) at (.75, 0) {};
				\node [vtx, above left of = b] (c) {};
				\node [vtx, above right of = c] (d) {};
				\node [vtx, above right of = b] (a) {};
		
				\drawedge{(b)}{(c)}
				\drawedge{(a)}{(b)}
				\drawedge{(c)}{(d)}
				\path[draw, thick, dotted]
				(a) edge node {} (d);
		
			\end{scope}
		\end{tikzpicture}
	}

	\caption{The set of obstructions $\obswp$ for well-partitioned chordal graphs.}
	\label{fig:obstructions}
\end{figure}
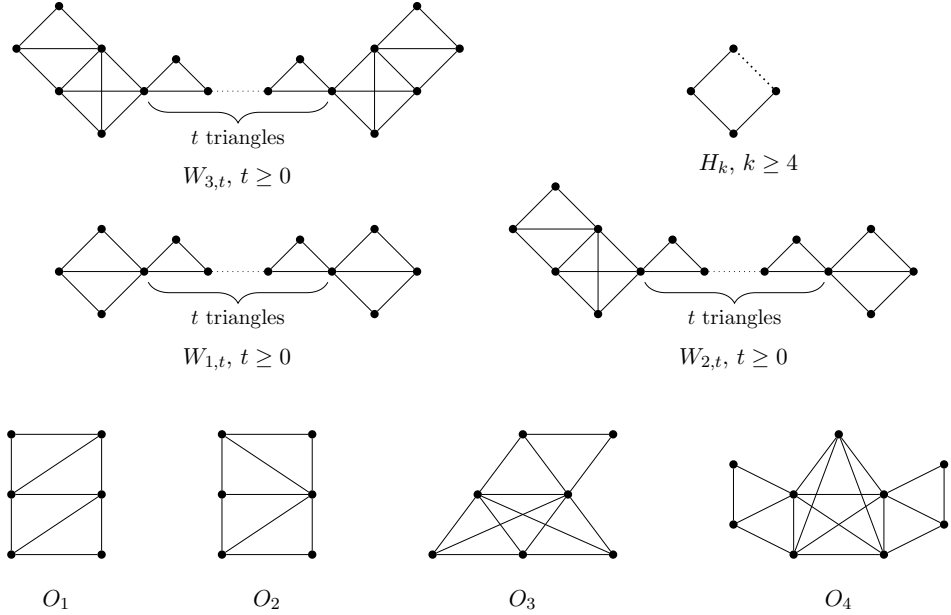

The main structural contribution of this work is a characterization of well-partitioned chordal graphs
by forbidden induced subgraphs. We also provide a polynomial-time recognition algorithm. 
We list the set $\obswp$ of obstructions in Figure~\ref{fig:obstructions}. 

\begin{theorem}\label{thm:obstructions}
	A graph is a well-partitioned chordal graph if and only if it has no induced subgraph isomorphic to a graph in~$\obswp$.
	Furthermore, there is a polynomial-time algorithm that given a graph $G$, output
	either an induced subgraph of $G$ isomorphic to a graph in $\obswp$, or 
	a partition tree for each connected component which confirms that $G$ is a well-partitioned chordal graph.
\end{theorem}

Before we proceed with the discussion of the algorithmic results of this paper, 
we would like to briefly touch on the relationship of well-partitioned chordal graphs and width parameters.
Each split graph is a well-partitioned chordal graph, 
and there are split graphs of whose \emph{maximum induced matching width} (mim-width) 
depends linearly on the number of vertices~\cite{MengelMIM}.
This rules out the applicability of any algorithmic meta-theorem based on one of the common width parameters
such as tree-width or clique-width,
to the class of well-partitioned chordal graphs.
It is known that mim-width is a lower bound for them~\cite{VatshelleThesis}.

Besides narrowing the complexity gap between the classes of chordal and split graphs,
the class of well-partitioned chordal graphs can also be useful as a step towards determining the 
yet unresolved complexity of a problem $\Pi$ on chordal graphs when it 
is known that $\Pi$ is easy on split graphs. This is the case in our current work.
Specifically, we study the the \probDP{} problem, formally defined as follows,
and generalizations thereof. 
Two paths $P_1$ and $P_2$ are called \emph{internally vertex-disjoint},
if for $i \in [2]$, no internal vertex of $P_i$ is contained in $P_{3-i}$.
(Note that this excludes the possibility that an endpoint of one path 
is used as an internal vertex in the other path.)
\fancyproblemdef
	{Disjoint Paths}
	{A graph $G$, a set $\mathcal{X}=\{(s_1,t_1),\ldots,(s_k,t_k)\}$ of 
	$k$ pairs of vertices of $G$, called \emph{terminals}.}
	{Does $G$ contain $k$ pairwise internally vertex-disjoint paths $P_1,\ldots,P_k$ 
	such that for all $i\in[k]$, $P_i$ is an $(s_i,t_i)$-path?}

This problem has already been shown by Karp to be \NP-complete~\cite{DPNPC}, and
as a cornerstone result in the early days of fixed-parameter tractability theory,
Robertson and Seymour showed that \textsc{Disjoint Paths} 
parameterized by $k$ is in \FPT~\cite{graphminors13}.
The dependence of the runtime on the number of vertices $n$ in the input graph
is cubic in Robertson and Seymour's algorithm, and
Kawarabayashi et al.\ improved this to a quadratic dependence on~$n$~\cite{Kawarabayashi2011}.
From the viewpoint of kernelization complexity, Bodlaender et al.\ showed that \probDP{} 
does not admit a polynomial kernel unless $\NP\subseteq\coNP/\poly$~\cite{BTY2011}.

Restricting the problem to chordal and split graphs,
Heggernes et al.\ showed that \probDP{} remains \NP-complete on split graphs,
and that it admits a polynomial kernel parameterized by $k$~\cite{Heg15},
and Kammer and Tholey showed that it has an \FPT-algorithm with linear dependence
on the size of the input chordal graph~\cite{DPFPTCHORDAL}.
The question whether \probDP{} has a polynomial kernel on chordal graphs remains open.
We go one step towards such a polynomial kernel, by showing that \probDP{} 
has a polynomial kernel on well-partitioned chordal graphs; generalizing the polynomial 
kernel on split graphs~\cite{Heg15}.

We also study a generalization of the \probDP{} problem, where in a solution, each path
$P_i$ can only use a restricted set of vertices $U_i$, which is specified for each terminal pair at the input.
This problem was recently introduced by Belmonte et al.\ and given the name \SRDP~\cite{patterns2014}.
Since this problem contains \probDP{} as a special case 
(setting all domains equal to the whole vertex set), it is \NP-complete.
Belmonte et al.\ showed that \SRDP{} parameterized by $k$ is in \XP{} on chordal graphs, 
and leave as an open question whether it is in \FPT{} or \W[1]-hard on chordal graphs.
Towards showing the former, we give an \FPT{}-algorithm on well-partitioned chordal graphs.
In particular, given a partition tree,
our algorithm runs in time $2^{\calO(k \log k)}\cdot n$,
so the runtime only depends linearly on the number of vertices in the input graph.
While we do not settle the kernelization complexity of \SRDP{} on well-partitioned chordal graphs,
we observe that our \FPT{}-algorithm implies a polynomial kernel on split graphs.
We summarize these results in Table~\ref{tab:results}.

\newcommand\ourmark{$^\star$}
\begin{table}
	\centering
	\begin{tabular}{r|c|c}
		Graph Class & \probDP{} & \SRDP \\
		\hline
		Chordal & linear \FPT~\cite{DPFPTCHORDAL} & \XP~\cite{patterns2014} \\
		\hline
		Well-partitioned chordal & $\calO(k^3)$ kernel [T.\,\ref{thm:dp:kernel},\,\ref{thm:tdp:kernel}] 
			& linear$^\star$ \FPT{} [T.\,\ref{thm:srdp},\,\ref{thm:srtdp}] \\
		\hline
		Split & $\calO(k^2)$ kernel~\cite{Heg15} & $\calO(k^2)$ kernel [C.\,\ref{cor:srdp:split}]
	\end{tabular}
	\caption{Summary of some results about the parameterized complexity results of the \probDP{} and \SRDP{} problems
	parameterized by the number $k$ of terminal pairs
	on split, well-partitioned chordal, and chordal graphs. 
	Size bounds for kernels are in terms of the number of \emph{vertices} of the kernelized instances.
	$^\star$Assuming that we are given a partition tree.}
	\label{tab:results}
\end{table}

Finally, we also consider the \SRDCS{} problem where we are given $k$ terminal \emph{sets} instead of pairs, and $k$ domains,
and the question is whether there are $k$ pairwise disjoint connected subgraphs, each one connecting one of the terminal sets,
using only vertices from the specified domain.
This problem was also introduced in~\cite{patterns2014} and shown to be in \XP{} on chordal graphs,
when the parameter is the total number of vertices in all terminal sets.
We observe that the ideas we use in our algorithm for the path problems can be used to show that
\SRDCS{} on well-partitioned chordal graphs is in \FPT{} with this parameter.

\plainparagraph{The case of coinciding terminal pairs.} 
We want to point out that the \probDP{} is studied in several variants in the literature.
According to the definition that we give here, pairs of terminal vertices may coincide,
i.e.\ it may happen that for $i \neq j$, $\{s_i, t_i\} = \{s_j, t_j\} \eqdef \{x, y\}$. 
If the edge $xy$ is present in the input graph, then it may be used both as the path $P_i$ and as the path $P_j$
in the solution, without violating the definition. 
However, as pointed out in e.g.~\cite{Heg15}, it is natural to impose the additional condition that
all paths in a solution have to be pairwise distinct. 
We study both variants of the \textsc{(Set-Restricted) Disjoint Paths} problem,
and refer to the variant that requires pairwise distinct paths in a solution as 
\textsc{(Set-Restricted) Totally Disjoint Paths}.

\section{Preliminaries}\label{sec2}
For a positive integer $n$, we let $[n]\defeq \{1, 2, \ldots, n\}$.
For a set $X$ and an integer $k$, we denote by $\binom{X}{k}$ the size-$k$ subsets of $X$.

A \emph{graph} $G$ is a pair of a \emph{vertex set} $V(G)$ and an \emph{edge set} $E(G) \subseteq \binom{V(G)}{2}$.
All graphs considered in this paper are finite, i.e.\ their vertex sets are finite.
For an edge $\{u, v\} \in E(G)$, we call $u$ and $v$ its \emph{endpoints} and we use the shorthand `$uv$' for `$\{u, v\}$'.
Let $G$ and $H$ be two graphs. 
We say that \emph{$G$ is isomorphic to $H$} if there is a bijection $\phi \colon V(G) \to V(H)$ 
such that for all $u, v \in V(G)$, $uv \in E(G)$ if and only if $\phi(u)\phi(v) \in E(H)$.
We say that $H$ is a \emph{subgraph} of $G$, denoted by $H \subseteq G$, if $V(H) \subseteq V(G)$ and $E(H) \subseteq E(G)$.

For a vertex $v$ of a graph $G$, $N_G(v) \defeq \{w \in V(G) \mid vw \in E(G)\}$ is the set of \emph{neighbors} of $v$ in $G$,
and we let $N_G[v] \defeq N_G(v) \cup \{v\}$.
The \emph{degree} of $v$ is $\deg_G(v) \defeq \card{N_G(v)}$.
Given a set $X \subseteq V(G)$, we let $N_G(X) \defeq \bigcup_{v \in X} N_G(v) \setminus X$ and $N_G[X] \defeq N_G(X) \cup X$.
In all of the above, we may drop $G$ as a subscript if it is clear from the context.
The \emph{subgraph induced by $X$}, denoted by $G[X]$, is the graph $(X, E(G) \cap \binom{X}{2})$.
We denote by $G - X$ the graph $G[V(G)\setminus X]$, and for a single vertex $x \in V(G)$, we use the shorthand `$G - x$' for `$G - \{x\}$'. 
For two sets $X,Y\subseteq V(G)$, we denote by $G[X,Y]$ the graph $(X \cup Y, \{xy \in E(G) \mid x \in X, y \in Y\})$.
We say that $X$ is \emph{complete to $Y$} if $X\cap Y=\emptyset$ and each vertex in $X$ is adjacent to every vertex in $Y$.

Let $G$ be a graph.
We say that $G$ is \emph{trivial} if $\abs{V(G)} = 1$.
$G$ is called \emph{complete} if $E(G) = \binom{V(G)}{2}$, and \emph{empty} if $E(G) = \emptyset$.
A set $X \subseteq V(G)$ is a \emph{clique} if $G[X]$ is complete,
and an \emph{independent set} if $G[X]$ is empty.
A clique of size $3$ is called a \emph{triangle}.
A graph $G$ is called \emph{bipartite} there is a $2$-partition $(A, B)$ of $V(G)$, called the \emph{bipartition of $G$}, such that $A$ and $B$ are independent sets in $G$.
A bipartite graph $G$ on bipartition $(A, B)$ is called \emph{complete bipartite} if $A$ is complete to $B$.
For integers $n$ and $m$, we denote by $K_{n, m}$ a complete bipartite graph with bipartition $(A, B)$ such that $\card{A} = n$ and $\card{B} = m$.
A graph is a \emph{star} if it is either trivial or isomorphic to $K_{1,n}$ for some positive integer $n$.

A graph $G$ is \emph{connected} if for each $2$-partition $(X, Y)$ of $V(G)$ with $X \neq \emptyset$ and $Y \neq \emptyset$, 
there is a pair $x \in X$, $y \in Y$ such that $xy \in E(G)$.
A connected component of $G$ is a maximal connected subgraph of $G$. A vertex $v \in V(G)$ is a \emph{cut vertex} if $G - v$ has more connected components than $G$. A graph is \emph{$2$-connected} if it has no cut vertices. A \emph{block} of a graph $G$ is a maximal $2$-connected component of $G$.
A graph $G$ is called \emph{$2$-regular} if all vertices of $G$ are of degree $2$.
A connected $2$-regular graph is a cycle. A graph that has no cycle as a subgraph is called a \emph{forest}, 
a connected forest is a \emph{tree}, and a tree of maximum degree $2$ is a \emph{path}.
The vertices of degree one in a tree are called \emph{leaves} and the leaves of a path are its \emph{endpoints}.
A connected subgraph of a tree is called a \emph{subtree}.

A \emph{hole} in a graph $G$ is an induced cycle of $G$ of length at least $4$.
A graph is \emph{chordal} if it has no induced subgraph isomorphic to a hole.
A vertex is \emph{simplicial} if $N_G(v)$ is a clique.
We say that a graph $G$ has a \emph{perfect elimination ordering} 
$v_1,\ldots,v_n$ if $v_i$ is simplicial in $G[\{v_i, v_{i+1},\ldots,v_n\}]$ for each $i\in[n-1]$.
It is known that a graph is chordal if and only if it has a perfect elimination ordering~\cite{FulkersonGross1965}.
We will use the following hole detecting algorithm and an algorithm to generate a perfect elimination ordering of a chordal graph.
\begin{theorem}[Nikolopoulos and Palios~\cite{Nikolo2007}]\label{thm:nikolo}
Given a graph $G$, one can detect a hole in $G$ in time $\mathcal{O}(|V(G)|+|E(G)|^2)$, if one exists.
\end{theorem}
\begin{theorem}[Rose, Tarjan, and Lueker~\cite{Rose1976}]\label{thm:generate}
Given a graph $G$, one can generate a perfect elimination ordering of $G$ in time $\mathcal{O}(|V(G)|+|E(G)|)$, if one exists.
\end{theorem}

A graph $G$ is a \emph{split graph} if there is a $2$-partition $(C, I)$ of $V(G)$ such that $C$ is a clique and $I$ is an independent set.
Let $\calS$ be a family of subsets of some set.
The \emph{intersection graph of $\calS$} is the graph on vertex set $\calS$ and edge set $\{\{S, T\} \in \binom{\calS}{2} \mid S \cap T \neq \emptyset\}$.
It is well-known that each chordal graph is the intersection graph of vertex sets of subtrees of some tree.
The following graph is called a \emph{diamond}. Note that for all $s \in [3]$, $t \in \bN$, the graph $W_{s, t}$ in $\obswp$ (see Figure~\ref{fig:obstructions}) contains two diamonds as induced subgraphs.
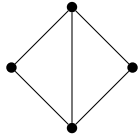
\begin{center}
	\begin{tikzpicture}
		\node[vtx] (a) at (0, 0) {};
		\node[vtx, above left=of a] (b) {};
		\node[vtx, above right=of a] (d) {};
		\node[vtx, above right=of b] (c) {};
		
		\draw[edge] (a) to (b);
		\draw[edge] (b) to (c);
		\draw[edge] (c) to (d);
		\draw[edge] (d) to (a);
		\draw[edge] (a) to (c);
	\end{tikzpicture}
\end{center}

\section{Well-partitioned chordal graphs}\label{sec3}
A connected graph $G$ is a \emph{well-partitioned chordal graph} if there exist a partition $\calP$ of $V(G)$ and a tree $\calT$ having $\calP$ as a vertex set such that the following hold.
\begin{enumerate}[label={(\roman*)}]
	\item\label{def:wp:cliques} Each part $X \in \calP$ is a clique in $G$.
	\item\label{def:wp:tree:parent} For each edge $XY \in E(\calT)$, there are subsets $X' \subseteq X$ and $Y' \subseteq Y$ such that $E(G[X, Y]) = X' \times Y'$.
	\item\label{def:wp:tree:nbh} For each pair of distinct $X, Y \in V(\calT)$ with $XY\notin E(\calT)$, $E(G[X, Y])=\emptyset$.
\end{enumerate}
The tree $\calT$ is called a \emph{partition tree of $G$}, and the elements of $\calP$ are called its \emph{bags}. 
A graph is a well-partitioned chordal graph if all of its connected components are well-partitioned chordal graphs.
We remark that a well-partitioned chordal graph can have more than one partition tree.
Also, observe that well-partitioned chordal graphs are closed under taking induced subgraphs.

A useful concept when considering partition trees of well-partitioned chordal graphs is that of a \emph{boundary of a bag}. Let $\mathcal{T}$ be a partition tree of a well-partitioned chordal graph $G$ and let $X, Y \in V(\calT)$ be two bags that are adjacent in~$\mathcal{T}$. The \emph{boundary of $X$ with respect to $Y$}, denoted by $\bdd{X}{Y}$, is the set of vertices of~$X$ that have a neighbor in $Y$, i.e.
	\begin{align*}
		\bdd{X}{Y} \defeq \{x \in X \mid N_G(x) \cap Y \neq \emptyset\}.
	\end{align*}
By item~\ref{def:wp:tree:parent} of the definition of the class, we know that $\bdd{X}{Y}$ is complete to $\bdd{Y}{X}$.

We now consider the relation between well-partitioned chordal graphs and other well-studied classes of graphs. It is easy to see that every well-partitioned chordal graph $G$ is a chordal graph because every leaf of the partition tree of a component of $G$ contains a simplicial vertex of $G$, and after removing this vertex, the remaining graph is still a well-partitioned chordal graph. Thus, we may construct a perfect elemination ordering.
We show that, in fact, well-partitioned chordal graphs constitute a subclass of substar graphs. 
A graph is a \emph{substar graph}~\cite{Chang1993, Joos2014} if it is an intersection graph of substars of a tree.  

\begin{proposition}\label{prop:wpissubstar}
Every well-partitioned chordal graph is a substar graph.
\end{proposition}

\begin{proof}
Let $G$ be a well-partitioned chordal graph with $V(G)=\{v_1,v_2,\ldots,v_n\}$ and a partition tree $\mathcal{T}$. We will exhibit a substar intersection model for $G$. That is, we will show that there exists a tree $\mathcal{T'}$ and $S_1,\ldots,S_n$ substars of $\mathcal{T'}$ such that $v_iv_j\in E(G)$ if and only if $V(S_i)\cap V(S_j)\neq\emptyset$.

Let $\mathcal{T'}$ be the tree obtained from $\mathcal{T}$ by the $1$-subdivision of every edge. 
We denote by $v_{XY} \in V(\calT')$ the vertex originated from the $1$-subdivision of the edge $XY \in E(\mathcal{T})$. 
Note that $N_{\mathcal{T'}}(v_{XY})=\{X,Y\}$. 
For every $v_i\in V(G)$, we create a substar of $\mathcal{T'}$ in the following way. Let $B\in V(\calT)$ be the bag containing~$v_i$. Then $S_i$ is a star with the center $B$ and the leaf set $\{v_{BY}~|~v_i\in\bdd{B}{Y}\}$.

To see that this is indeed an intersection model for $G$, let $v_iv_j\in E(G)$. If there exists $B\in V(\calT)$ such that $v_i,v_j\in B$, then $B\in V(S_i)\cap V(S_j)$. If $v_i$ and $v_j$ are not contained in the same bag, by item~\ref{def:wp:tree:parent}, there exist $A,B\in V(\calT)$ such that $v_i\in A$, $v_j\in B$ and $AB\in E_\calT$. Then, $v_{AB}\in V(S_i)\cap V(S_j)$. In both cases we have that $V(S_i)\cap V(S_j)\neq\emptyset$. Now suppose $V(S_i)\cap V(S_j)\neq\emptyset$. Note that, by construction, two stars that intersect either have the same center or they intersect in a vertex that is a leaf of both of them. If $S_i$ and $S_j$ have the same center~$B$, then $v_i,v_j\in B$ and hence, by item~\ref{def:wp:cliques}, $v_iv_j\in E(G)$. If $S_i$ and $S_j$ have a common leaf, then this leaf is a vertex originated by the $1$-subdivision of an edge. Then, there exist $A,B\in V(\calT)$ such that $v_i\in \bdd{A}{B}$ and $v_j\in\bdd{B}{A}$ and thus, by item~\ref{def:wp:tree:parent}, $v_iv_j\in E(G)$.
\end{proof}

From the definition of well-partitioned chordal graphs, one can also see that every split graph is a well-partitioned chordal graph. Indeed, if $G$ is a split graph with clique $K$ and independent set $S$, the partition tree of $G$ will be a star, with the clique $K$ as its central bag and each vertex of $S$ contained in a different leaf bag. We show that, in fact, every starlike graph is a well-partitioned chordal graph. A \emph{starlike graph}~\cite{STARLIKE} is an intersection graph of substars of a star.

\begin{proposition}\label{prop:starlikeiswp}
Every starlike graph is a well-partitioned chordal graph.
\end{proposition}

\begin{proof}
Let $G$ be a starlike graph with $V(G)=\{v_1,\ldots,v_n\}$ and let $\calS$ be the host star of the substar intersection model of $G$ and $S_i$ be the substar of $\calS$ associated with vertex $v_i$. We may assume that $G$ is connected and every vertex of $\calS$ is contained in some substar of the intersection model. 

We know that $v_iv_j\in E(G)$ if and only if $V(S_i)\cap V(S_j)\neq\emptyset$. To show that $G$ is a well-partitioned chordal graph, we will construct a partition tree for $G$. Let $c$ be the center of the star $\calS$ and $f_1,\ldots,f_k$ be its leaves. The partition tree $\calT$ for $G$ will be a star with center $C$ and leaves $F_1,\ldots,F_k$ such that $C=\{v_i\in V(G)~|~c\in S_i\}$ and $F_j=\{v_i\in V(G)~|~V(S_i)=\{f_j\}\}$. Note that this is indeed a partition of the vertex set of $G$, since each substar of $\calS$ either contains the center or consists of a single leaf and every vertex of $\calS$ is contained in some substar of the intersection model. Now we show this is indeed a partition tree for $G$. Note that, by construction, each bag is a clique, so item~\ref{def:wp:cliques} holds. Also note that, for every $i$, if $v\in F_i$, then $N_G(v)\subseteq F_i\cup C$, thus item~\ref{def:wp:tree:nbh} of the definition holds. Finally, note that the vertices of $F_i$ are true twins in $G$, since the substars of $\calS$ corresponding to those vertices consist of a single vertex, namely $f_i$. Hence, item~\ref{def:wp:tree:parent} also holds. This concludes the proof that $\calT$ is a partition tree for $G$ and thus $G$ is a well-partitioned chordal graph.
\end{proof}

We will show that the graph $O_1$ in Figure~\ref{fig:obstructions} is not a well-partitioned chordal graph. On the other hand, it is not difficult to see that $O_1$ is a substar graph.
Also note that a path graph on $5$ vertices is a well-partitioned chordal graph but not a starlike graph.
These observations with Propositions~\ref{prop:wpissubstar} and~\ref{prop:starlikeiswp} show that we have the following hierarchy of graph classes between split graphs and chordal graphs:

\vspace{-0.2in}

\begin{center}
\begin{tabular}{m{0.4in}cm{0.6in}cm{1.3in}cm{0.5in}cm{0.5in}}
\begin{center}split graphs\end{center} & {\Large$\subsetneq$} & \begin{center}starlike graphs\end{center} & {\Large$\subsetneq$} & \begin{center}well-partitioned chordal graphs\end{center} & {\Large$\subsetneq$} & \begin{center}substar graphs\end{center} & {\Large$\subsetneq$} & \begin{center}chordal graphs\end{center} 
\end{tabular}
\end{center}

\section{Characterization by forbidden induced subgraphs}\label{sec4}

This section is entirely devoted to the proof of Theorem~\ref{thm:obstructions}.
That is, we show that the set $\obswp$ of graphs depicted in Figure~\ref{fig:obstructions} is the set of all forbidden induced subgraphs for well-partitioned chordal graphs, and give a polynomial-time recognition algorithm for this graph class.
For convenience, we say that an induced subgraph of a graph that is isomorphic to a graph in $\obswp$ is an \emph{obstruction} for well-partitioned chordal graphs, or simply an obstruction.

In Subsection~\ref{subsec:obstruction}, we show that the graphs in $\obswp$ are not well-partitioned chordal graphs (Proposition~\ref{prop:characterization:wpc:to:obs}).
In Subsection~\ref{subsec:partial}, we introduce the notion of a boundary-crossing path which is the main tool for devising the polynomial-time recognition algorithm. 
We present the certifying algorithm in Subsection~\ref{subsec:certifying}, which also concludes the proof of the characterization by forbidden induced subgraphs for well-partitioned chordal graphs.

It is not difficult to observe that no graph in $\obswp$ contains another graph in $\obswp$ as an induced subgraph.
Thus, the results in this section also implies that graphs in $\obswp$ are minimal graphs with respect to the induced subgraph relation that are not well-partitioned chordal graphs.

\subsection{Graphs in $\obswp$ are not well-partitioned chordal graphs}\label{subsec:obstruction}
To argue that none of the graphs in $\obswp$ is a well-partitioned chordal graph, we make the following observation about triangles, which follows immediately from the definition of the partition tree.
\begin{observation}\label{obs:triangle}
	Let $G$ be a connected well-partitioned chordal graph, and $D \subseteq V(G)$ be the vertex set of a triangle in $G$. In any partition tree $\calT$ of $G$, there are at most two bags whose intersection with $D$ is non-empty.
\end{observation}

	Given a connected well-partitioned chordal graph $G$ and	
	a triangle with vertex set $D \subseteq V(G)$, we say that a partition tree of $G$
	\emph{respects $D$} if it contains a bag that contains all the vertices of $D$.
	For a non-empty proper subset $D' \subset D$, we say that a partition tree $\calT$ \emph{splits $D$ into $(D', D \setminus D')$} 
	if $\calT$ contains two distinct bags $B_1$ and $B_2$
	such that $B_1 \cap D = D'$ and $B_2 \cap D = D \setminus D'$.
	If a partition tree splits $D$ into $(D', D \setminus D')$ for some $D' \subset D$, 
	then we may simply say that it \emph{splits $D$}.
	By Observation~\ref{obs:triangle}, each partition tree either respects or splits each triangle.

	For $s\in \{1,2,3\}$ and $t\ge 0$, the vertex set of a block of $W_{s,t}$ having more than $3$ vertices is called a \emph{wing} of $W_{s,t}$.

\newcommand\obstruction{O}
\newcommand\ABar{\ensuremath{\obstruction_1}}
\newcommand\FourFan{\ensuremath{\obstruction_2}}
\newcommand\AhnOne{\ensuremath{\obstruction_3}}
\newcommand\AhnTwo{\ensuremath{\obstruction_4}}
\newcommand\Holes{\ensuremath{H_k, k \ge 4}}

\begin{proposition}\label{prop:characterization:wpc:to:obs}
	The graphs in $\obswp$ are not well-partitioned chordal graphs.
\end{proposition}
\begin{proof}
	For $k \ge 4$, $H_k$ is not a chordal graph, so it is not a well-partitioned chordal graph.
	
	We prove an auxiliary claim that will be useful to show that the graphs \ABar{}, \FourFan{}, \AhnOne{}, and \AhnTwo{} in $\obswp$ are not well-partitioned chordal graphs.
	\begin{nestedclaim}\label{claim:triangle:o1:o2}
		Let $H$ be a connected graph and $D=\{x, y, z\} \subseteq V(H)$ be a triangle in $H$.
		\begin{enumerate}[label={(\roman*)}]
			\item\label{claim:triangle:o1o2:type1} If there are adjacent vertices 
				$u, v \in V(H) \setminus D$ such that 
				$D \not\subseteq N_H(u)$ and $D \not\subseteq N_H(v)$, and
				$\emptyset \neq N_H(u) \cap D \neq N_H(v) \cap D \neq \emptyset$, 
				then $H$ has no partition tree respecting $D$.
			\item\label{claim:triangle:o1o2:type2} If there exists a vertex $u \in V(H) \setminus D$ such that 
				$N_H(u) \cap D = \{y, z\}$,
					then $H$ has no partition tree splitting $D$ into $(\{x, y\}, \{z\})$.	
			\item\label{claim:triangle:o1o2:type3} If there exist two non-adjacent vertices $u, v \in V(H) \setminus D$ 
					such that $N_H(u) \cap D = D = N_H(v) \cap D$, 
					then $H$ has no partition tree splitting $D$.
		\end{enumerate}
	\end{nestedclaim}
	\begin{claimproof} 
		In order to prove item~\ref{claim:triangle:o1o2:type1}, 
		suppose there is a partition tree $\calT$ of $H$ respecting $D$, and
		let $B$ be the bag containing $D$.
		First, since $D \not\subseteq N_H(u)$ and $D \not\subseteq N_H(v)$, we have that neither $v$ nor $u$ is contained in $B$ as $B$ is a clique in $H$. Furthermore, since $N_H(u) \cap D \neq \emptyset$ and $N_H(v) \cap D \neq \emptyset$, and since $uv \in E(G)$, it cannot be the case that $u$ and $v$ are in distinct bags, otherwise there would be a triangle in $\calT$. However, since $N_H(u) \cap D \neq N_H(v) \cap D$, $u$ and $v$ cannot be in the same bag either.
		
		Now we proceed to the proof of item~\ref{claim:triangle:o1o2:type2}.
		Suppose there is a partition tree $\calT$ of $H$ that splits $D$ into $(\{x, y\}, \{z\})$, and
		denote the two bags intersecting $D$ by $B_1$ and $B_2$ with $B_1 \cap D = \{x, y\}$ and $B_2 \cap D = \{z\}$.  
		Since $u$ is not adjacent to $x$, $u\notin B_1$.
		Since $x\in N_H(z)\cap B_1$ and $x\notin N_H(u)\cap B_1$,
	 $u$ cannot be contained in $B_2$ either.
		However, since $uz,uy \in E(G)$, if $u$ is in a bag other than $B_1$ and $B_2$, then $\{u,y,z\}$ is a triangle that intersects three distinct bags of $\calT$, a contradiction with Observation~\ref{obs:triangle}.
		
		To conclude, we prove item~\ref{claim:triangle:o1o2:type3}.
		Suppose there is a partition tree $\calT$ of $H$ that splits~$D$, and again
		denote the two bags intersecting $D$ by $B_1$ and $B_2$, 
		with $B_1 \cap D = \{x, y\}$ and $B_2 \cap D = \{z\}$.
		First, since $u$ and $v$ are non-adjacent, they cannot be in the same bag. 
		Furthermore, there cannot be a bag $B_3 \in V(\calT) \setminus \{B_1, B_2\}$ such that $\{u, v\} \cap B_3 \neq \emptyset$: 
		both $u$ and $v$ have neighbors in $B_1$ and in $B_2$, so this would imply the existence of a triangle that intersects three distinct bags of $\calT$ ($B_1$, $B_2$, and $B_3$). 
		The last case that remains is when $u \in B_1$ and $v \in B_2$. However, in this case, $B_1$ contains a vertex that is adjacent to $v$, namely $x$, and a vertex that is not adjacent to $v$, namely $u$, a contradiction. 
	\end{claimproof}

	Now, let us consider the obstructions $O_1$, $O_2$, $O_3$ and $O_4$ and assume that their vertices are labelled as in Figure~\ref{fig:labelling1}. 

	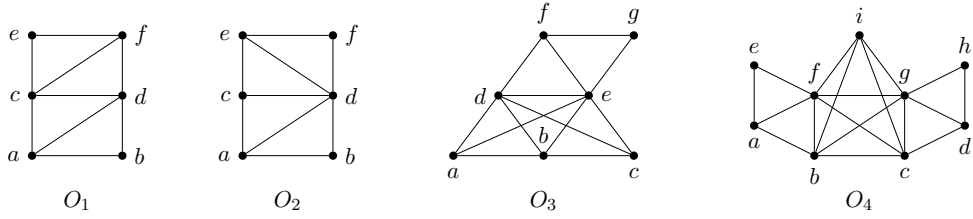
\begin{figure}
		\centering
		\scalebox{0.8}{
			\newcommand\drawedge[2]{\draw[edge] #1 to #2;}

			\def\nameindent{-.75}
			\def\nameindentsmall{-0.5}

			\begin{tikzpicture}
				\begin{scope}[]
					\node [] at (0.75, \nameindent) {\ABar};
					\node [vtx, label=left:{$a$}] (a) at (0, 0) {};
					\node [vtx, label=right:{$b$}] (b) at (1.5, 0) {};
					\node [vtx, label=left:{$c$}] (c) at (0, 1) {};
					\node [vtx, label=right:{$d$}] (d) at (1.5, 1) {};
					\node [vtx, label=left:{$e$}] (e) at (0, 2) {};
					\node [vtx, label=right:{$f$}] (f) at (1.5, 2) {};
	
					\drawedge{(a)}{(b)}
					\drawedge{(a)}{(c)}
					\drawedge{(a)}{(d)}
					\drawedge{(b)}{(d)}
					\drawedge{(c)}{(d)}
					\drawedge{(c)}{(e)}
					\drawedge{(c)}{(f)}
					\drawedge{(d)}{(f)}
					\drawedge{(e)}{(f)}
				\end{scope}
	
				\begin{scope}[xshift=3.5cm]
					\node[] at (0.75, \nameindent) {\FourFan};
					\node [vtx, label=left:$a$] (a) at (0, 0) {};
					\node [vtx, label=right:$b$] (b) at (1.5, 0) {};
					\node [vtx, label=left:$c$] (c) at (0, 1) {};
					\node [vtx, label=right:$d$] (d) at (1.5, 1) {};
					\node [vtx, label=left:$e$] (e) at (0, 2) {};
					\node [vtx, label=right:$f$] (f) at (1.5, 2) {};
		
					\drawedge{(a)}{(b)}
					\drawedge{(a)}{(c)}
					\drawedge{(a)}{(d)}
					\drawedge{(b)}{(d)}
					\drawedge{(c)}{(d)}
					\drawedge{(c)}{(e)}
					\drawedge{(d)}{(e)}
					\drawedge{(d)}{(f)}
					\drawedge{(e)}{(f)}
				\end{scope}
	
				\begin{scope}[xshift=7cm]
					\node [] (name) at (1.5, \nameindent) {\AhnOne};
					\node [vtx, label=below:$a$] (a) at (0, 0) {};
					\node [vtx, label=above:{$b$}] (b) at (1.5, 0) {};
					\node [vtx, label=below:$c$] (c) at (3, 0) {};
					\node [vtx, label=left:$d$] (d) at (0.75, 1) {};
					\node [vtx, label=right:$e$] (e) at (2.25, 1) {};
					\node [vtx, label=above:$f$] (f) at (1.5, 2) {};
					\node [vtx, label=above:$g$] (g) at (3, 2) {};
		
					\drawedge{(a)}{(b)}
					\drawedge{(a)}{(d)}
					\drawedge{(a)}{(e)}
					\drawedge{(b)}{(c)}
					\drawedge{(b)}{(d)}
					\drawedge{(b)}{(e)}
					\drawedge{(c)}{(d)}
					\drawedge{(c)}{(e)}
					\drawedge{(d)}{(e)}
					\drawedge{(d)}{(f)}
					\drawedge{(e)}{(f)}
					\drawedge{(e)}{(g)}
					\drawedge{(f)}{(g)}
				\end{scope}
	
				\begin{scope}[xshift=12cm]
					\node [] (name) at (1.75, \nameindent) {\AhnTwo};
					\node [vtx, label=below:$b$] (a) at (1, 0) {};
					\node [vtx, label=below:$c$] (b) at (2.5, 0) {};
					\node [vtx, label=below:$a$] (c) at (0, 0.5) {};
					\node [vtx, label=below:$d$] (d) at (3.5, 0.5) {};
					\node [vtx, label=above:$f$] (e) at (1, 1) {};
					\node [vtx, label=above:$g$] (f) at (2.5, 1) {};
					\node [vtx, label=above:$e$] (g) at (0, 1.5) {};
					\node [vtx, label=above:$i$] (h) at (1.75, 2) {};
					\node [vtx, label=above:$h$] (i) at (3.5, 1.5) {};
		
					\drawedge{(a)}{(b)}
					\drawedge{(a)}{(c)}
					\drawedge{(a)}{(e)}
					\drawedge{(a)}{(f)}
					\drawedge{(a)}{(h)}
					\drawedge{(b)}{(d)}
					\drawedge{(b)}{(e)}
					\drawedge{(b)}{(f)}
					\drawedge{(b)}{(h)}
					\drawedge{(c)}{(e)}
					\drawedge{(c)}{(g)}
					\drawedge{(d)}{(f)}
					\drawedge{(d)}{(i)}
					\drawedge{(e)}{(g)}
					\drawedge{(e)}{(f)}
					\drawedge{(e)}{(h)}
					\drawedge{(f)}{(h)}
					\drawedge{(f)}{(i)}
				\end{scope}
			\end{tikzpicture}
		}
		\caption{Labellings of graphs $O_1, O_2, O_3$, and $O_4$.}
		\label{fig:labelling1}
	\end{figure}

	By Observation~\ref{obs:triangle}, each partition tree either respects or splits a triangle.	
	First, consider the graph \ABar{} and consider the triangle $D = \{a, c, d\}$. Because of the vertices $e$ and $f$, 
	we can observe that, by Claim~\ref{claim:triangle:o1:o2}\ref{claim:triangle:o1o2:type1}, no partition tree of \ABar{} respects $D$.
	Furthermore, because of the vertices $b$, $f$, and $b$, we obtain by 	Claim~\ref{claim:triangle:o1:o2}\ref{claim:triangle:o1o2:type2} 
	that no partition tree splits $D$ into $(\{a, c\}, \{d\})$, $(\{a, d\}, \{c\})$, and $(\{c, d\}, \{a\})$, 
	respectively. Thus, no partition tree of \ABar{} splits $D$.
	Hence, \ABar{} does not admit a partition tree and therefore it is not a well-partitioned chordal graph.
	
	For \FourFan{}, consider again the triangle $\{a, c, d\}$. 
	The arguments are similar to the previous ones, except that the vertex $e$ should be used 
	to show that no partition tree splits $\{a, c, d\}$ into $(\{a, d\}, \{c\})$.
	
	For \AhnOne{}, consider the triangle $D = \{b, d, e\}$. Because of the vertices $f$ and $g$, we observe that, by Claim~\ref{claim:triangle:o1:o2}\ref{claim:triangle:o1o2:type1}, no partition tree of \AhnOne{} respects $D$. On the other hand, because of $a$ and $c$, we observe that by Claim~\ref{claim:triangle:o1:o2}\ref{claim:triangle:o1o2:type3}, no partition tree of \AhnOne{} splits $D$.
	Hence, \AhnOne{} is not a well-partitioned chordal graph. 
	
	For \AhnTwo{}, consider the triangle $D = \{b, c, g\}$. Because of the vertices $d$ and $h$, we can conclude by Claim~\ref{claim:triangle:o1:o2}\ref{claim:triangle:o1o2:type1} that no partition tree of \AhnTwo{} respects $D$. 
	Since $N_{\AhnTwo{}}(d) \cap D = \{c, g\} = D \setminus \{b\}$, 
	by Claim~\ref{claim:triangle:o1:o2}\ref{claim:triangle:o1o2:type2},
	no partition tree of \AhnTwo{} splits $D$ into $(\{b, c\}, \{g\})$ or $(\{b, g\}, \{c\})$. 	
	Thus, we may assume that each partition tree splits $D$ into $(\{c, g\}, \{b\})$. 
	Let $B_1$ and $B_2$ be the two bags such that $B_1\cap D=\{c, g\}$ and $B_2\cap D=\{b\}$.

	Since $hc\notin E(G)$, we have that $h\notin B_1$. Also, since $N_{\AhnTwo{}}(h)\cap \{g,c\}\neq N_{\AhnTwo{}}(d)\cap\{g,c\}$, 
	$h$ and $d$ cannot be in the same bag. 
	Thus, we conclude that $d\in B_1$, otherwise $\{d,h,g\}$ would be a triangle that intersects 
	three distinct bags of $\calT$.
	By considering the triangle $\{b,c,f\}$, we can conclude by symmetry that 
	$\{a,b,f\}$ are contained in the same bag, which is $B_2$.
	As $i$ is adjacent to neither $a$ nor $d$, 
	the bag containing $i$ forms a triangle with $B_1$ and $B_2$, a contradiction.
	We can conclude that \AhnTwo{} is not a well-partitioned chordal graph.
	
	Next, we show that for all $s \in \{1,2,3\}$ and $t \ge 0$, $W_{s, t}$ is not a well-partitioned chordal graph. This will be done by induction on $t$, and for the base case $t = 0$, we consider these graphs with their vertices labelled as in Figure~\ref{fig:labelling2}.
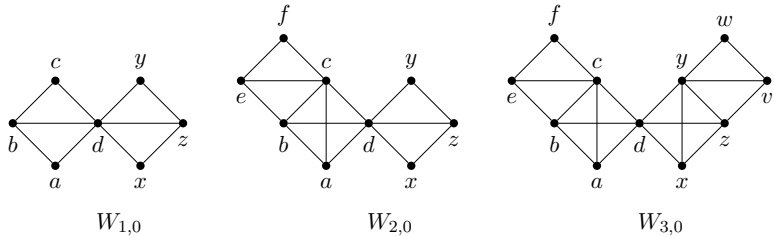
\begin{figure}
		\centering
		\scalebox{0.8}{
			\newcommand\drawedge[2]{\draw[edge] #1 to #2;}
			\newcommand\drawdottededge[2]{\draw[nonedge] #1 to #2;}

			\def\nameindent{-0.75cm}

			\begin{tikzpicture}
				\begin{scope}[xshift=.75cm]
					\node [vtx, label=below:{$a$}] (a) at (0.75, 0) {};
					\node [vtx, above left of = a, label=below:{$b$}] (b) {};
					\node [vtx, above right of = b, label=above:{$c$}] (c) {};
					\node [vtx, above right of = a, label=below:{$d$}] (d) {};
					\node [vtx, above right of = d, label=above:{$y$}] (y) {};
					\node [vtx, below right of = y, label=below:{$z$}] (z) {};
					\node [vtx, below left of = z, label=below:{$x$}] (w) {};
					\node [] at ($(a) !0.5! (z)$) [below=1cm] {$W_{1, 0}$};
		
					\drawedge{(a)}{(b)}
					\drawedge{(a)}{(d)}
					\drawedge{(b)}{(c)}
					\drawedge{(b)}{(d)}
					\drawedge{(c)}{(d)}
					\drawedge{(w)}{(d)}
					\drawedge{(w)}{(z)}
					\drawedge{(d)}{(y)}
					\drawedge{(d)}{(z)}
					\drawedge{(y)}{(z)}
		
				\end{scope}
	
				\begin{scope}[xshift=5.25cm]
					\node [vtx, label=below:{$a$}] (a) at (0.75, 0) {};
					\node [vtx, above left of = a, label=below:{$b$}] (b) {};
					\node [vtx, above right of = b, label=above:{$c$}] (c) {};
					\node [vtx, above right of = a, label=below:{$d$}] (d) {};
					\node [vtx, above left of = b, label=below:{$e$}] (e) {};
					\node [vtx, above right of = e, label=above:{$f$}] (f) {};
					\node [vtx, above right of = d, label=above:{$y$}] (y) {};
					\node [vtx, below right of = y, label=below:{$z$}] (z) {};
					\node [vtx, below left of = z, label=below:{$x$}] (w) {};
					\node [] at ($(a) !0.5! (z)$) [below=1cm] {$W_{2, 0}$};
		
					\drawedge{(a)}{(b)}
					\drawedge{(a)}{(c)}
					\drawedge{(a)}{(d)}
					\drawedge{(b)}{(c)}
					\drawedge{(b)}{(d)}
					\drawedge{(c)}{(d)}
		
					\drawedge{(b)}{(e)}
					\drawedge{(c)}{(e)}
					\drawedge{(c)}{(f)}
					\drawedge{(e)}{(f)}		
		
					\drawedge{(w)}{(d)}
					\drawedge{(w)}{(z)}
					\drawedge{(d)}{(y)}
					\drawedge{(d)}{(z)}
					\drawedge{(y)}{(z)}
		
				\end{scope}
	
				\begin{scope}[xshift=9.75cm]
					\node [vtx, label=below:{$a$}] (a) at (0.75, 0) {};
					\node [vtx, above left of = a, label=below:{$b$}] (b) {};
					\node [vtx, above right of = b, label=above:{$c$}] (c) {};
					\node [vtx, above right of = a, label=below:{$d$}] (d) {};
					\node [vtx, above left of = b, label=below:{$e$}] (e) {};
					\node [vtx, above right of = e, label=above:{$f$}] (f) {};
					\node [vtx, above right of = d, label=above:{$y$}] (y) {};
					\node [vtx, below right of = y, label=below:{$z$}] (z) {};
					\node [vtx, below left of = z, label=below:{$x$}] (w) {};
					\node [vtx, above right of = z, label=below:{$v$}] (u) {};
					\node [vtx, above right of = y, label=above:{$w$}] (v) {};
					\node [] at ($(a) !0.5! (z)$) [below=1cm] {$W_{3, 0}$};
		
					\drawedge{(a)}{(b)}
					\drawedge{(a)}{(c)}
					\drawedge{(a)}{(d)}
					\drawedge{(b)}{(c)}
					\drawedge{(b)}{(d)}
					\drawedge{(c)}{(d)}
		
					\drawedge{(b)}{(e)}
					\drawedge{(c)}{(e)}
					\drawedge{(c)}{(f)}
					\drawedge{(e)}{(f)}	
		
					\drawedge{(w)}{(y)}
					\drawedge{(y)}{(u)}
					\drawedge{(y)}{(v)}
					\drawedge{(z)}{(u)}
					\drawedge{(u)}{(v)}
		
					\drawedge{(w)}{(d)}
					\drawedge{(w)}{(z)}
					\drawedge{(d)}{(y)}
					\drawedge{(d)}{(z)}
					\drawedge{(y)}{(z)}
		
				\end{scope}
			\end{tikzpicture}
		}
		\caption{Labellings of graphs $W_{1,0}, W_{2,0}$, and $W_{3,0}$.}
		\label{fig:labelling2}
	\end{figure}

	For convenience, we labelled the cut vertex of each graph $d$. We recall that for $s \in \{1,2,3\}$, the vertex set of a block of $W_{s, 0}$ is a wing of $W_{s, 0}$. For instance, $\{a, b, c, d\}$ and $\{d, x, y, z\}$ are the wings of $W_{1, 0}$.
	\begin{nestedclaim}\label{claim:w0:wings}
		For $s \in \{1,2,3\}$, $W_{s, 0}$ has no partition tree having a bag whose intersection with a wing of $W_{s, 0}$ consists of only the cut vertex.
	\end{nestedclaim}
	\begin{claimproof}
		Suppose there is a partition tree $\calT_1$ of $W_{1, 0}$ that contains a bag $B$ such that $B \cap \{a, b, c, d\} = \{d\}$. This implies that there exists $B_1$ such that $\{a,b\}\subseteq B_1$, otherwise $\{a,b,d\}$ would be a triangle that intersects three distinct bags of $\calT_1$.
Since $c$ is not adjacent to $a$, $c \notin B_1$ and, by assumption, $c \notin B$. Thus $\{b,c,d\}$ is a triangle intersecting three bags of $\calT_1$, a contradiction. 
		
		Next, suppose there is a partition tree $\calT_2$ of $W_{2, 0}$ that contains a bag $B$ whose intersection with a wing of $W_{2, 0}$ consists of the cut vertex $d$ alone. If the affected wing is $\{d, x, y, z\}$, then the argument follows from the same argument given before.
		If $B \cap \{a, b, c, d, e, f\} = \{d\}$, we observe the following. First, there must exist a bag $B_1$ containing $\{a, b, c\}$, otherwise there is a triangle violating Observation~\ref{obs:triangle}. Since neither $e$ nor $f$ is adjacent to $a$, $\{e, f\} \cap B_1 = \emptyset$. Since $N_{W_{2, 0}}(e) \cap B_1 \neq N_{W_{2, 0}}(f) \cap B_1$, by Claim~\ref{claim:triangle:o1o2:type1}, there is no partition tree respecting $\{a,b,c\}$, a contradiction.
	
		The claim regarding $W_{3, 0}$ follows as well, noting that the wings of $W_{3, 0}$ are isomorphic to the one considered in the latter case.
	\end{claimproof}
	
	\begin{nestedclaim}
		For each $s \in \{1,2,3\}$ and $t \ge 0$, $W_{s, t}$ is not a well-partitioned chordal graph.
	\end{nestedclaim}
	\begin{claimproof}
		We prove the claim by induction on $t$. For $t = 0$, we observe that no bag of a partition tree can contain vertices from both wings of a $W_{s,0}$, unless it is the cut vertex. Hence, this case follows from Claim~\ref{claim:w0:wings}.

Now, suppose that for every $k\leq t-1$, $W_{s,k}$ is not a well-partitioned chordal graph. Consider $W_{s,t}$, with $t\geq 1$.
We may assume that besides the wings, $W_{s, t}$ has at least one triangle, call that triangle $D = \{p, q, s\}$. Suppose there is a partition tree $\calT_t$ for $W_{s, t}$. We will show how to transform this partition tree into a partition tree of $W_{s, t-1}$, contradicting the induction hypothesis. We know that $\calT_t$ either respects or splits $D$. 
		
		If $\calT_t$ respects $D$, let $B$ denote the bag that contains $D$. Suppose $q$ is the vertex in $D$ that has degree $2$ in $W_{s, t}$. Then, $B = D$, as no other vertex in $W_{s, t}$ is adjacent to $q$.
		Now, if we contract $D$ to a single vertex, say $p^*$, then we obtain $W_{s, t-1}$. 
		Moreover, if we replace $B$ by $B^* \defeq \{p^*\}$, and make $B^*$ adjacent to all bags in $N_{\calT_t}(B)$,
		then this gives a partition tree for $W_{s, t-1}$, a contradiction.

		If $\calT_t$ splits $D$ into $(\{p,s\}, \{q\})$, note that since no vertex other than $q$ is adjacent to both $p$ and $s$, we have that $B_1=\{p,s\}$ and $B_2=\{q\}$. As in the previous case, if we contract $D$ to a single vertex $p^*$, we can delete $B_2$ and replace $B_1$ by $B^* \defeq \{p^*\}$ to obtain a partition tree for $W_{s, t-1}$, a contradiction.
		
		Suppose that $\calT_t$ splits $D$ into $(\{p, q\}, \{s\})$, and let $B_1$ and $B_2$ be the bags of $\calT_t$ such that $B_1 \cap D = \{p, q\}$ and $B_2 \cap D = \{s\}$. Again, since no vertex other than $p$ and $s$ is adjacent to $q$, we have that $B_1 = \{p, q\}$. Now, if we simply remove $B_1$ from $\calT_t$ and make $B_2$ adjacent to all bags in $N_{\calT_t}(B_1) \setminus \{B_2\}$, then we obtain a partition tree for $W_{s, t-1}$, a contradiction. The case in which $\calT_t$ splits $D$ into $(\{q,s\}, \{p\})$ is symmetric to the latter case.
	\end{claimproof}
	This concludes the proof of Proposition~\ref{prop:characterization:wpc:to:obs}.
\end{proof}

\subsection{Boundary-crossing paths}\label{subsec:partial}

In the remaining part of this section, we present the certifying algorithm for well-partitioned chordal graphs. 
Here, we define the main concept of a boundary-crossing path and prove some useful lemmas. 

Let $G$ be a connected well-partitioned chordal graph with a partition tree $\mathcal{T}$. 
For a bag $X$ of $\mathcal{T}$ and $B\subseteq X$, a vertex $z \in V(G) \setminus X$ is said to \emph{cross $B$ in $X$}, if it has a neighbor both in $B$ and in $X \setminus B$. In this case, we also say that $B$ has a crossing vertex. In the following definitions, a path $X_1X_2 \ldots X_{\ell}$ in $\calT$ is considered to be ordered from $X_1$ to $X_{\ell}$.
Let $\ell\ge 3$ be an integer.
A path $X_1 X_2 \ldots X_\ell$ in $\calT$ is called a \emph{boundary-crossing path} if for each $1\leq i\leq \ell-2$, there is a vertex in $X_i$ that crosses $\bdd{X_{i+1}}{X_{i+2}}$.
If for each $1\leq i\leq \ell-2$, there is no bag $Y \in V(\calT) \setminus \{X_i\}$ containing a vertex that crosses $\bdd{X_{i+1}}{X_{i+2}}$, then we say the path is \emph{exclusive}.
If for each $1\leq i\leq \ell-2$, $\bdd{X_i}{X_{i+1}}$ is complete to $X_{i+1}$, then we say the path is \emph{complete}. If a boundary-crossing path is both complete and exclusive, then we call it \emph{good}. 
For convenience, we say that any path in $\calT$ with at most two bags is a boundary-crossing path.

The outline of the algorithm is as follows. First we may assume that a given graph $G$ is chordal, as we can detect a hole in polynomial time using Theorem~\ref{thm:nikolo} if it exists.
We may also assume that $G$ is connected.
So, it has a simplicial vertex $v$, and by an inductive argument, we can assume that $G-v$ is a well-partitioned chordal graph.
As $v$ is simplical, $G-v$ is also connected, and thus it admits a partition tree $\calT$.
If $v$ has neighbors in one bag of $\calT$, 
then we can simply put $v$ as a new bag adjacent to that bag.
Thus, we may assume that $v$ has neighbors in two distinct bags, say $C_1$ and $C_2$.
Then our algorithm is divided into three parts:
\begin{enumerate}
	\item\label{bcp:outline:1} We find a maximal good boundary-crossing path ending in $C_2C_1$ (or $C_1C_2$). To do this, when we currently have a good boundary-crossing path $C_iC_{i-1} \ldots C_2C_1$, find a bag $C_{i+1}$ containing a vertex crossing $\bdd{C_i}{C_{i-1}}$. If there is no such bag, then this path is maximal. Otherwise, we argue that in polynomial time either we can find an obstruction, or verify that $C_{i+1}C_i \ldots C_2C_1$ is good. 
	\item\label{bcp:outline:2} Assume that $C_kC_{k-1} \ldots C_2C_1$ is the obtained maximal good boundary-crossing path.
	Then we can in polynomial time modify $\calT$ so that no vertex crosses $\bdd{C_2}{C_1}$.
	\item\label{bcp:outline:3} We show that if no vertex crosses $\bdd{C_2}{C_1}$ and no vertex crosses $\bdd{C_1}{C_2}$, then we can extend $\calT$ to a partition tree of $G$.
\end{enumerate}

	For the lemmas of this section, we fix that $G$ is a connected chordal graph, $v$ is a simplicial vertex, and $G-v$ is a connected well-partitioned chordal graph with partition tree $\calT$, and furthermore assume that $v$ has neighbors in two distinct bags $C_1$ and $C_2$.

	Regarding Step~\ref{bcp:outline:2}, Lemma~\ref{lem:shorten:path} shows that when a maximal good boundary-crossing path $C_kC_{k-1} \ldots C_2C_1$ is given, we can modify $\calT$
	to a partition tree $\calT'$ such that no vertex crosses $\bdd{C_2'}{C_1'}$, 
	where $C_1'$ and $C_2'$ are the bags in $\calT'$ that correspond to $C_1$ and $C_2$ in $\calT$, respectively --
	in particular, they are the bags containing the neighbors of $v$.

	\begin{lemma}\label{lem:shorten:path}
		Let $C_k C_{k-1} \ldots C_1$ be a good boundary-crossing path for some integer $k \ge 3$ 
		such that no vertex crosses $\bdd{C_k}{C_{k-1}}$. 
		One can in polynomial time output a partition tree $\calT'$ of $G-v$ that contains
		a good boundary-crossing path $C_{k-1}' C_{k-2} \ldots C_1$
		such that no vertex in $G-v$ crosses $\bdd{C_{k-1}'}{C_{k-2}}$.
\end{lemma}
\begin{proof}
	Since no vertex crosses $\bdd{C_k}{C_{k-1}}$, we can partition the neighbors of $C_k$ in $\calT$ into $\calS_1$ and $\calS_2$ such that for all $S_1 \in \calS_1$, we have that $\bdd{C_k}{S_1} \subseteq C_k \setminus \bdd{C_k}{C_{k-1}}$, and for all $S_2 \in \calS_2$, $\bdd{C_k}{S_2} \subseteq \bdd{C_k}{C_{k-1}}$. 
		Let $C_{k}' \defeq C_k \setminus \bdd{C_k}{C_{k-1}}$ and $C_{k-1}' \defeq C_{k-1} \cup \bdd{C_k}{C_{k-1}}$.
		We obtain $\calT'$ from $\calT$ as follows.
		\begin{itemize}
			\item Remove $C_k$ and $C_{k-1}$, and add $C_k'$ and $C_{k-1}'$.
			\item Make all bags that have been adjacent to $C_{k-1}$ in $\calT$ adjacent to $C_{k-1}'$.
			\item Make all bags in $\calS_1$ adjacent to $C_k'$, and all bags in $\calS_2$ adjacent to $C_{k-1}'$.
		\end{itemize}
		
		Since $\bdd{C_k}{C_{k-1}}$ is complete to $C_{k-1}$, $C_{k-1}'$ is indeed a clique in $G-v$, and thus we conclude	
		that $\calT'$ is a partition tree of $G-v$. 
		Since $C_{k-1}'$ contains $C_{k-1}$ and there is no edge between $\bdd{C_k}{C_{k-1}}$ and $C_{k-2}$,
		we know that $C_{k-1}' C_{k-2} \ldots C_1$ is 
		a good boundary-crossing path.
		Clearly, $\calT'$ can be obtained in polynomial time.
		
		We claim that no vertex crosses $\bdd{C_{k-1}'}{C_{k-2}}$.
		Suppose for a contradiction that there exists a vertex $q \in V(G-v) \setminus C_{k-1}'$ that crosses $\bdd{C_{k-1}'}{C_{k-2}}$. 
		We consider two cases. First, we assume $q$ also crosses $\bdd{C_{k-1}}{C_{k-2}}$. 
		Since $C_{k} \ldots C_1$ is exclusive, any vertex crossing $\bdd{C_{k-1}}{C_{k-2}}$ is in $\bdd{C_k}{C_{k-1}}$.
		This means that $q \in C_{k-1}'$, a contradiction.
		Now assume that $q$ is a vertex in $V(G-v) \setminus (C_k \cup C_{k-1})$ that is adjacent to a vertex in 
		$\bdd{C_{k-1}'}{C_{k-2}} = \bdd{C_{k-1}}{C_{k-2}}$ and a vertex in $C_{k-1}' \setminus C_{k-1} = \bdd{C_k}{C_{k-1}}$. 
		But this would mean that there is a triangle in $\calT$, a contradiction. 
		
		We conclude that no vertex in $G-v$ crosses $\bdd{C_{k-1}'}{C_{k-2}}$.
\end{proof}
	
	With respect to Step~\ref{bcp:outline:3}, we prove the following lemma.
\begin{lemma}\label{lem:base:nocrossing}
		If every vertex of $G-v$ crosses neither $\bdd{C_1}{C_2}$ nor $\bdd{C_2}{C_1}$, then one can output a partition tree for $G$ in polynomial time.
\end{lemma}
	\begin{proof}
		Assume that every vertex of $G-v$ crosses neither $\bdd{C_1}{C_2}$ nor $\bdd{C_2}{C_1}$.
Let $\calS_1$ denote all neighbors of $C_1$ in $\calT$ such that for each $S_1 \in \calS_1$, $\bdd{C_1}{S_1} \subseteq C_1 \setminus \bdd{C_1}{C_2}$; let $\calS_2$ denote the set of all neighbors of $C_2$ in $\calT$ such that for each $S_2 \in \calS_2$, $\bdd{C_2}{S_2} \subseteq C_2 \setminus \bdd{C_2}{C_1}$; and let $\calS_{12}$ denote the set of all neighbors of $C_1$ or $C_2$ such that for each $S_1 \in \calS_{12} \cap N_\calT(C_1)$, $\bdd{C_1}{S_1} \subseteq \bdd{C_1}{C_2}$, and for each $S_2 \in \calS_{12} \cap N_\calT(C_2)$, $\bdd{C_2}{S_2} \subseteq \bdd{C_2}{C_1}$.
		Since no vertex of $G-v$ crosses neither $\bdd{C_1}{C_2}$ nor $\bdd{C_2}{C_1}$, $\calS_1 \cup \calS_2 \cup \calS_{12} = N_\calT(C_1) \cup N_\calT(C_2) \setminus \{C_1, C_2\}$.
		
		Now, let $C_1' \defeq C_1 \setminus \bdd{C_1}{C_2}$, $C_2' \defeq C_2 \setminus \bdd{C_2}{C_1}$, and $C_{12}' \defeq \bdd{C_1}{C_2} \cup \bdd{C_2}{C_1}$. We obtain $\calT'$ from $\calT$ as follows.
		\begin{itemize}
			\item Remove $C_1$ and $C_2$; add $C_1'$, $C_2'$, and $C_{12}'$; make $C_1'$ and $C_2'$ adjacent to $C_{12}'$.
			\item Make all bags in $\calS_1$ adjacent to $C_1'$, all bags in $\calS_2$ adjacent to $C_2'$, 
				and all bags in $\calS_{12}$ adjacent to $C_{12}'$.
			\item Add a new bag $C_v \defeq \{v\}$, and make it adjacent to $C_{12}'$.
		\end{itemize}
		This yields a partition tree for $G$.
	\end{proof}

	Considering Step~\ref{bcp:outline:1}, we present some lemmas useful to find an obstruction. 
	To describe subparts of the long obstructions $W_{s,t}$, we use the 
	graphs $W_{1, t}^-$ and $W_{2, t}^-$ as shown in Figure~\ref{fig:Ws}. Note that each of them has a distinguished vertex~$r$, that we call \emph{terminal}.
	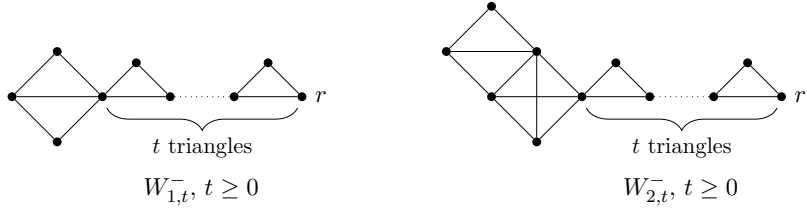
\begin{figure}
		\centering
		\scalebox{0.85}{
			\usetikzlibrary{decorations.pathreplacing,calc}
			\tikzset{
				bag/.style={
					shape=circle,
					draw,
					minimum width=2.5cm,
				}
			}

				\begin{tikzpicture}
					\begin{scope}[]
					\node [] (name) at (3, -0.75) {$W_{1, t}^-$, $t \ge 0$};
					\node [vtx] (a) at (0.75, 0) {};
					\node [vtx, above left of = a] (b) {};
					\node [vtx, above right of = b] (c) {};
					\node [vtx, above right of = a] (d) {};
					\node [vtx, above right of = d, node distance=.75cm] (t12) {};
					\node [vtx, below right of = t12, node distance=.75cm] (t13) {};
					\node [vtx, right of = t13] (t31) {};
					\node [vtx, above right of = t31, node distance=.75cm] (t32) {};
					\node [vtx, below right of = t32, node distance=.75cm, label=right:{$r$}] (t33) {};
				
					\draw[edge] (a) to (b);
					\draw[edge] (a) to (d);
					\draw[edge] (b) to (c);
					\draw[edge] (b) to (d);
					\draw[edge] (c) to (d);
		
		
					\draw[edge] (d) to (t12);
		
					\draw[edge]  (d) to (t13);

					\draw[edge] (t12) to (t13);
				
					\draw[nonedge] (t13) to (t31);
					\draw[edge] (t31) to (t32);
					\draw[edge] (t31) to (t33);
					\draw[edge] (t32) to (t33);
		
					\draw [decorate,decoration={brace,amplitude=10pt,raise=5pt}]
							(t33) -- (d) node [black,midway,yshift=-.5cm] {{\small $t$ triangles}};
		
				\end{scope}
				\begin{scope}[xshift=7.5cm]
					\node [] (name) at (3, -0.75) {$W_{2, t}^-$, $t \ge 0$};
					\node [vtx] (a) at (0.75, 0) {};
					\node [vtx, above left of = a] (b) {};
					\node [vtx, above right of = b] (c) {};
					\node [vtx, above right of = a] (d) {};
					\node [vtx, above left of = b] (e) {};
					\node [vtx, above right of = e] (f) {};
					\node [vtx, above right of = d, node distance=.75cm] (t12) {};
					\node [vtx, below right of = t12, node distance=.75cm] (t13) {};
					\node [vtx, right of = t13] (t31) {};
					\node [vtx, above right of = t31, node distance=.75cm] (t32) {};
					\node [vtx, below right of = t32, node distance=.75cm, label=right:{$r$}] (t33) {};
		
					\draw[edge] (a) to (b);
					\draw[edge] (a) to (c);
					\draw[edge] (a) to (d);
					\draw[edge] (b) to (c);
					\draw[edge] (b) to (d);
					\draw[edge] (c) to (d);
		
					\draw[edge] (b) to (e);
					\draw[edge] (c) to (e);
					\draw[edge] (c) to (f);
					\draw[edge] (e) to (f);	
		
		
					\draw[edge] (d) to (t12);
		
					\draw[edge]  (d) to (t13);

					\draw[edge] (t12) to (t13);
				
					\draw[nonedge] (t13) to (t31);
					\draw[edge] (t31) to (t32);
					\draw[edge] (t31) to (t33);
					\draw[edge] (t32) to (t33);
		
					\draw [decorate,decoration={brace,amplitude=10pt,raise=5pt}]
							(t33) -- (d) node [black,midway,yshift=-.5cm] {{\small $t$ triangles}};
		
				\end{scope}
				\end{tikzpicture}
		}
		\caption{The graphs $W_{1, t}^-$ and $W_{2, t}^-$.}
		\label{fig:Ws}
	\end{figure}
	The following lemma will be useful to find a wing at the beginning of a boundary-crossing path.
	\begin{figure}
		\centering
		\scalebox{0.8}{
			\begin{tikzpicture}
				\begin{scope}[]
					\node[bag] (C1) at (0,0) {};
					\node[anchor = south west] at (C1.30) {$Y$};
	
					\draw[bend left=45] (C1.300) to (C1.60);
		
					\node[smallbag] (C0) at (C1.270) [below = 1cm] {};
					\node[anchor = north west] at (C0.330) {$X$};
		
					\draw[bend right=45] (C0.165) to (C0.15);
		
					\node[bag] (C2) at (C1.0) [right = 1cm] {};
					\node[anchor = south west] at (C2.30) {$Z$};
		
					\draw[bend left =45] (C2.210) to (C2.330);
		
					\node[] (mid) at ($(C1)!0.5!(C2)$) {};
					\node[figlabel] at ($(mid)+(0,-4.5)$) {(a)};
		
					\node[vtx, label=below:{$x$}] (q) at (C0.90) [below=.15cm] {};		
		
					\node[vtx] (p) at (C1.0) [left=.5cm] {};
					\node[] at ($(p)+(0.1,0.25)$) {$y$};		
		
					\node[vtx] (pp) at (C1.180) [right=.5cm] {}; 
					\node[] at ($(pp)+(-0.225,0)$) {$y'$};
					\node[] at ($(pp)+(5,-1)$) {$B$};
		
					\node[vtx, label=right:{$z$}] (s) at (C2.270) [above=.45cm] {};
					\node[vtx, label=right:{$z'$}] (ss) at (C2.90) [below=.45cm] {};
		
					\draw[edge] (p) to (pp);
					\draw[edge] (p) to (q);
					\draw[edge] (pp) to (q);
		
					\draw[edge] (p) to (s);
					\draw[edge] (p) to (ss);
					\draw[edge] (s) to (ss);
		
					\node[vtx] (ppp) at (C1.90) [below=.3cm] {};
					\node[] at ($(ppp)+(-0.25,0)$) {$y''$};
		
					\draw[nonedge] (ppp) to (q);
					\draw[edge] (p) to (ppp);
					\draw[edge] (pp) to (ppp);
		
				\end{scope}
	
				\begin{scope}[xshift=8cm]
					\node[bag] (C1) at (0,0) {};
					\node[anchor = south west] at (C1.30) {$Y$};
	
					\draw[bend left=45] (C1.285) to (C1.75);
		
					\node[smallbag] (C0) at (C1.270) [below = 1cm] {};
					\node[anchor = north west] at (C0.330) {$X$};
		
					\draw[bend right=45] (C0.165) to (C0.15);
		
					\node[bag] (C2) at (C1.0) [right = 1cm] {};
					\node[anchor = south west] at (C2.30) {$Z$};
		
					\node[] (mid) at ($(C1)!0.5!(C2)$) {};
					\node[figlabel] at ($(mid)+(0,-4.5)$) {(b)};
		
					\draw[bend left =45] (C2.210) to (C2.330);
		
					\node[vtx, label=below:{$x$}] (q) at (C0.90) [below=.15cm] {};		
		
					\node[vtx] (p) at (C1.330) [left=.15cm] {};
					\node[] at ($(p)+(-0.225,-0.175)$) {$y$};
		
					\node[vtx] (pp) at (C1.180) [right=.45cm] {};
					\node[] at ($(pp)+(0,0.3)$) {$y'$}; 
		
					\node[vtx, label=right:{$z$}] (s) at (C2.270) [above=.45cm] {};
					\node[vtx, label=right:{$z'$}] (ss) at (C2.90) [below=.45cm] {};
		
					\draw[edge] (p) to (pp);
					\draw[edge] (p) to (q);
					\draw[edge] (pp) to (q);
		
					\draw[edge] (p) to (s);
					\draw[edge] (p) to (ss);
					\draw[edge] (s) to (ss);
		
					\node[vtx] (ppp) at (C1.30) [left=.15cm] {};
					\node[] at ($(ppp)+(-0.275,0.125)$) {$y''$};
		
					\draw[nonedge, bend left=15] (q) to (ppp);
					\draw[edge] (p) to (ppp);
					\draw[edge] (pp) to (ppp);
					\draw[edge] (ppp) to (s);
					\draw[edge] (ppp) to (ss);
							\node[] at ($(pp)+(5,-1)$) {$B$};

				\end{scope}
			\end{tikzpicture}
		}
		\caption{Visual aides to the proof of Lemma~\ref{claim:step:crossing-general}.}
		\label{fig:step:crossing}
	\end{figure}
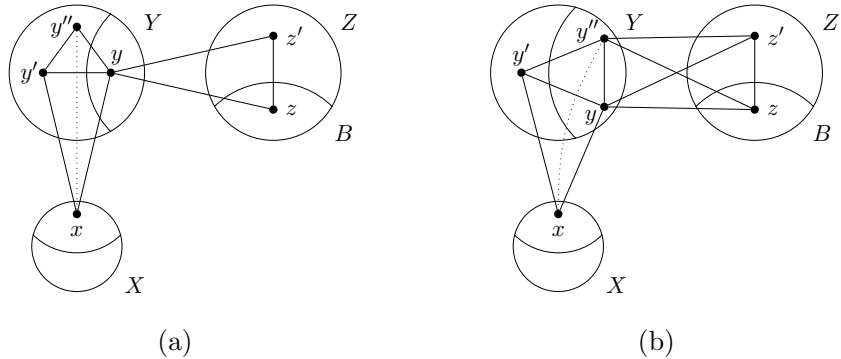
		\begin{lemma}\label{claim:step:crossing-general}
		Let $XYZ$ be a boundary-crossing path in $\calT$ such that $\bdd{Y}{Z}$ is complete to $Z$, and $B$ be a non-empty proper subset of $Z$.
		If one of the following conditions does not hold, then
		one can in polynomial time output an induced subgraph $H$ of $G[X \cup X'\cup Y \cup Z]$ for some neighbor $X'$ of $Y$ in $\calT$ ($X'$ can be $X$) that is isomorphic to 
		$W_{1, 1}^-$ or $W_{2, 0}^-$, 
		with the terminal vertex being mapped to a vertex in $B$, say $r_H$,
		such that $V(H) \cap B = \{r_H\}$.	
		\begin{enumerate}[label={\normalfont(\roman*)}]
			\item\label{claim:step-crossing-general:complete} $\bdd{X}{Y}$ is complete to $Y$.
			\item\label{claim:step-crossing-general:unique} There is no bag $X' \in V(\calT)\setminus \{X\}$ that contains vertices crossing $\bdd{Y}{Z}$.
		\end{enumerate}
	\end{lemma}
	\begin{proof}
		Let $x \in \bdd{X}{Y}$ be a vertex that crosses $\bdd{Y}{Z}$. 
		Choose a neighbor $y$ in $\bdd{Y}{Z}$ and a neighbor $y'$ in $Y \setminus \bdd{Y}{Z}$ of $x$. 
		Since $\bdd{Y}{Z}$ is complete to $Z$ by assumption, $y$ has a neighbor in $B$ and a neighbor in $Z \setminus B$. Let $z$ and $z'$ be these neighbors, respectively. We illustrate this situation and the following arguments in Figure~\ref{fig:step:crossing}.
		
		Suppose that~\ref{claim:step-crossing-general:complete} does not hold, i.e.\ that
		$\bdd{X}{Y}$, in particular the vertex $x$ is not complete to $Y$. Then, $x$ has a non-neighbor, say $y''$ in $Y$. If $y'' \in Y\setminus\bdd{Y}{Z}$, then the set $\{x, y, y', y'', z, z'\}$ induces a $W_{1, 1}^-$ with the terminal being mapped to $z$. See Figure~\ref{fig:step:crossing}(a). On the other hand, if $y'' \in \bdd{Y}{Z}$, then $\{x, y, y', y'', z, z'\}$ induce a $W_{2, 0}^-$ with the terminal vertex being mapped to $z$. See Figure~\ref{fig:step:crossing}(b).
		
		Now suppose that~\ref{claim:step-crossing-general:unique} does not hold, and let $x' \in X'$ be a vertex crossing $\bdd{Y}{Z}$. Then, $x'$ has a neighbor $y \in \bdd{Y}{Z}$ and a neighbor $y' \in Y \setminus \bdd{Y}{Z}$. Let $x \in X$. By (i), $x$ is adjacent to $y$ and $y'$. Then $G[X\cup X'\cup Y\cup Z]$ contains a $W_{1, 1}^-$ with terminal $z$. 
	\end{proof}

	We use the following lemmas to find an obstruction or extend a good boundary-crossing path.
	
	\begin{lemma}\label{lem:extendgood-pre}
	Let $G$ be a connected well-partitioned chordal graph with partition tree $\mathcal{T}$, and let $B$ be a vertex set contained in some bag $C_1$.
	If $C_kC_{k-1} \ldots C_1$ is a boundary-crossing path for some $k\ge 2$ such that $C_2$ has a vertex that crosses $B$ in $C_1$, then 
	one can in polynomial time either 
	\begin{enumerate}
	\item\label{extendgood-pre:1} output an induced subgraph $H$ isomorphic to $W_{s,t}^-$ for some $s\in \{1,2,3\}$ and $t\ge 0$ with terminal $v$ such that $V(H)\cap B=\{v\}$,
	\item\label{extendgood-pre:2} output an induced subgraph $H$ isomorphic to $W_{1,0}^-$ on $\{a, z_1, z_2, w\}$ such that both $a$ and $w$ have degree $2$ in $H$, $a\in \bdd{D}{C_1}$ for some neighbor bag $D$ of $C_1$, $z_2\in C_1\setminus B$, and $z_1, w\in B$, or
	\item\label{extendgood-pre:3} verify that it is a good boundary-crossing path such that $\bdd{C_2}{C_1}$ is complete to $C_1$ and no other bag contains a vertex crossing $B$ in $C_1$.	  
	\end{enumerate}
	\end{lemma}
	\begin{proof}
	We prove the lemma by induction on $k$.
	Assume that $k=2$. 
	We check whether $\bdd{C_2}{C_1}$ is complete to $C_1$. 
	Suppose not.
	Let $a\in \bdd{C_2}{C_1}$.
	Let $z_1$ be a neighbor of $a$ in $B$, 
	$z_2$ be a neighbor of $a$ in $C_1\setminus B$, 
	and $w$ be a non-neighbor of $a$ in $C_1$.
	If $w\in C_1\setminus B$, then $\{a,z_1, z_2, w\}$ induces $W_{1,0}^-$ with terminal $z_1$,
	so we have outcome~\ref{extendgood-pre:1}.
	If $w\in B$, then $\{a,z_1,z_2,w\}$ induces a graph as in case~\ref{extendgood-pre:2}.
	Otherwise, we conclude that $\bdd{C_2}{C_1}$ is complete to $C_1$.

	We find a bag $D\neq C_2$ in $\calT$ containing a vertex $d$ crossing $B$ in $C_1$.
		If such a vertex $d$ exists, then by the above procedure, we may assume that $d$ is complete to $C_1$.
	Then similarly to the previous case when $w\in C_1\setminus B$, again we can find an induced subgraph isomorphic to the diamond on $\{a,d,z_1,z_2\}$.
	If such a vertex does not exist, then we can conclude that no other bag contains a vertex crossing $B$ in $C_1$.

	Now, we assume that $k\ge 3$.
	By the induction hypothesis, the claim holds for the path $C_{k-1}C_{k-2} \ldots C_1$.
		We can assume that it is good.
		
		We check whether $\bdd{C_k}{C_{k-1}}$ is not complete to $C_{k-1}$,
		and there is a bag $D \in V(\calT) \setminus \{C_k\}$ that has a vertex crossing $\bdd{C_{k-1}}{C_{k-2}}$.
		If neither of them holds, then we verified that $C_kC_{k-1} \ldots C_1$ is good.
		Assume one of two statements holds.
		
	Let $X\defeq \bdd{C_{k-2}}{C_{k-3}}$ if $k\ge 4$ and $X=B$ if $k=3$.
		Now, by applying Lemma~\ref{claim:step:crossing-general} to the pair $(C_kC_{k-1}C_{k-2}, X)$, 
		we can find an induced subgraph $H$ isomorphic to $W_{1, 1}^-$ or $W_{2, 0}^-$ in $G[C_k \cup D\cup C_{k-1} \cup C_{k-2}]$ for some neighbor $D$ of $C_{k-1}$ in $\calT$
		so that its terminal $r$ is mapped to some vertex in $X$ and $V(H)\cap X=\{r\}$. 
		If $k=3$, then we have outcome~\ref{extendgood-pre:1} as $X=B$.
		
		Assume $k\ge 4$.
		Let $x_{k-2}\defeq r$.
		We recursively choose pairs of vertices $(x_i,y_i)$ for $i\in \{1,2, \ldots, k-3\}$ as follows.
		First assume $i>1$ and $x_{i+1}$ is defined but $x_i$ is not defined yet.
		Then choose a neighbor $x_i$ of $x_{i+1}$ in $\bdd{C_i}{C_{i-1}}$ and a neighbor 
		$y_i$ of $x_{i+1}$ in $C_i\setminus \bdd{C_i}{C_{i-1}}$.
		Such neighbors exists since $x_{i+1}$ crosses $\bdd{C_i}{C_{i-1}}$.
		When $i=1$, choose a neighbor $x_1$ of $x_2$ in $B$ and $y_1$ of $x_2$ in $C_1\setminus B$. 
		Then it is clear that $G[ \{x_1, y_1, x_2, y_2, \ldots, x_{k-3}, y_{k-3} \} \cup V(H)]$ is isomorphic to $W_{s', t'}^-$ for some $s'\in \{1,2\}$ and $t'\ge 0$ with terminal $x_1$ such that its intersection on $B$ is exactly $x_1$. 
		This concludes the lemma.
	\end{proof}
	
	\begin{lemma}\label{lem:extendgood}
	Let $G_1$ and $G_2$ be two connected graphs with non-empty sets $A\subseteq V(G_1)$ and $B\subseteq V(G_2)$, and 
	$G$ be the graph obtained from the disjoint union of $G_1$ and $G_2$ by adding all edges between $A$ and $B$ such that 
	\begin{itemize} 
		\item for every $v\in B$, $G[V(G_1)\cup \{v\}]$ is isomorphic to $W^-_{s,t}$ with terminal $v$ for some $s\in \{1,2\}$ and $t\ge 0$,
		\item $G_2$ is a well-partitioned chordal graph with a partition tree $\mathcal{T}$ such that $B$ is contained in some bag $C_1$.
	\end{itemize}
	Then the following two statements hold.
	\begin{enumerate}[label=(\arabic*)]
		\item If $C_kC_{k-1} \ldots C_1$ is a boundary-crossing path in $\calT$ for some $k\ge 2$ such that $C_2$ has a vertex that crosses $B$ in $C_1$, then 
	one can in polynomial time either output an obstruction in $G$, 
	or verify that it is a good boundary-crossing path such that $\bdd{C_2}{C_1}$ is complete to $C_1$ and no other bag contains a vertex crossing $B$ in $C_1$.	  
		\item If $C_2C_1$ is a boundary-crossing path in $\calT$, that is, an edge in $\calT$, then one can in polynomial time either output an obstruction in $G$, 
		or find a maximal good boundary-crossing path ending in $C_2C_1$ such that $\bdd{C_2}{C_1}$ is complete to $C_1$ and no other bag contains a vertex crossing $B$ in $C_1$.	 
		\end{enumerate}
	\end{lemma}
	\begin{proof}
	We prove (1). We apply Lemma~\ref{lem:extendgood-pre} to $G_2$ and $B$, then 
	we conclude that in polynomial time, we can either
	\begin{enumerate}
	\item output an induced subgraph $H$ isomorphic to $W_{s,t}^-$ for some $s\in \{1,2,3\}$ and $t\ge 0$ with terminal $v$ such that $V(H)\cap B=\{v\}$,
	\item output an induced subgraph $H$ isomorphic to the diamond on $\{a, z_1, z_2, w\}$ such that $a, w$ have degree $2$ in $H$,  $a\in \bdd{D}{C_1}$ for some neighbor bag $D$ of $C_1$, $z_2\in C_1\setminus B$, and $z_1, w\in B$, or
	\item verify that it is a good boundary-crossing path such that $\bdd{C_2}{C_1}$ is complete to $C_1$ and no other bag contains a vertex crossing $B$ in $C_1$.	  
	\end{enumerate}
   For case (i) it is clear that together with an obstruction $W^-_{s,t}$ in $G[V(G_1)\cup \{v\}]$ given by the assumption, 
   $G[V(H)\cup V(G_1)]$ is isomorphic to $W_{s,t}$ for some $s\in \{1,2,3\}$ and $t\ge 0$.
   For case (ii), we can observe that $G[V(H)\cup V(G_1)]$ is an obstruction as follows.
   \begin{itemize}
   	\item If $G[V(G_1)\cup\{z_1\}]$ is isomorphic to $W^-_{1,0}$, then $G[V(G_1)\cup\{a,w,z_1,z_2\}]$ is isomorphic to $O_3$.
   	\item If $G[V(G_1)\cup\{z_1\}]$ is isomorphic to $W^-_{2,0}$, then $G[V(G_1)\cup\{a,w,z_1,z_2\}]$ is isomorphic to $O_4$.
   	\item If $G[V(G_1)\cup\{z_1\}]$ is isomorphic to $W^-_{s,t}$ for some $s\in\{1,2\}$ and $t\geq1$, then $G[V(G_1)\cup\{a,w,z_1,z_2\}]$ is isomorphic to $W_{s',t-1}$ for some $s'\in\{2,3\}$.
   \end{itemize}
   It shows the statement (1).
		
		Now, we show (2).
	By (1), we can in polynomial time either output an obstruction, 
	or verify that $\bdd{C_2}{C_1}$ is complete to $C_1$ and no other bag crosses $\bdd{C_2}{C_1}$.
	For $i\ge 3$,
	we recursively find a neighbor bag $C_i$ of $C_{i-1}$ that has a vertex crossing $\bdd{C_{i-1}}{C_{i-2}}$.
	If there is such a bag $C_i$, 
	then by applying (1),
	one can in polynomial time find an obstruction or 
	guarantee that $C_iC_{i-1} \ldots C_1$ is good.
		As the graph is finite, this procedure terminates with some path $C_kC_{k-1} \ldots C_3$ such that 
	it is good and no vertex crosses $\bdd{C_k}{C_{k-1}}$, unless we found an obstruction.
	\end{proof}

\subsection{A certifying algorithm}\label{subsec:certifying}
In this subsection, we prove the following.
\begin{proposition}\label{prop:char}
	Given a graph $G$, one can in polynomial time 
	either output an obstruction in $G$ 
	or output a partition tree of $G$ confirming that $G$ is a well-partitioned chordal graph.
\end{proposition}

As explained in Subsection~\ref{subsec:partial},
	we mainly consider the case when $G$ is a connected chordal graph, $v$ is a simplicial vertex of $G$ and $G-v$ is a connected well-partitioned chordal graph with partition tree $\calT$, and $v$ has neighbors in two distinct bags $C_1$ and $C_2$.
	With these assumptions, 
 we deal with the following three cases, and in each case, we show that either one can in polynomial time find an obstruction or output a partition tree of $G$.
 \begin{itemize}
 	\item (Lemma~\ref{lem:charcase1}) $C_1\subseteq N_G(v)$.
	\item (Lemma~\ref{lem:charcase2}) $\bdd{C_1}{C_2} \setminus N_G(v) \neq \emptyset$ and  $C_2 \setminus N_G(v) \neq \emptyset$.
	\item (Lemma~\ref{lem:charcase3}) $C_1 \setminus N_G(v) \neq \emptyset$,  $C_2 \setminus N_G(v) \neq \emptyset$ and $N_G(v)=\bdd{C_1}{C_2}\cup \bdd{C_2}{C_1}$.
 \end{itemize}
 We give a proof of Proposition~\ref{prop:char} assuming that these lemmas hold.
\begin{proof}[Proof of Proposition~\ref{prop:char}]
	We apply Theorem~\ref{thm:nikolo} to find a hole in $G$ if one exists.  We may assume that $G$ is chordal.
	Since a graph is a well-partitioned chordal graph if and only if its connected components are well-partitioned chordal graphs, 
	it is sufficient to show it for each connected component. From now on, we assume that $G$ is connected.
	Using the algorithm in Theorem~\ref{thm:generate}, we can find a perfect elimination ordering $(v_1, v_2, \ldots, v_n)$ of $G$ in polynomial time.
	
	For each $i\in \{1, 2, \ldots, n \}$, let $G_i\defeq G[\{v_i, v_{i+1}, \ldots, v_n\}]$.
	Observe that since $G$ is connected and $v_i$ is simplicial in $G_i$ for all $1\le i\le n-1$, each $G_i$ is connected.
	From $i=n$ to $1$, we recursively find either an obstruction or a partition tree of $G_i$.	
	Clearly, $G_n$ admits a partition tree.
	Let $1\le i\le n-1$, and assume that we obtained a partition tree $\calT$ of $G_{i+1}$.
	Recall that $v_i$ is simplicial in $G_i$.

	Since $v_i$ is simplicial in $G_i$, $N_{G_i}(v_i)$ is a clique. 
	This implies that there are at most two bags in $V(\calT)$ that have a non-empty intersection with $N_{G_i}(v_i)$. 
	If there is only one such bag in $V(\calT)$, say $C$, we can construct a partition tree for $G_i$ 
	by simply adding a bag consisting of $v_i$ and making it adjacent to $C$. 
	
	Hence, from now on, we can assume that there are precisely two distinct adjacent bags $C_1, C_2 \in V(\calT)$ that have a non-empty intersection with $N_{G_i}(v_i)$.
	As $N_{G_i}(v_i)$ is a clique, we can observe that $N_{G_i}(v_i) \subseteq \bdd{C_1}{C_2} \cup \bdd{C_2}{C_1}$.

	If $C_1\subseteq N_{G_i}(v_i)$ or $C_2\subseteq N_{G_i}(v_i)$, then by Lemma~\ref{lem:charcase1}, 
	we can in polynomial time either output an obstruction or output a partition tree of $G_i$.
	Thus, we may assume that $C_1 \setminus N_{G_i}(v_i) \neq \emptyset$ and	 $C_2 \setminus N_{G_i}(v_i) \neq \emptyset$.
	If $\bdd{C_1}{C_2} \setminus N_{G_i}(v_i) \neq \emptyset$ or $\bdd{C_2}{C_1} \setminus N_{G_i}(v_i) \neq \emptyset$, 
	then by Lemma~\ref{lem:charcase2}, 
	we can in polynomial time either output an obstruction or output a partition tree of $G_i$.
	Thus, we may further assume that 
	$\bdd{C_1}{C_2} \setminus N_{G_i}(v_i)=\emptyset$ and $\bdd{C_2}{C_1} \setminus N_{G_i}(v_i)=\emptyset$.
	Then by Lemma~\ref{lem:charcase3}, 
	we can in polynomial time either output an obstruction or output a partition tree of $G_i$, and this concludes the proposition.
\end{proof}

Now, we focus on proving the three lemmas.

\begin{lemma}\label{lem:charcase1}
	If $C_1\subseteq N_G(v)$,  
	then one can in polynomial time either
	output an obstruction in $G$ or output a partition tree of $G$ confirming that $G$ is a well-partitioned chordal graph.
\end{lemma}
\begin{proof}
Since $v$ is a simplicial vertex, we have that $\bdd{C_1}{C_2} = C_1$.
If $N_G(v) \cap C_2 = \bdd{C_2}{C_1}$, then we can obtain a partition tree for $G$ by adding $v$ to $C_1$. 
	Thus, we may assume that $N_G(v) \cap C_2 \neq \bdd{C_2}{C_1}$. 
	
	Assume that $C_2 = \bdd{C_2}{C_1}$.
	Since $\bdd{C_2}{C_1}$ is complete to $C_1$, we have that $C_1 \cup C_2$ is a clique. Hence, we can obtain a partition tree $\calT'$ for $G$ from $\calT$ by removing $C_1$ and $C_2$, adding a new bag $C^* = C_1 \cup C_2$, making all neighbors of $C_1$ and $C_2$ in $\calT$ adjacent to $C^*$, and adding a new bag $C_v \defeq \{v\}$ and making $C_v$ adjacent to $C^*$.
	Thus, we may assume that $C_2 \setminus \bdd{C_2}{C_1} \neq \emptyset$.
	
	Since $C_1= \bdd{C_1}{C_2}$, no vertex of $G-v$ crosses $\bdd{C_1}{C_2}$.
	If no vertex of $G-v$ crosses $\bdd{C_2}{C_1}$,
	then by Lemma~\ref{lem:base:nocrossing}, we can obtain a partition tree for $G$ in polynomial time.
	Thus, we may assume that there is a bag $C_3$ having a vertex that crosses $\bdd{C_2}{C_1}$.
	So, $C_3C_2C_1$ is a boundary-crossing path.
	We will find either an obstruction or a maximal good boundary-crossing path ending in $C_3C_2C_1$.
	We first check that $C_3C_2 C_1$ is good, unless some obstruction from $\obswp$ appears.

	\begin{figure}
		\centering
			\scalebox{0.8}{
				\tikzset{
					bag/.style={
						shape=circle,
						draw,
						minimum width=2cm,
					}
				}
					\begin{tikzpicture}
						\begin{scope}[]
							\node[smallbag] (C1) at (0,0) {};
							\node[anchor = north west] at (C1.330) {$C_{1}$};
			
							\draw[bend right=45] (C1.165) to (C1.15);
			
							\node[vtx, label=below:{$z_1$}] (z1) at (C1.90) [below=.15cm] {};	
			
							\node[bag] (C2) at (C1.90) [above=.5cm] {};	
							\node[anchor = south west] at (C2.30) {$C_2$};
			
							\draw[bend left=45] (C2.195) to (C2.345);	
			
							\node[vtx] (w) at (C2.240) [above=.25cm] {};
							\node[] at ($(w)+(-0.2,0.2)$) {$w$};	
			
							\node[vtx] (z2) at (C2.300) [above=.25cm] {};
							\node[] at ($(z2)+(0.2,0.2)$) {$z_2$};
							\node[vtx, label=above:{$a$}] (a) at (C2.90) [below=.5cm] {};
			
							\node[vtx, label=right:{$v$}] (v) at ($(C1) !0.45! (C2)$) [right = 1.25cm] {}; 
			
							\draw[edge] (v) to (z1);
							\draw[edge] (v) to (z2);
							\draw[edge] (z1) to (z2);
							\draw[edge] (z1) to (w);
							\draw[edge] (w) to (a);
							\draw[edge] (z2) to (a);
							\draw[edge] (w) to (z2);	
		
						\end{scope}
		
						\begin{scope}[xshift=3.5cm,yshift=1cm]
							\node[vtx, label=left:{$z_2$}] (z2) at (0,0) {};
							\node[vtx, label=left:{$a$}, above of = z2] (a) {};
							\node[vtx, label=left:{$v$}, below of = z2] (v) {};
								
							\node[vtx, label=right:{$z_1$}, right of = v] (z1) {};
							\node[vtx, label=right:{$w$}, right of = z2] (w) {};
							\node[vtx, label=right:{$x$}, right of = a] (x) {};
				
							\draw (v) -- (z1);
							\draw (v) -- (z2);
							\draw (z1) -- (z2);
							\draw (w) -- (z2);
							\draw (a) -- (z2);
							\draw (x) -- (a);
							\draw (x) -- (w);
							\draw (a) -- (w);
							\draw (z1) -- (w);
						\end{scope}
		
						\begin{scope}[xshift=7cm, yshift=1cm]
							\node[vtx, label=left:{$z_2$}] (z2) at (0,0) {};
							\node[vtx, label=left:{$x$}, above of = z2] (x) {};
							\node[vtx, label=left:{$v$}, below of = z2] (v) {};
								
							\node[vtx, label=right:{$z_1$}, right of = v] (z1) {};
							\node[vtx, label=right:{$w$}, right of = z2] (w) {};
							\node[vtx, label=right:{$a$}, right of = x] (a) {};
				
							\draw (v) -- (z1);
							\draw (v) -- (z2);
							\draw (z1) -- (z2);
							\draw (w) -- (z2);
							\draw (a) -- (z2);
							\draw (x) -- (a);
							\draw (x) -- (z2);
							\draw (a) -- (w);
							\draw (z1) -- (w);
						\end{scope}
		
						\begin{scope}[xshift=10.5cm, yshift=1cm]
							\node[vtx, label=left:{$z_2$}] (z2) at (0,0) {};
							\node[vtx, label=left:{$v$}, below of = z2] (v) {};
							\node[vtx, label=left:{$x$}, above of = z2] (x) {};
								
							\node[vtx, label=right:{$z_1$}, right of = v] (z1) {};
							\node[vtx, label=right:{$w$}, right of = z2] (w) {};
							\node[vtx, label=above:{$a$}, above of = w] (a) {};
							\node[vtx, label=right:{$y$}, right of = a] (y) {};
				
							\draw (v) -- (z1);
							\draw (v) -- (z2);
							\draw (z1) -- (z2);
							\draw (w) -- (z2);
							\draw (a) -- (z2);
							\draw (a) -- (w);
							\draw (z1) -- (w);
							\draw (x) -- (a);
							\draw (x) -- (w);
							\draw (x) -- (z2);
							\draw (y) -- (a);
							\draw (y) -- (w);
							\draw (y) -- (z2);
						\end{scope}
					\end{tikzpicture}
			}
			\caption{Proof of Claim~\ref{claim:no}.}
			\label{fig:claim:no}
		\end{figure}
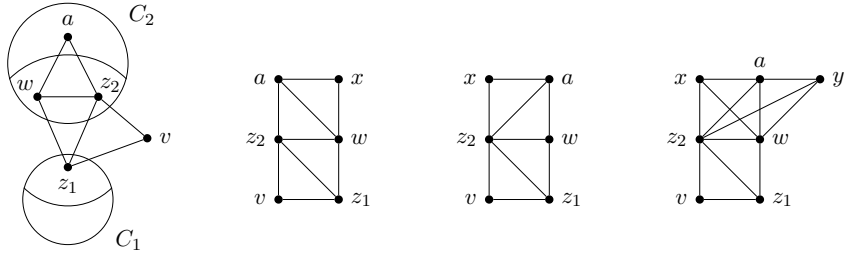
	
	\begin{nestedclaim}\label{claim:no}
		Let $z_1\in N_G(v)\cap C_1$, $z_2\in N_G(v)\cap C_2$, $w\in  \bdd{C_2}{C_1}\setminus N_G(v)$, and $a\in C_2\setminus \bdd{C_2}{C_1}$.
		\begin{enumerate}[label={(\roman*)}]
			\item\label{claim:no:x:aw} If there is a vertex $x \in V(G) \setminus \{v, a, w, z_1, z_2\}$ such that~$N(x) \cap \{v, a, w, z_1, z_2\} = \{a, w\}$, 
			then $G[\{v,a,w,z_1,z_2,x\}]$ is isomorphic to \ABar{}.
			\item\label{claim:no:x:az2} If there is a vertex $x \in V(G) \setminus \{v, a, w, z_1, z_2\}$ such that~$N(x) \cap \{v, a, w, z_1, z_2\} = \{a, z_2\}$, 
			then $G[\{v,a,w,z_1,z_2,x\}]$ is isomorphic to \FourFan{}.
			\item\label{claim:no:xy:awz2} If there is a pair of distinct non-adjacent vertices $x, y \in V(G) \setminus \{v, a, w, z_1, z_2\}$ 
		such that $N(x) \cap \{v, a, w, z_1, z_2\} = N(y) \cap \{v, a, w, z_1, z_2\} = \{a, w, z_2\}$, then 
		$G[\{v,a,w,z_1,z_2,x,y\}]$ is isomorphic to \AhnOne{}.
		\end{enumerate}
	\end{nestedclaim}
	\begin{claimproof}
	It is straightforward to check it; see Figure~\ref{fig:claim:no}.
	\end{claimproof}
	
	\begin{nestedclaim}\label{claim:base:crossing}
		One can in polynomial time output an obstruction 
		or verify that $C_3C_2C_1$ is good.
	\end{nestedclaim}
	\begin{claimproof}
		We consider the bag $C_3$, and first check whether $\bdd{C_3}{C_2}$ is complete to $C_2$.
		If so, then we are done.
		Otherwise, choose a vertex $p \in \bdd{C_3}{C_2}$,
		and a non-neighbor $q$ of $p$ in $C_2$. 
		As $p$ crosses $\bdd{C_2}{C_1}$, 
		$p$ has a neighbor $a$ in $C_2\setminus \bdd{C_2}{C_1}$ and a neighbor $b$ in $\bdd{C_2}{C_1}$.
		There are three possibilities; $q$ is contained in one of $N_G(v)\cap C_2$, $\bdd{C_2}{C_1}\setminus N_G(v)$, or $C_2\setminus \bdd{C_2}{C_1}$.
		Let $z_1\in N_G(v)\cap C_1$.
		
		If $q$ and $b$ are in distinct parts of $N_G(v)\cap C_1$ and $\bdd{C_2}{C_1}\setminus N_G(v)$, 
		then $G[\{p, q, z_1, a, b, v\}]$ is isomorphic to \ABar{} or \FourFan{} by Claims~\ref{claim:no}\ref{claim:no:x:aw} and \ref{claim:no:x:az2}.
		Assume $q$ and $b$ are in the same part of $N_G(v)\cap C_1$ or $\bdd{C_2}{C_1}\setminus N_G(v)$.
		Then by the previous argument, we may assume that $p$ is complete to the set, one of $N_G(v)\cap C_1$ and $\bdd{C_2}{C_1}\setminus N_G(v)$,
		that does not contain $q$.
		Then by choosing a vertex in this set, we can again output \ABar{} or \FourFan{}.
		Thus, we may assume that $q$ is contained in $C_2\setminus \bdd{C_2}{C_1}$ and $p$ is complete to $\bdd{C_2}{C_1}$.
		Then $q\neq a$ and by using vertices from $N_G(v)\cap C_1$ and $\bdd{C_2}{C_1}\setminus N_G(v)$ together with $\{a,p,q,v,z_1\}$, we can output \AhnOne{} by Claim~\ref{claim:no}\ref{claim:no:xy:awz2}.
		
		To verify whether $C_3 C_2 C_1$ is exclusive, we check if there exists another neighbor bag $D\neq C_3$ of $C_2$ having a vertex $q$ that crosses $\bdd{C_2}{C_1}$.
		If there is such a vertex $q$, then by applying the previous procedure, we may assume that $q$ is complete to $C_2$. 
		Then by using vertices from each of $N_G(v)\cap C_2$, $\bdd{C_2}{C_1}\setminus N_G(v)$, 
		and $C_2\setminus \bdd{C_2}{C_1}$ together with $\{p,q,z_1,v\}$, we can output \AhnOne{} by Claim~\ref{claim:no}\ref{claim:no:xy:awz2}.
		Otherwise, $C_3C_2C_1$ is a good boundary-crossing path.
	\end{claimproof}

	By Claim~\ref{claim:base:crossing}, 
	we may assume that $C_3C_2C_1$ is good.
	If no bag contains a vertex crossing $\bdd{C_3}{C_2}$, 
	then $C_3C_2C_1$ is a maximal good boundary-crossing path. So, we may assume that 
	there is a bag $C_4$ containing a vertex crossing $\bdd{C_3}{C_2}$.
	
	We choose
	$z_1\in N_G(v)\cap C_1$, $z_2\in N_G(v)\cap C_2$, $w\in  \bdd{C_2}{C_1}\setminus N_G(v)$, and $a\in C_2\setminus \bdd{C_2}{C_1}$.
	To apply Lemma~\ref{lem:extendgood}, 
	let $G_1=G[\{v,z_1,z_2,w,a\}]$ and $G_2$ be the component of $G-V(C_2)$ that contains $C_3$ and 
	$G'=G[V(G_1)\cup V(G_2)]$.
	It is clear that $G'$ can be obtained from the disjoint union of $G_1$ and $G_2$ by adding edges between 
	$\bdd{C_3}{C_2}$ and $\{w,a,z_2\}$. 
	Also, for each vertex $p\in \bdd{C_3}{C_2}$, $\{p\}\cup V(G_1)$ is a wing of $W_{2,0}$ with terminal $p$.
	
	Thus, by (2) of Lemma~\ref{lem:extendgood}, we can in polynomial time either output an obstruction, 
	or find a maximal good boundary-crossing path ending in $C_4C_3$ in $G_2$ such that 
	$\bdd{C_4}{C_3}$ is complete to $C_3$ and no other bag contains a vertex crossing $C_3$.
	Thus, in the latter case, we obtain a maximal boundary-crossing path ending in $C_2C_1$ in $G-v$.
	We now repeatedly apply Lemma~\ref{lem:shorten:path} to modify $\calT$ along this path and obtain a partition tree $\calT'$ for $G-v$ such that no vertex crosses $\bdd{C_2}{C_1}$. Note that, for simplicity, we call again $C_1$ and $C_2$ the bags of $\calT'$ containing the neighbors of $v$. We can now apply Lemma~\ref{lem:base:nocrossing} to obtain a partition tree for the entire graph $G$ in polynomial time.
\end{proof}

\begin{lemma}\label{lem:charcase2}
		If $\bdd{C_1}{C_2} \setminus N(v) \neq \emptyset$ and  $C_2 \setminus N(v) \neq \emptyset$,  
	then one can in polynomial time either
	output an obstruction in $G$ or output a partition tree of $G$ confirming that $G$ is a well-partitioned chordal graph.
\end{lemma}
\begin{proof}
	We choose
 	 a neighbor $z_1$ of $v$ in $\bdd{C_1}{C_2}$, a neighbor $z_2$ of $v$ in $\bdd{C_2}{C_1}$ and 
	a non-neighbor $x$ of $v$ in $\bdd{C_1}{C_2}$.
	We first consider the case when $\bdd{C_1}{C_2} = C_1$.

	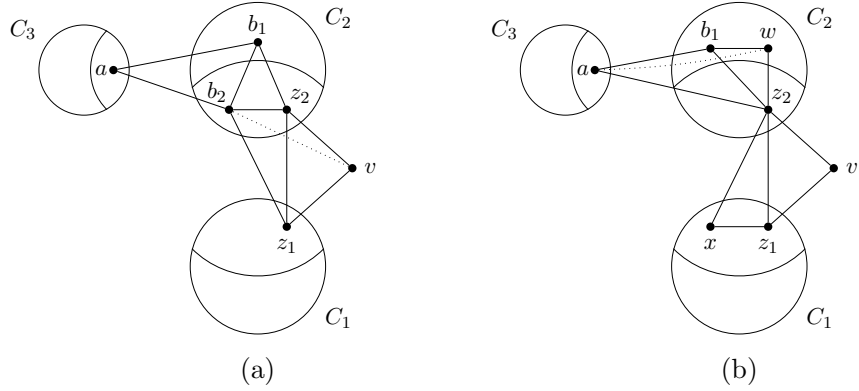
\begin{figure}
		\centering
		\scalebox{0.8}{
			\begin{tikzpicture}
				\begin{scope}[]
					\node[bag] (C2) at (0,0) {};
					\node[anchor = south west] at (C2.30) {$C_2$};
					\node[figlabel] at ($(C2)+(0,-5)$) {(a)};
	
					\draw[bend left=45] (C2.195) to (C2.345);
	
					\node[bag] (C1) at (C2.270) [below=1cm] {};
					\node[anchor = north west] at (C1.330) {$C_1$};
	
					\draw[bend right=45] (C1.165) to (C1.15);
	
					\node[vtx, label=below:{$z_1$}] (z1) at (C1.65) [below=.3cm] {};
	
					\node[vtx] (z2) at (C2.295) [above=.3cm] {};
					\node[] at ($(z2)+(0.225,0.225)$) {$z_2$};
	
					\node[vtx, label=right:{$v$}] (v) at ($(C1) !0.5! (C2)$) [right=1.5cm] {};
	
					\draw[edge] (z1) to (z2);
					\draw[edge] (v) to (z1);
					\draw[edge] (v) to (z2);
	
					\node[smallbag] (C3) at (C2.180) [left=1cm] {};
					\node[anchor = south east] at (C3.150) {$C_3$};
	
					\draw[bend left=45] (C3.300) to (C3.60);
	
					\node[vtx] (a) at (C3.0) [left=.2cm] {};
					\node[] at ($(a)+(-0.2,0)$) {$a$};
		
					\node[vtx] (b1) at (C2.90) [below=.6cm] {};
					\node[] at ($(b1)+(0,0.3)$) {$b_1$};
					\node[vtx] (b2) at (C2.245) [above=.3cm] {};
					\node[] at ($(b2)+(-0.2,0.3)$) {$b_2$};
		
					\draw[edge] (a) to (b1);
					\draw[edge] (a) to (b2);
					\draw[edge] (b1) to (z2);
					\draw[edge] (b2) to (z2);
					\draw[edge] (b1) to (b2);
					\draw[edge] (b2) to (z1);
		
					\draw[nonedge] (v) to (b2);
				\end{scope}
	
				\begin{scope}[xshift=8cm]
					\node[bag] (C2) at (0,0) {};
					\node[anchor = south west] at (C2.30) {$C_2$};
					\node[figlabel] at ($(C2)+(0,-5)$) {(b)};
	
					\draw[bend left=45] (C2.195) to (C2.345);
	
					\node[bag] (C1) at (C2.270) [below=1cm] {};
					\node[anchor = north west] at (C1.330) {$C_1$};
	
					\draw[bend right=45] (C1.165) to (C1.15);
	
					\node[vtx, label=below:{$z_1$}] (z1) at (C1.65) [below=.3cm] {};
					\node[vtx, label=below:{$x$}] (x) at (C1.115) [below=.3cm] {};
	
					\node[vtx] (z2) at (C2.295) [above=.3cm] {};
					\node[] at ($(z2)+(0.225,0.225)$) {$z_2$};
	
					\node[vtx, label=right:{$v$}] (v) at ($(C1) !0.5! (C2)$) [right=1.5cm] {};
	
					\draw[edge] (x) to (z1);
					\draw[edge] (z1) to (z2);
					\draw[edge] (x) to (z2);
					\draw[edge] (v) to (z1);
					\draw[edge] (v) to (z2);
	
					\node[smallbag] (C3) at (C2.180) [left=1cm] {};
					\node[anchor = south east] at (C3.150) {$C_3$};
	
					\draw[bend left=45] (C3.300) to (C3.60);
	
					\node[vtx] (a) at (C3.0) [left=.2cm] {};
					\node[] at ($(a)+(-0.2,0)$) {$a$};
		
					\node[vtx] (b1) at (C2.115) [below=.6cm] {};
					\node[] at ($(b1)+(0,0.3)$) {$b_1$};
		
					\draw[edge] (a) to (b1);
					\draw[edge] (b1) to (z2);
					\draw[edge] (a) to (z2);
		
					\node[vtx] (w) at (C2.65) [below=.6cm] {};
					\node[] at ($(w)+(0,0.25)$) {$w$};
		
					\draw[edge] (w) to (b1);
					\draw[edge] (w) to (z2);
		
					\draw[nonedge, bend right=5] (a) to (w);		
		
				\end{scope}
			\end{tikzpicture}
		}
		\caption{Illustration of some obstructions appearing in the proof of Claim~\ref{claim:base2:crossing}.}
		\label{fig:case2-base}
	\end{figure}
	\begin{case}[$\bdd{C_1}{C_2} = C_1$]\label{case:big:2} 
	Note that no vertex in $G-v$ crosses $\bdd{C_1}{C_2}$. If no vertex in $G-v$ crosses $\bdd{C_2}{C_1}$, then we can obtain a partition tree of $G$ from $\calT$ by Lemma~\ref{lem:base:nocrossing}. We may assume that there is a bag $C_3$ containing a vertex $a$ that crosses $\bdd{C_2}{C_1}$. 
	
	\begin{nestedclaim}\label{claim:base2:crossing}
	One can in polynomial time output an obstruction 
		or verify that $C_3C_2C_1$ is good.
	\end{nestedclaim}
	\begin{claimproof}
		As $a$ crosses $\bdd{C_2}{C_1}$, 
		$a$ has a neighbor both in $C_2 \setminus \bdd{C_2}{C_1}$ and in $\bdd{C_2}{C_1}$.
		Let $b_1$ and $b_2$ be neighbors of $a$ in $C_2 \setminus \bdd{C_2}{C_1}$ and $\bdd{C_2}{C_1}$, respectively.
		
		Assume that $a$ and $v$ have no common neighbors.
		Then $b_2$ is not adjacent to $v$ and $z_2$ is not adjacent to $a$.
		So, $G[\{a, b_1, b_2, z_2, z_1, v\}]$ is isomorphic to \ABar{}, see Figure~\ref{fig:case2-base}(a).
		Thus, we may assume that $a$ and $v$ have the common neighbor in $\bdd{C_2}{C_1}$. 
		We assume that $z_2$ is a common neighbor.

		Now suppose that there is a vertex $w \in C_2 \setminus \bdd{C_2}{C_1}$ that is not adjacent to $a$. Recall that since $N(v) \cap C_2 \subseteq \bdd{C_2}{C_1}$, we have that $v$ is not adjacent to $w$. Thus, we can output a $W_{1, 0}$ on $\{v, x, z_1, z_2, a, b_1, w\}$, see Figure~\ref{fig:case2-base}(b).
		So, we may assume that $a$ is complete to $C_2 \setminus \bdd{C_2}{C_1}$.
		
		Assume that there is a vertex $w \in \bdd{C_2}{C_1}$ that is not adjacent to $a$. Note that $v$ may or may not be adjacent to $w$. If $v$ is adjacent to~$w$, then $G$ contains \AhnOne{} as an induced subgraph, and if $v$ is not adjacent to $w$, then $G$ contains \FourFan{} as an induced subgraph; both these cases are illustrated in Figure~\ref{fig:case2-base2}.
		Otherwise, we can conclude that $a$ is complete to~$C_2$.
		
		To check whether $C_3C_2C_1$ is exclusive,  
		we find a bag $D\neq C_3$ containing a vertex $w$ crossing $\bdd{C_2}{C_1}$. 
		If there is no such a vertex, then it is exclusive. Assume that such a vertex $w$ exists.
		By repeating the above argument, we may assume that $w$ is complete to $C_2$.
	Then, $G[\{v, z_1, z_2, x, a, b_1, w\}]$ is isomorphic to $W_{1, 0}$ (see Figure~\ref{fig:case2-base}(b), but note that in this case $w\notin C_2$).
	\end{claimproof}
	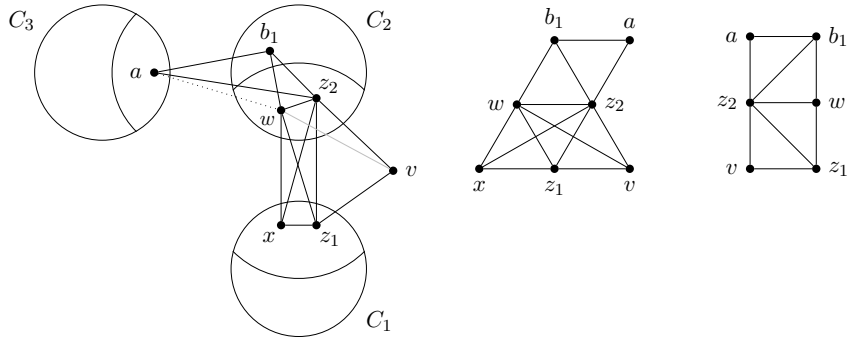
\begin{figure}
		\centering
		\scalebox{0.8}{
			\begin{tikzpicture}
				\begin{scope}[]
					\node[bag] (C2) at (0,0) {};
					\node[anchor = south west] at (C2.30) {$C_2$};
	
					\draw[bend left=45] (C2.195) to (C2.345);
	
					\node[bag] (C1) at (C2.270) [below=1cm] {};
					\node[anchor = north west] at (C1.330) {$C_1$};
	
					\draw[bend right=45] (C1.165) to (C1.15);
	
					\node[vtx] (z1) at (C1.75) [below=.3cm] {};
					\node[] at ($(z1)+(0.225,-0.225)$) {$z_1$};		
		
					\node[vtx] (x) at (C1.105) [below=.3cm] {};
					\node[] at ($(x)+(-0.2,-0.2)$) {$x$};
	
					\node[vtx] (z2) at (C2.285) [above=.6cm] {};
					\node[] at ($(z2)+(0.2,0.2)$) {$z_2$};
	
					\node[vtx, label=right:{$v$}] (v) at ($(C1) !0.5! (C2)$) [right=1.5cm] {};
	
					\draw[edge] (x) to (z1);
					\draw[edge] (z1) to (z2);
					\draw[edge] (x) to (z2);
					\draw[edge] (v) to (z1);
					\draw[edge] (v) to (z2);
	
					\node[bag] (C3) at (C2.180) [left=1cm] {};
					\node[anchor = south east] at (C3.150) {$C_3$};
	
					\draw[bend left=45] (C3.300) to (C3.60);
	
					\node[vtx, label=left:{$a$}] (a) at (C3.0) [left=.2cm] {};
		
					\node[vtx] (b1) at (C2.115) [below=.6cm] {};
					\node[] at ($(b1)+(0,0.3)$) {$b_1$};
		
					\draw[edge] (a) to (b1);
					\draw[edge] (b1) to (z2);
					\draw[edge] (a) to (z2);
		
					\node[vtx] (ww) at (C2.255) [above = .4cm] {};
					\node[scale=0.9] at ($(ww)+(-0.225,-0.175)$) {$w$};
		
					\draw[nonedge] (a) to (ww);
		
					\draw[edge] (ww) to (b1);
					\draw[edge] (ww) to (z2);
					\draw[edge] (ww) to (x);
					\draw[edge] (ww) to (z1);
					\draw[edge, highlighta] (ww) to (v);
				\end{scope}
	
				\begin{scope}[xshift=3cm, yshift=-1.6cm, node distance=1.25cm]
					\node[vtx, label=below:{$x$}] (x) at (0,0) {};
					\node[vtx, right of = x, label=below:{$z_1$}] (z1) {};
					\node[vtx, label=left:{$w$}] (w) at ($(x) !0.5! (z1)$) [above=1cm] {};
					\node[vtx, right of = w, label=right:{$z_2$}] (z2) {};		
		
					\node[vtx, right of = z1, label=below:{$v$}] (v) {};
					\node[vtx, label=above:{$b_1$}] (b1) at ($(w) !0.5! (z2)$) [above=1cm] {};
					\node[vtx, right of = b1, label=above:{$a$}] (a) {};
		
					\draw[edge] (x) to (w);
					\draw[edge] (x) to (z2);
					\draw[edge] (x) to (z1);
					\draw[edge] (z1) to (w);
					\draw[edge] (z1) to (z2);
					\draw[edge] (z1) to (v);
					\draw[edge] (v) to (w);
					\draw[edge] (v) to (z2);
					\draw[edge] (w) to (z2);
					\draw[edge] (w) to (b1);
					\draw[edge] (z2) to (b1);
					\draw[edge] (z2) to (a);
					\draw[edge] (a) to (b1);
				\end{scope}
	
				\begin{scope}[xshift=7.5cm, yshift=-1.6cm, node distance=1.1cm]
					\node[vtx, label=left:{$v$}] (v) at (0,0) {};
					\node[vtx, right of = v, label=right:{$z_1$}] (z1) {};
					\node[vtx, above of = v, label=left:{$z_2$}] (z2) {};
					\node[vtx, above of = z1, label=right:{$w$}] (w) {};
					\node[vtx, above of = z2, label=left:{$a$}] (a) {};
					\node[vtx, above of = w, label=right:{$b_1$}] (b1) {};
		
					\draw[edge] (v) to (z1);
					\draw[edge] (v) to (z2);
					\draw[edge] (z1) to (z2);
					\draw[edge] (z1) to (w);
					\draw[edge] (z2) to (w);
					\draw[edge] (z2) to (a);
					\draw[edge] (z2) to (b1);
					\draw[edge] (w) to (b1);
					\draw[edge] (a) to (b1);
				\end{scope}
			\end{tikzpicture}
		}
		\caption{Illustration of some more obstructions appearing in the proof of Claim~\ref{claim:base2:crossing}. Note that the edge between $v$ and $w$ may or may not be present, depending on which we either have an \AhnOne{} or an \FourFan{} as an induced subgraph in $G$.}
		\label{fig:case2-base2}
	\end{figure}	

	By Claim~\ref{claim:base2:crossing}, we may assume that $C_3C_2C_1$ is good. Let $a\in C_2\setminus \bdd{C_2}{C_1}$.
	If no bag contains a vertex crossing $\bdd{C_3}{C_2}$, 
	then $C_3C_2C_1$ is a maximal good boundary-crossing path. So, we may assume that
	there is a bag $C_4$ containing a vertex crossing $\bdd{C_3}{C_2}$. 
	
	To apply Lemma~\ref{lem:extendgood}, 
	let $G_1=G[\{v,x, z_1, z_2, a\}]$ and $G_2$ be the component of $G-V(C_2)$ containing $C_3$ and 
	$G'=G[V(G_1)\cup V(G_2)]$.
	It is clear that $G'$ can be obtained from the disjoint union of $G_1$ and $G_2$ by adding edges between 
	$\bdd{C_3}{C_2}$ and $\{a,z_2\}$. 
	Observe that for each vertex $p\in \bdd{C_3}{C_2}$, $G[\{p, a, v, z_1, z_2, z\}]$ is isomorphic to $W^-_{1,1}$ with terminal $p$.
	
	By (2) of Lemma~\ref{lem:extendgood}, we can in polynomial time either output an obstruction, 
	or find a maximal good boundary-crossing path ending in $C_4C_3$ in $G_2$ such that 
		$\bdd{C_4}{C_3}$ is complete to $C_3$ and no other bag contains a vertex crossing $C_3$.
Thus, in the latter case, we obtain a maximal boundary-crossing path ending in $C_2C_1$ in $G-v$.
	We can now repeatedly apply Lemma~\ref{lem:shorten:path} to modify $\calT$ along this path and obtain a partition tree $\calT'$ for $G-v$ such that no vertex crosses $\bdd{C_2}{C_1}$. Note that, for simplicity, we call again $C_1$ and $C_2$ the bags of $\calT'$ containing the neighbors of $v$. We can now apply Lemma~\ref{lem:base:nocrossing} to obtain a partition tree for the entire graph $G$ in polynomial time.
	\end{case}
	
		\begin{case}[$C_1 \setminus \bdd{C_1}{C_2} \neq \emptyset$]
		If there is no vertex crossing $\bdd{C_1}{C_2}$ and no vertex crossing $\bdd{C_2}{C_1}$ in $G-v$, then by Lemma~\ref{lem:base:nocrossing}, 
		one can output a partition tree of $G$ from $\calT$ in polynomial time. 
	Recall that we have neighbors of $v$, namely $z_1 \in \bdd{C_1}{C_2}$ and $z_2 \in \bdd{C_2}{C_1}$, and a non-neighbor of $v$, namely $x \in \bdd{C_1}{C_2}$.	
		
		\begin{nestedclaim}
			If there is a vertex crossing $\bdd{C_1}{C_2}$ or $\bdd{C_2}{C_1}$, then one can in polynomial time output an obstruction or output a partition tree of $G$ from $\calT$ confirming that $G$ is a well-partitioned chordal graph. 
		\end{nestedclaim}
		\begin{claimproof}
			First we consider the case in which only $\bdd{C_1}{C_2}$ has a crossing vertex.
			Let $a$ be a vertex in a bag $C_3 \in V(\calT) \setminus \{C_1, C_2\}$ that crosses $\bdd{C_1}{C_2}$. Let $b \in C_1 \setminus \bdd{C_1}{C_2}$ be a neighbor of $a$. Note that a neighbor of $a$ in $\bdd{C_1}{C_2}$ is either adjacent to $v$, as~$z_1$, or non-adjacent to~$v$, as $x$. 
			As in Claim~\ref{claim:no}, we can restrict the way $N(a)$ intersects $\{x,b,z_1\}$, and as we did in Claim~\ref{claim:base:crossing}, we can deduce that $\bdd{C_3}{C_1}$ is complete to $C_1$ 
			and that there is no bag other than $C_3$ containing a vertex that crosses $\bdd{C_1}{C_2}$.
			
			Observe that $\{v, z_1, z_2, x, a, b\}$ induces a $W_{2, 0}^-$ with terminal vertex $a$. 
			By applying Lemma~\ref{lem:extendgood} similarly in Case~\ref{case:big:2},
		one can in polynomial time find an obstruction or 
		find a maximal good boundary-crossing path ending in $C_3C_1C_2$.
		In the latter case, we apply Lemma~\ref{lem:shorten:path} to modify $\calT$ along this path and obtain a partition tree $\calT'$ for $G-v$ such that no vertex crosses $\bdd{C_1}{C_2}$. 
	Then, since both $\bdd{C_1}{C_2}$ and $\bdd{C_2}{C_1}$ have no crossing vertices, we can apply Lemma~\ref{lem:base:nocrossing} to obtain a partition tree for $G$.
			
			Now we consider the case in which only $\bdd{C_2}{C_1}$ has a crossing vertex. Let $a$ be a vertex in a bag 
			$C_3$ that crosses $\bdd{C_2}{C_1}$.
			Note that $\{v,z_1,z_2,x\}$ is a wing of $W_{1,0}$ with terminal $z_2$, as in Case~\ref{case:big:2} (see (b) of Figure~\ref{fig:case2-base}). 
			As in Claims~\ref{claim:base2:crossing} and Lemma~\ref{lem:extendgood}, we can find a maximal good boundary-crossing path ending in $C_2C_1$. 		We apply Lemma~\ref{lem:shorten:path} to modify $\calT$ along this path and obtain a partition tree $\calT'$ for $G-v$ such that no vertex crosses $\bdd{C_2}{C_1}$. 
	Then, since both $\bdd{C_1}{C_2}$ and $\bdd{C_2}{C_1}$ have no crossing vertices, we can apply Lemma~\ref{lem:base:nocrossing} to obtain a partition tree for $G$.
				 
			 To conclude, in the case in which both $\bdd{C_1}{C_2}$ and $\bdd{C_2}{C_1}$ have crossing vertices, we can first modify $\calT$ along a maximal boundary-crossing path ending in $C_2C_1$, then along a maximal boundary-crossing path ending in $C_1C_2$. In this way we obtain a partition tree for $G-v$ in which, again, both $\bdd{C_1}{C_2}$ and $\bdd{C_2}{C_1}$ have no crossing vertices and we proceed with Lemma~\ref{lem:base:nocrossing}.
		\end{claimproof}
	\end{case}

	This concludes the lemma.
\end{proof}

\begin{lemma}\label{lem:charcase3}
		If $C_1 \setminus N(v) \neq \emptyset$,  $C_2 \setminus N(v) \neq \emptyset$ and $N(v)=\bdd{C_1}{C_2}\cup \bdd{C_2}{C_1}$, 
	then one can in polynomial time either
	output an obstruction in $G$ or output a partition tree of $G$ confirming that $G$ is a well-partitioned chordal graph.
\end{lemma}
\begin{proof}	
	We first show that if at least one of $\bdd{C_1}{C_2}$ and $\bdd{C_2}{C_1}$ has no crossing vertex, then we can obtain a partition tree for $G$.
		\begin{nestedclaim}\label{claim:case:2-2}
			If there is no vertex crossing $\bdd{C_1}{C_2}$, then one can obtain a partition tree of $G$ from $\calT$ in polynomial time. The same holds for $\bdd{C_2}{C_1}$.
		\end{nestedclaim}
		\begin{claimproof}
			We prove the claim for $\bdd{C_1}{C_2}$ and note that the argument for $\bdd{C_2}{C_1}$ is symmetric. 
			Let $C_1' \defeq C_1 \setminus \bdd{C_1}{C_2}$, and $C_{12}' \defeq \bdd{C_1}{C_2} \cup \{v\}$.
			Let $\calS_1 \subseteq N_\calT(C_1)$ be such that for all $S_1 \in \calS_1$, $\bdd{C_1}{S_1} \subseteq C_1 \setminus \bdd{C_1}{C_2}$, 
			and let $\calS_2 \subseteq N_\calT(C_1)$ be such that for all $S_2 \in \calS_2$, $\bdd{C_1}{S_2} \subseteq \bdd{C_1}{C_2}$.
			We obtain a partition tree $\calT'$ for $G$ from $\calT$ as follows.
			\begin{itemize}
				\item[$\cdot$] Remove $C_1$; add $C_1'$ and $C_{12}'$; make $C_1'$ adjacent to $C_{12}'$, and $C_{12}'$ adjacent to $C_2$.
				\item[$\cdot$] Make each bag in $\calS_1$ adjacent to $C_1'$, and each bag in $\calS_2$ adjacent to $C_{12}'$.
			\end{itemize}
			This yields a partition tree for $G$.
		\end{claimproof}
		
		From now on, we assume that both $\bdd{C_1}{C_2}$ and $\bdd{C_2}{C_1}$ have crossing vertices. 
		Let $C_2'$ be a bag containing a vertex crossing $\bdd{C_1}{C_2}$, and 
		let $C_3$ be a bag containing a vertex crossing $\bdd{C_2}{C_1}$.
		For convenience, let $C_1'\defeq C_1$.
		
		Using Lemma~\ref{lem:extendgood-pre} with $B=\bdd{C_1'}{C_2}$, 
		we recursively find a longer good boundary-crossing path or a partial obstruction.
		Starting from $C_2'C_1'$, for a path $C_i'C_{i-1}' \ldots C_1'$, 
		we find a neighbor bag $C_{i+1}'$ of $C_i'$ that contains a vertex crossing $\bdd{C_i'}{C_{i-1}'}$.
		At the end, either we can find one of first two outcomes in Lemma~\ref{lem:extendgood-pre}, or 
		we can find a maximal good boundary-crossing path ending in $C_2'C_1'C_2$.
		In the latter case, we can repeatedly apply Lemma~\ref{lem:shorten:path} to modify $\calT$ along this path and obtain a partition tree $\calT'$ for $G-v$ such that no vertex crosses $\bdd{C_1'}{C_2'}$. We can now apply Claim~\ref{claim:case:2-2} to obtain a partition tree for the entire graph $G$.
		Thus, we may assume that we have an induced subgraph $H_1$ which is one of two outcomes in Lemma~\ref{lem:extendgood-pre}.
		Let $v_1$ be the terminal of $H_1$ in $\bdd{C_1'}{C_2}$.
		
		By applying the same argument for $C_2C_3$, 
		we may assume that we have an induced subgraph $H_2$ which is one of two outcomes in Lemma~\ref{lem:extendgood-pre}.
		Let $v_2$ be the terminal of $H_2$ in $\bdd{C_2}{C_1}$.
			
		If both $H_1$ and $H_2$ are the first outcome in Lemma~\ref{lem:extendgood-pre}, 
		then $G[V(H_1)\cup V(H_2)\cup \{v\}]$ is isomorphic to $W_{s,t}$ for some $s\in \{1,2,3\}$ and $t\ge 0$.
		If $H_1$ is the first outcome and $H_2$ is the second outcome of Lemma~\ref{lem:extendgood-pre}, 
		then $G[V(H_1)\cup V(H_2)\cup \{v\}]$ is isomorphic to $W_{s,t}$ for some $s\in \{2,3\}$ and $t\ge 0$, 
		where $G[V(H_2)\cup \{v, v_1\}]$ is isomorphic to $W_{2,0}^-$.
		If both are the second outcomes in Lemma~\ref{lem:extendgood-pre}, 
		then $G[V(H_1)\cup V(H_2)\cup \{v\}$ is isomorphic to $O_4$, and this concludes the lemma.
\end{proof}

\section{Algorithmic applications}\label{sec6}
In this section, we give several \FPT-algorithms and kernels for problems on well-partitioned chordal graphs.
Specifically, we consider variants of the \textsc{Disjoint Paths} problem, 
where each path additionally has to be from a predefined domain.
Before we proceed with the algorithmic description, we review the variants of the term `disjoint paths' 
that we use in this section. 
Let $P_1$ be an $(s_1, t_1)$-path and $P_2$ be an $(s_2, t_2)$-path.
The paths $P_1$ and $P_2$ being disjoint is most literally translated to $V(P_1) \cap V(P_2) = \emptyset$.
However, in some applications of disjoint paths problems, e.g.~\cite{patterns2014,Heg15}, 
notions of \emph{internally disjoint} paths are used, meaning that the intersection of 
$\{s_1, t_1\}$ and $\{s_2, t_2\}$ may be nonempty.
In this section, we deal with two variants of the latter definition and note that they in fact both
generalize the setting where we require $V(P_1) \cap V(P_2) = \emptyset$.
Specifically, in Section~\ref{sec:alg:srdp}, we consider the \SRDP{} problem, asking for \emph{internally vertex-disjoint} paths.
We say that $P_1$ and $P_2$ are internally vertex-disjoint, 
if for $i \in [2]$, $(V(P_i) \setminus \{s_i, t_i\}) \cap V(P_{3-i}) = \emptyset$,
meaning that no internal vertex of one path is used as a vertex on the other path.
However, if $\{s_1, t_1\} = \{s_2, t_2\}$ and $s_1 t_1 \in E(G)$, then according to this definition, 
the edge $s_1t_1$ can be used both as the path $P_1$ and as the path $P_2$ in a solution without violating the definition.
In Section~\ref{sec:alg:srtdp}, we study the \SRTDP{} problem that asks for internally vertex-disjoint paths that are also distinct.
In Section~\ref{sec:alg:srdcs}, we sketch how to use the same methods to solve 
the related \SRDCS{} problem on well-partitioned chordal graphs.

Moreover, in Section~\ref{sec:kernels:split}, we observe that the algorithms given in Sections~\ref{sec:alg:srdp} and~\ref{sec:alg:srtdp}
imply polynomial kernels for the \SRDP{} and \SRTDP{} problems on split graphs.
In Section~\ref{sec:kernels:wpc}, we show that with two more simple reduction rules regarding degree-two bags of partition trees
we can obtain polynomial kernels for \probDP{} and \probTDP{} on well-partitioned chordal graphs.

\subsection{Set-Restricted Disjoint Paths}\label{sec:alg:srdp}
In this section we deal with the following parameterized problem, and show that it is in \FPT{} 
on well-partitioned chordal graphs.

\fancyparproblemdef
	{Set-Restricted Disjoint Paths}
	{A graph $G$, a set $\terminals = \{(s_1, t_1), \ldots, (s_k, t_k)\}$ of 
	$k$ pairs of vertices of $G$, called \emph{terminals}, 
	a set $\mathcal{U} = \{U_1, \ldots, U_k\}$ of $k$ vertex subsets of $G$, called \emph{domains}.}
	{$k$}
	{Does $G$ contain $k$ pairwise internally vertex-disjoint paths $P_1, \ldots, P_k$ such that 
			for $i \in [k]$, $P_i$ is an $(s_i, t_i)$-path with $V(P_i) \subseteq U_i$?
	}

We let $V(\terminals) \defeq \bigcup_{i \in [k]} \{s_i, t_i\}$.
For a set of vertices $S \subseteq V(G)$, let $\labeling \colon S \to [k]$ be a labeling.
We say that $\labeling$ is \emph{domain-preserving}, if for each $v \in S$, $v \in U_{\labeling(v)}$.
We use domain-preserving labelings in our algorithm later 
to gradually build paths that only use vertices of the prescribed domains.

\begin{remark}\label{remark:terminals:domains}
	Let $(G, k, \terminals, \domains)$ be an instance of \SRDP{}.
	For ease of exposition, we make the following assumptions.
	First, we assume that $G$ is connected, since otherwise, 
	we can simply solve the problem on each connected component separately.
	Furthermore, for each $i \in [k]$, we assume that 
	$\{s_i, t_i\}\subseteq U_i$ and $V(\terminals) \cap (U_i\setminus \{s_i, t_i\})=\emptyset$.
	This way, we ensure directly that no path $P_i$ can use a terminal $s_j$ or $t_j$ ($i \neq j$)
	as an internal vertex.
\end{remark}

Suppose that for some terminal pair $(s_i, t_i)$, we have that $s_i t_i \in E(G)$.
Then, we can use the edge $s_i t_i$ as a path in a solution.
Since such a path has no internal vertex, and the \SRDP{} problem asks for
\emph{internally} vertex-disjoint paths, using this edge as the $(s_i, t_i)$-path cannot
create any conflict with any other path. Therefore, we can safely remove the terminal 
pair $(s_i, t_i)$ from the instance without changing the answer to the problem.
\begin{reduction}\label{red:srdp:terminal:edge}
	Let $(G, k, \terminals, \domains)$ be an instance of \textsc{Set-Restricted Disjoint Paths}
	such that for some $i \in [k]$, $s_i t_i \in E(G)$. Then, reduce this instance to 
	$(G, k-1, \terminals \setminus \{(s_i, t_i)\}, \domains \setminus \{U_i\})$.
\end{reduction}

Next, we observe that finding pairwise internally vertex-disjoint paths is equivalent to finding
pairwise internally vertex-disjoint \emph{induced} paths.
We call a solution consisting of induced paths a \emph{minimal} solution.
The following observation is an immediate consequence by the fact
that each bag of a partition tree is a clique in the underlying graph.
\begin{observation}\label{obs:srdp:intersection:bag}
	Let $(G, k, \terminals, \domains)$ be an instance of \textsc{Set-Restricted Disjoint Paths}
	such that $G$ is a connected well-partitioned chordal graph.
 	Let $P$ be a path in a minimal solution to $(G, k, \terminals, \domains)$, and let $B$ be a bag 
	of the partition tree of $G$. Then, $\card{V(P) \cap B} \le 2$.
\end{observation}

Throughout the following, let $(G, k, \terminals, \domains)$ 
be an instance of \SRDP{},
where $G$ is a connected well-partitioned chordal graph that is given 
together with a partition tree $\calT$ of $G$.
Based on Observation~\ref{obs:srdp:intersection:bag}, we now describe a marking procedure that 
marks at most $4k^2$ vertices of each bag $B \in V(\calT)$ such that if there is a solution
to $(G, k, \terminals, \domains)$, then there is a solution whose paths use only marked vertices as internal vertices.

\begin{lemma}\label{lem:srdp:marking}
	Let $(G, k, \terminals, \domains)$ be an instance of \textsc{Set-Restricted Disjoint Paths},
	such that $G$ is a well-partitioned chordal graph with partition tree $\calT$,
	after exhaustive application of Reduction~\ref{red:srdp:terminal:edge}.
	Then, there is a $\calO(k^2 \cdot n)$ time algorithm that computes sets $M_1, \ldots, M_k \subseteq V(G)$
	such that for all $i \in [k]$, $M_i \subseteq U_i$, and the following hold.
	\begin{itemize}
		\item For each bag $B \in V(\calT)$, and each $i \in [k]$, $\card{B \cap M_i} \le 4k$.
		\item Let $\calT'$ be the forest in $\calT$ induced by all bags with a nonempty intersection
			with $\bigcup_{i \in [k]} M_i$.
			Then, $\calT'$ has at most $2k$ bags of degree one.
		\item $(G, k, \terminals, \domains)$ is a \yes-instance if and only if there exists 
		a minimal solution $(P_1, \ldots, P_k)$ such that for each $i \in [k]$, $V(P_i) \subseteq M_i \cup \{s_i, t_i\}$.
	\end{itemize}
\end{lemma}
\begin{proof}
	We initialize $M_i \defeq \emptyset$ for all $i \in [k]$.
	For each $i \in [k]$, we do the following.
	Let $B_1 B_2 \ldots B_\ell$ be the path in $\calT$ such that $s_i \in B_1$ and $t_i \in B_\ell$.
	Then, for each $j \in [\ell-1]$, we add to $M_i$ a maximal subset of $U_i \cap \bdd{B_j}{B_{j+1}}$ of size at most $2k$,
	and a maximal subset of $U_i \cap \bdd{B_{j+1}}{B_j}$ of size at most $2k$. This finishes the construction of the sets $M_i$,
	and it is not difficult to see that it can be implemented to run in time $\calO(k^2\cdot n)$.
	
	We prove that $M$ has the claimed properties. 
	The first item is immediate.
	The second item follows from the observation that
	$\calT'$ consists of the union of $k$ paths in $\calT$.
	For the third item, suppose $(G, k, \terminals, \domains)$ is a \yes-instance, 
	and let $(P_1, \ldots, P_k)$ be a minimal solution.
	Suppose that for some bag $B \in V(\calT)$ and some $i \in [k]$, $V(P_i) \cap B \not\subseteq M_i \cup \{s_i, t_i\}$.
	Let $B_1 B_2 \ldots B_\ell$ be the path in $\calT$ from the bag containing $s_i$ to the bag containing $t_i$.
	We argue that we may assume that $B$ is a bag on this path.
	Suppose not, and let $j \in [\ell]$ be such that $B_j$ is the closest bag to $B$ on the path,
	and let $B_j'$ be the neighboring bag of $B_j$ on the path from $B_j$ to $B$.
	Then, $P_i$ must use two vertices from $\bdd{B_j}{B_j'}$, and hence there is a triangle in $P_i$,
	a contradiction with $P_i$ being a path of a minimal solution.
	
	From now on, let $j \in [\ell]$ be such that $B = B_j$, and let $Y = (V(P_i) \cap B_j) \setminus \{s_i, t_i\}$.
	First, if $j = 1$, then we may assume that there is only one vertex in $y \in Y \setminus M_i$, with $y \in \bdd{B_1}{B_2}$, 
	and that $s_i \notin \bdd{B_1}{B_2}$ (otherwise $P_i$ was not an induced path). 
	Since $y$ was not marked, there are $2k$ marked vertices in $\bdd{B_1}{B_2} \cap M_i$.
	Since the paths $P_j$, $j \neq i$, use at most $2(k-1)$ vertices in total from $B$ by Observation~\ref{obs:srdp:intersection:bag}, 
	there is at least one vertex in $\bdd{B_1}{B_2} \cap M_i$ that is not used by any other path, call that vertex $y'$. 
	We replace $y$ by $y'$ in $P_i$, and maintain the property that $P_i$ is an induced $(s_i, t_i)$-path, since $y$ and $y'$
	are twins in $G[B_1 \cup B_2]$. A similar argument can be given if $j = \ell$.
	
	Now suppose that $1 < j < \ell$.
	Again we have that $\card{Y} \le 2$.
	Suppose that $\card{Y \setminus M_i} = 2$, and let $\{y_1, y_2\} = Y \setminus M_i$.
	Assume wlog that $y_1 \in \bdd{B_j}{B_{j-1}}$ and that $y_2 \in \bdd{B_j}{B_{j+1}}$.
	By the same argument as above, there are vertices $y_1' \in \bdd{B_j}{B_{j-1}} \cap M_i$
	and $y_2' \in \bdd{B_j}{B_{j+1}} \cap M_i$ such that neither $y_1'$ nor $y_2'$
	are used by any other path in the solution.
	We can replace $\{y_1, y_2\}$ by $\{y_1', y_2'\}$ in $P_i$, and $P_i$ remains an $(s_i, t_i)$-path.
	If for $r \in [2]$, $y_r' \in \bdd{B_j}{B_{j-1}} \cap \bdd{B_j}{B_{j+1}}$, then we remove $y_{3-r}'$
	from $P_i$ to ensure that $P_i$ remains an induced path.
	The last case, when $\card{Y \setminus M_i} = 1$, can be treated with similar arguments and we therefore 
	skip the details.
	
	We have shown how to modify the paths $P_1, \ldots, P_k$ such that they remain induced paths, 
	and for all $i \in [k]$, $V(P_i) \subseteq M_i \cup \{s_i, t_i\}$, so the third item follows.
\end{proof}

Given any bag $B$, we can immediately observe which paths of a solution need to use some vertices
from $B$ as internal vertices. 
The next definition captures the property of a bag having a set of vertices that can be used 
as the internal vertices of all paths that need to go through $B$.
Note that in the algorithm of this section, we only need the special case of $[k]$-feasible bags;
however in the algorithm for \SRTDP{}, we need to be able to restrict this definition to a subset of $[k]$
which is why we give it in a more general form here.

\begin{definition}[$I$-Feasible Bag]\label{def:feasible}
	Let $I \subseteq [k]$.
	Let $B \in V(\calT)$ be a bag and $M_i \subseteq V(G)$, $i \in I$, be sets of vertices.
	Then, we say that $B$ is \emph{$I$-feasible w.r.t. $\{M_i \mid i \in I\}$},
	if there is a set $X \subseteq B$ and a labeling $\labeling \colon X \to [k]$ such that the following hold.
	For each $i \in I$ such that $B$ lies on the path 
	from the bag containing $s_i$ to the bag containing $t_i$ in $\calT$,
	and each neighbor $C$ of $B$ on that path,
	either $\{s_i, t_i\} \cap \bdd{B}{C} \neq \emptyset$,
	or there is a vertex $x_i \in X \cap M_i \cap \bdd{B}{C}$ such that $\labeling(x_i) = i$.
	%
	%
	We use the shorthand `feasible' for `$[k]$-feasible'.
\end{definition}

\begin{lemma}\label{lem:srdp:feasible}
	Let $(G, k, \terminals, \domains)$ be an instance of \textsc{Set-Restricted Disjoint Paths},
	such that $G$ is a well-partitioned chordal graph with partition tree $\calT$,
	after exhaustive application of Reduction~\ref{red:srdp:terminal:edge}.
	Let $M_1, \ldots, M_k$ be sets of vertices given by Lemma~\ref{lem:srdp:marking}.
	Then, $(G, k, \terminals, \domains)$ is a \yes-instance if and only if each bag of $\calT$
	is feasible w.r.t.\ $M_1, \ldots, M_k$.
\end{lemma}
\begin{proof}
	Suppose that $(G, k, \terminals, \domains)$ is a \yes-instance.
	By Lemma~\ref{lem:srdp:marking}, there is a minimal solution $(P_1, \ldots, P_k)$ 
	such that for all $i \in [k]$, $V(P_i) \subseteq M_i \cup \{s_i, t_i\}$.
	Let $B \in V(\calT)$ be a bag. Then, we let 
	$X \defeq (B \setminus V(\terminals)) \cap \bigcup_{i \in [k]} V(P_i)$,
	and $\labeling \colon X \to [k]$ be such that for all $x \in X$, 
	$\labeling(x) = i$ if $x \in V(P_i)$.
	Then, it is not difficult to see that $X$ and $\labeling$ show that $B$ is a feasible bag.
	
	For the other direction, suppose that each bag of $\calT$ is feasible w.r.t.\ $M_1, \ldots, M_k$.
	Let $i \in [k]$, and denote by $B_1 B_2 \ldots B_\ell$ the path in $\calT$ 
	such that $s_i \in B_1$ and $t_i \in B_\ell$.
	Then, for $j \in [\ell]$, let $X_j$ and $\labeling_j$ be the subset of $B_j$ and its labeling, respectively,
	that show that $B_j$ is feasible.
	Then, for each $j \in [\ell-1]$, there is a vertex $x_j \in \bdd{B_j}{B_{j+1}}$ such that $\labeling_j(x_j) = i$,
	and a vertex $y_{j+1} \in \bdd{B_{j+1}}{B_j}$ such that $\labeling_{j+1}(y_{j+1}) = i$.
	Then the sequence 
	$$s_i, x_1, y_2, x_2, y_3, x_3, \ldots, y_{\ell-1}, x_{\ell-1}, y_\ell, t_i$$
	(where it may happen that $s_i = x_1$ or $y_\ell = t_i$ or for some $j \in [\ell]$, $x_j = y_j$)
	can be used to obtain an $(s_i, t_i)$-path in $G$, and we can take any induced subpath of it to obtain an induced 
	$(s_i, t_i)$-path $P_i$ in $G$. 
	Since all labelings $\labeling_j$ are domain-preserving, we have that 
	$V(P_i) \subseteq U_i \cup \{s_i, t_i\}$.
\end{proof}

\begin{theorem}\label{thm:srdp}
	There is an algorithm that solves each instance $(G, k, \terminals, \domains)$ of \SRDP{}
	where $G$ is a well-partitioned chordal graph given along with a partition tree $\calT$,
	in time $2^{\calO(k\log k)} \cdot n$.
\end{theorem}
\begin{proof}
	The algorithm works as follows.
	\begin{description}
		\item[Step 1.] Apply Reduction~\ref{red:srdp:terminal:edge} exhaustively.
		\item[Step 2.] Use Lemma~\ref{lem:srdp:marking} to obtain the sets $M_1, \ldots, M_k$ of marked vertices.
		\item[Step 3.] For each bag $B \in V(\calT)$, check if $B$ is feasible w.r.t.\ $M_1, \ldots, M_k$. 
			If all bags are feasible w.r.t.\ $M_1, \ldots, M_k$, then report \yes, and if not, report \no.
	\end{description}
	Correctness follows from Lemma~\ref{lem:srdp:feasible}.
	For the runtime, we have that the marking in Step 2 can be done in time $\calO(k^2 \cdot n)$ by Lemma~\ref{lem:srdp:marking}.
	Let $\calM \defeq \bigcup_{i \in [k]} M_i$.
	Since for each bag $B \in V(\calT)$, $\card{B \cap \calM} = \calO(k^2)$ by Lemma~\ref{lem:srdp:marking},
	we can try all $2^{\calO(k \log k)}$ sets $X$ and all $2^{\calO(k \log k)}$ labelings of each such $X$
	to check if $B$ is feasible. Therefore, the feasibility of $B$ can be checked in time $2^{\calO(k \log k)}$.
	Since there are at most $n$ bags in $\calT$, we have that the total runtime of verifying feasibility is $2^{\calO(k \log k)} \cdot n$.
\end{proof}

\subsection{Set-Restricted Totally Disjoint Paths}\label{sec:alg:srtdp}
As discussed above, we now adapt the algorithm of Theorem~\ref{thm:srdp} to find 
\emph{totally internally disjoint paths} instead, as it is done e.g.\ in the setting of finding 
topological minors in a graph~\cite{Heg15}.
Specifically, we are dealing with the following problem.

\fancyparproblemdef
	{Set-Restricted Totally Disjoint Paths}
	{A graph $G$, a set $\terminals = \{(s_1, t_1), \ldots, (s_k, t_k)\}$ of 
	$k$ pairs of vertices of $G$, called \emph{terminals}, 
	a set $\mathcal{U} = \{U_1, \ldots, U_k\}$ of $k$ vertex subsets of $G$, called \emph{domains}.}
	{$k$}
	{Does $G$ contain $k$ pairwise distinct and internally vertex-disjoint paths $P_1, \ldots, P_k$ such that 
			for $i \in [k]$, $P_i$ is an $(s_i, t_i)$-path with $V(P_i) \subseteq U_i$?
	}

Recall that $V(\terminals) \defeq \bigcup_{i \in [k]} \{s_i, t_i\}$, and that
for a set of vertices $S \subseteq V(G)$, a labeling $\labeling \colon S \to [k]$ is called
\emph{domain-preserving}, if for each $v \in S$, $v \in U_{\labeling(v)}$.
Furthermore, as in Remark~\ref{remark:terminals:domains}, we assume that in any instance we consider,
the input graph $G$ is connected and
for all $i \in [k]$, $\{s_i, t_i\} \subseteq U_i$ and $V(\terminals) \cap (U_i \setminus \{s_i, t_i\}) = \emptyset$.

Following the notation introduced in~\cite{Heg15}, we call an edge $xy \in E(G)$ \emph{heavy} 
if for some $w \ge 2$, there are pairwise distinct indices $i_1, \ldots, i_w$ 
such that for each $j \in [w]$, $\{x, y\} = \{s_{i_j}, t_{i_j}\}$. In that case, we call $w$ 
the \emph{weight} of the edge $xy$, and we say that the indices $i_1, \ldots, i_w$
\emph{weigh down on $xy$}. For each such heavy edge $e$, we use the notation $I(e) \defeq \{i_1, \ldots, i_w\}$.
We say that paths $P_{i_1}, \ldots, P_{i_w}$ \emph{satisfy} $e$ if there is precisely
one $a \in [w]$ such that $P_{i_a}$ consists of the edge $e$,
and for each $b \in [w] \setminus \{a\}$, $P_{i_b}$ is a length-$2$ $(x, y)$-path.
Furthermore, we call an index $i$ a \emph{heavy index} if 
$s_i t_i$ is a heavy edge. We say that an index is \emph{light} if it is not heavy.

First, we observe that if for some edge $xy \in E(G)$, there is \emph{precisely one} $i$
such that $\{x, y\} = \{s_i, t_i\}$, then we can again remove the terminal pair $(s_i, t_i)$
from the instance without changing the answer to the problem: We can always use the edge $xy$ 
as the path connecting $s_i$ to $t_i$. 
(Note that the following is a weaker form of Reduction~\ref{red:srdp:terminal:edge}, 
and that it does not apply to heavy edges.)
\begin{reduction}\label{red:srtdp:terminal:edge}
	Let $(G, k, \terminals, \domains)$ be an instance of 
	\textsc{Set-Restricted Totally Disjoint Paths}
	such that for some edge $xy \in E(G)$ there is precisely one $i \in [k]$
	such that $\{x, y\} = \{s_i, t_i\}$. 
	Then, reduce this instance to 
	$(G, k-1, \terminals \setminus \{(s_i, t_i)\}, \domains \setminus \{U_i\})$.
\end{reduction}

Due to Reduction~\ref{red:srtdp:terminal:edge}, we may from now on assume that for each 
$i \in [k]$, either $s_i t_i \notin E(G)$, or $s_it_i$ is a heavy edge. The former we can
handle as in the algorithm of Theorem~\ref{thm:srdp}, and we explain how to deal with the 
latter. 
The existence of a heavy edge rules out the approach of looking for \emph{minimal} 
solutions: if $xy$ is a heavy edge of weight $w$, then each solution contains $w-1$
paths between $x$ and $y$ that are not induced due to the existence of the edge $xy$.

However, a slight modification of this approach works.
We base our marking procedure on the existence of \emph{minimum} solutions,
i.e. solutions $(P_1, \ldots, P_k)$ such that there is no other solution
that contains fewer edges: for all solutions $(P_1', \ldots, P_k')$, it holds that
$\sum_{i \in [k]} \card{E(P_i)} \le \sum_{i \in [k]} \card{E(P_i')}$.
As shown in the following lemma due to Heggernes et al.~\cite{Heg15}, 
minimum solutions have a very restricted structure
in chordal graphs as well.\footnote{Note that Lemma 2 in~\cite{Heg15} is proved 
for the \textsc{Disjoint Paths} problem; however, the proof goes through for
\textsc{Set-Restricted Totally Disjoint Paths} as well.}
\begin{lemma}[Cf.~Lemma 2 in~\cite{Heg15}]\label{lem:srtdp:minimum}
	Let $(G, k, \terminals, \domains)$ be an instance of \textsc{Set-Restricted Disjoint Paths}, such that $G$
	is a chordal graph. If $(P_1, \ldots, P_k)$ is a minimum solution to $(G, k, \terminals, \domains)$, 
	then for each $i \in [k]$, precisely one of the following holds.
	\begin{itemize}
		\item $P_i$ is an induced path;
		\item $P_i$ is a path of length two, and there exists some $j \in [k]$ such that 
			$P_j$ is of length one and $\{s_i, t_i\} = \{s_j, t_j\}$.
	\end{itemize}
\end{lemma}

By the second part of the previous lemma we know that in each minimum solution,
for each heavy edge $e$, there is a collection of paths of length at most $2$ satisfying $e$.
Hence, to accommodate the heavy edges in our marking scheme, it is enough 
to consider common neighbors of the endpoints of heavy edges.
If our input graph $G$ is a well-partitioned chordal graph, this gives us a lot 
of structure we can exploit.

We adapt the marking procedure of \SRDP{} algorithm as follows.
\begin{lemma}\label{lem:srtdp:marking}
	Let $(G, k, \terminals, \domains)$ be an instance of 
	\SRTDP{},
	such that $G$ is a well-partitioned chordal graph with partition tree $\calT$,
	after exhaustive application of Reduction~\ref{red:srtdp:terminal:edge}.
	Then, there is an $\calO(k^2 \cdot n)$ time algorithm that computes sets $M_1, \ldots, M_k \subseteq V(G)$
	such that for all $i \in [k]$, $M_i \subseteq U_i$, and the following hold. Let $I$ denote the light indices.
	\begin{itemize}
		\item For each heavy index $i \in [k] \setminus I$, $\card{M_i} \le 2k$.
		\item For each bag $B \in V(\calT)$, and each light index $i \in I$, $\card{B \cap M_i} \le 4k$.
		\item Let $\calT'$ be the forest in $\calT$ induced by all bags with a nonempty intersection
			with $\bigcup_{i \in I} M_i$.
			Then, $\calT'$ has at most $2\cdot \card{I}$ bags of degree one.
		\item $(G, k, \terminals, \domains)$ is a \yes-instance if and only if there exists 
		a minimal solution $(P_1, \ldots, P_k)$ such that for each $i \in [k]$, $V(P_i) \subseteq M_i \cup \{s_i, t_i\}$.
	\end{itemize}
\end{lemma}
\begin{proof}
	For each light index, we proceed as in Lemma~\ref{lem:srdp:marking}.
	Let $xy$ be a heavy edge. 
	There are two cases we need to consider.
	First, if there are bags $B_1, B_2 \in V(\calT)$ such that 
	$x \in \bdd{B_1}{B_2}$ and $y \in \bdd{B_2}{B_1}$, then the only vertices that
	can be used to form a length-$2$ $(x, y)$-path are in 
	$\bdd{B_1}{B_2} \cup \bdd{B_2}{B_1}$.
	Hence, for each index $i$ weighing down on $xy$,
	we let $M_i$ be a maximal subset of $U_i \cap (\bdd{B_1}{B_2} \cup \bdd{B_2}{B_1})$ of size at most $2k$.
	Second, there is a bag $B$ such that $\{x, y\} \subseteq B$.
	Then, the remaining vertex on a path corresponding to the terminal pair $(x, y)$ 
	has to be contained in $B$ or in a neighbor 
	bag $C$ of $B$ such that $x$ and $y$ are both in the boundary of $B$ to $C$.
	Therefore, we let $M_i$ be a maximal subset of 
	$$U_i \cap \left(B \cup \bigcup\nolimits_{C \in N_\calT(B), \{x, y\} \subseteq \bdd{B}{C}} \bdd{C}{B}\right)$$
	of size at most $2k$.
	We can argue in the same way as in the proof of Lemma~\ref{lem:srdp:marking} that this is correct.
\end{proof}

\begin{theorem}\label{thm:srtdp}
	There is an algorithm that solves each instance $(G, k, \terminals, \domains)$ of \SRTDP{}
	where $G$ is a well-partitioned chordal graph given along with a partition tree $\calT$,
	in time $2^{\calO(k\log k)} \cdot n$.
\end{theorem}
\begin{proof}
	Let $I \subseteq [k]$ be the light indices and $J \subseteq [k]$ the heavy indices.
	The algorithm works almost the same way as the one for \SRDP{}, with an additional guessing stage of the
	vertices satisfying the heavy edges.
	\begin{description}
		\item[Step 1.] Apply Reduction~\ref{red:srtdp:terminal:edge} exhaustively.
		\item[Step 2.] Use Lemma~\ref{lem:srtdp:marking} to obtain the sets $M_1, \ldots, M_k$ of marked vertices.
		\item[Step 3.] For each subset $Y \subseteq \bigcup_{j \in J} M_j$ of size at most $\card{J}$
			check if the vertices in $Y$ can be used to satisfy all heavy edges.
		\item[Step 4.] If Step 3 succeeded for the set $Y$, then for all $i \in I$, let $M_i' \defeq M_i \setminus Y$, and
			check for each bag $B \in V(\calT)$ if it is $I$-feasible w.r.t.~$\{M_i' \mid i \in I\}$. 
	\end{description}
	Correctness of this algorithm follows in a similar way as in Theorem~\ref{thm:srdp}.
	By Lemma~\ref{lem:srtdp:marking}, if $(G, k, \terminals, \domains)$ is a \yes-instance,
	then there is a solution only using vertices from $M_1, \ldots, M_k$. 
	Therefore, if the instance has a solution, then one of the guesses in Step 3 must succeed,
	in such a way that all bags are $I$-feasible with respect to $\{M_i' \mid i \in I\}$.
	(Recall that $I$ are the light indices and that they can be handled in the same way as in the algorithm of Theorem~\ref{thm:srdp})
	
	By Lemma~\ref{lem:srtdp:marking}, $\card{\cup_{j \in J} M_j} = \calO(k^2)$, 
	so there are $2^{\calO(k \log k)}$ choices to try in Step 3.
	The runtime for Step 4 is again $2^{\calO(k \log k)} \cdot n$,
	therefore the runtime is $2^{\calO(k \log k)} \cdot n$.
\end{proof}

\newcommand\terminalsets{\calS}
\subsection{Set-Restricted Disjoint Connected Subgraphs}\label{sec:alg:srdcs}
In this section we sketch how to solve another related problem, \SRDCS{} that was also
introduced in~\cite{patterns2014}.
For convenience, we now require the solution to be pairwise \emph{vertex-disjoint}.
However, the variants of the problem that ask for internally vertex-disjoint solutions 
can be handled by similar methods used in the previous sections.

\fancyparproblemdef
	{Set-Restricted Disjoint Connected Subgraphs}
	{A graph $G$, a set $\terminalsets = \{S_1, \ldots, S_k\}$ of pairwise disjoint vertex sets $G$ called \emph{terminal sets},
	a set $\mathcal{U} = \{U_1, \ldots, U_k\}$ of $k$ vertex subsets of $G$, called \emph{domains}.}
	{$s \defeq \sum_{i \in [k]}\card{S_i}$}
	{Does $G$ contain $k$ pairwise vertex-disjoint connected subgraphs $F_1, \ldots, F_k$ such that 
			for $i \in [k]$, $S_i \subseteq V(F_i) \subseteq U_i$?
	}
	
Let $(G, k, \terminalsets, \domains)$ be an instance of \SRDCS{} such that $G$ is a well-partitioned chordal graph
with partition tree $\calT$.
Similarly to before, it is not difficult to see that if $(G, k, \terminalsets, \domains)$ is a \yes-instance, then
it has a solution that uses at most $s$ non-terminal vertices from each bag $B$. 
To adapt the marking procedure for this problem, for each $i \in [k]$, 
let $\calT_i$ be the smallest subtree of $\calT$ that contains all
bags that have a non-empty intersection with $S_i$.
For each bag $B \in V(\calT_i)$, and each $C \in N_{\calT_i}(B)$, 
we mark a maximal subset of $\bdd{B}{C}$ of size at most $s$.
Let the marked vertices for $i$, $M_i$, be the union of all these sets.

Then, if $(G, k, \terminalsets, \domains)$ is a \yes-instance, 
there is a solution such that each $F_i$ uses only vertices from $M_i$.
We can now define a notion of feasibility based on the subtrees $\calT_i$, in analogy with Definition~\ref{def:feasible}.
Then, checking if each bag is feasible again solves the problem.
Since each bag contains at most $\calO(s^2)$ marked vertices, we obtain the following runtime bound.
\begin{theorem}
	There is an algorithm that solves each instance $(G, k, \terminals, \domains)$ of \SRDCS{}
	where $G$ is a well-partitioned chordal graph given along with a partition tree $\calT$,
	in time $2^{\calO(s\log s)} \cdot n$.
\end{theorem}

\subsection{Polynomial kernels on split graphs}\label{sec:kernels:split}
We observe that Lemmas~\ref{lem:srdp:marking} and~\ref{lem:srtdp:marking} imply polynomial
kernels of \SRDP{} and \SRTDP{} on split graphs.
\begin{corollary}\label{cor:srdp:split}
	The \SRDP{} and \SRTDP{} problems on split graphs admit kernels on $\cO(k^2)$ vertices.
\end{corollary}
\begin{proof}
	Let $(G, k, \terminals, \domains)$ be an instance of \SRDP{} such that $G$ is a split graph,
	with partition tree $\calT$. Note that $\calT$ is a star, and assume that $\calT$ is rooted at 
	the center. Then, all leaves of this star contain only one vertex.
	Moreover, by Lemma~\ref{lem:srdp:marking}, we can reduce the number of vertices in the center 
	bag to $\cO(k^2)$, and we can remove all but $2k$ leaves, without changing the answer to the instance.
	
	Now suppose that $(G, k, \terminals, \domains)$ is an instance of \SRTDP{}, and let $\calT$ be as above.
	By Lemma~\ref{lem:srtdp:marking}, we can again reduce the number of vertices in the center bag
	to $\calO(k^2)$, and remove all but $\cO(k^2)$ leaves, without changing the answer to the instance.
\end{proof}

\subsection{Polynomial kernel for \probDP{} on well-partitioned chordal graphs}\label{sec:kernels:wpc}
We first study the variant of the problem \probDP{} as defined in the introduction, i.e.\ the problem that asks for 
pairwise \emph{internally} vertex-disjoint paths between the terminals.
Later, we consider a polynomial kernel for \probTDP{}, which asks for 
pairwise distinct and internally vertex-disjoint paths.

Towards such a kernel, let $(G, k, \terminals)$ be an instance of \probDP{} such that 
$G$ is a well-partitioned chordal graph with partition tree $\calT$, 
and let $M_1, \ldots, M_k \subseteq V(G)$ be the set of marked vertices due to 
Lemma~\ref{lem:srdp:marking}.\footnote{Where for each $i \in [k]$, we let $U_i = V(G)$.}
Let $\calT'$ be the forest in $\calT$ given by Lemma~\ref{lem:srdp:marking}.
If $(G, k, \terminals)$ is a \yes-instance, then there is a solution only using vertices 
that are contained in bags of $\calT'$.
We can therefore remove all vertices in bags $V(\calT) \setminus V(\calT')$, and continue working 
with the subgraph of $G$ induced by the bags in $\calT'$.
\begin{observation}\label{obs:kernel:dp:leaves}
	Let $(G, k, \terminals)$ be an instance of \probDP{} such that each connected component
	of $G$ is a well-partitioned chordal graph. 
	We may assume that the forest $\calF$, consisting of the partition trees of the connected components of $G$,
	has at most $2k$ bags of degree one.
\end{observation}

For convenience, we call $\calF$ as in the previous definition a \emph{partition forest}.
Suppose from now on that $(G, k, \terminals)$ is as asserted by Observation~\ref{obs:kernel:dp:leaves}.
Unless the number of vertices in $G$ is polynomial in $k$, $\calF$ contains many bags that are of degree two, 
and contain no terminal. 
We now introduce two new reduction rules that show that either we can conclude that
we are dealing with a \no-instance, or we can remove them.
\begin{reduction}\label{red:small:boundary}
	Let $(G, k, \terminals)$ be an instance of \probDP{} such that 
	the components of $G$ are well-partitioned chordal with partition forest $\calF$.
	Let $B \in V(\calT)$ such that $B \cap V(\terminals) = \emptyset$ and $\deg_\calF(B) = 2$. 
	Let $A$ and $C$ be the two neighbors of $B$ in $\calF$.
	Let $I \subseteq [k]$ be the indices such that $B$ lies on the path of the bag containing $s_i$ to the bag containing $t_i$.
	If $\card{\bdd{B}{A}} < \card{I}$ or $\card{\bdd{B}{C}} < \card{I}$,
	then reduce $(G, k, \terminals)$ to a trivial \no-instance.
\end{reduction}

\begin{lemma}\label{lem:small:boundary}
	If the conditions of Reduction~\ref{red:small:boundary} are satisfied, then $(G, k, \terminals)$ is a \no-instance.
\end{lemma}
\begin{proof}
	For each $i \in I$, the intersection of the vertices of the path $P_i$ in a solution with $B$ must be nonempty.
	Moreover, $P_i$ must use a vertex from both $\bdd{B}{C}$ and $\bdd{B}{A}$ (which could be the same vertex).
	However, by the pigeonhole principle, there cannot be a solution to $(G, k, \terminals)$ 
	since at least one of $\bdd{B}{A}$ and $\bdd{B}{C}$ has less than $\card{I}$ vertices.
\end{proof}

We now show the orthogonal to Reduction~\ref{red:small:boundary}: if both boundaries of $B$ have at least $\card{I}$ vertices, then we 
can remove $B$ and all its vertices.
Note that the statement of the following reduction also tells us how to obtain a partition tree of the reduced graph.
\begin{reduction}\label{red:large:boundary}
	Let $(G, k, \terminals)$ be an instance of \probDP{} such that 
	the components of $G$ are well-partitioned chordal with partition forest $\calF$.
	Let $B \in V(\calT)$ such that $B \cap V(\terminals) = \emptyset$ and $\deg_\calF(B) = 2$. 
	Let $A$ and $C$ be the two neighbors of $B$ in $\calF$.
	Let $I \subseteq [k]$ be the indices such that $B$ lies on the path of the bag containing $s_i$ to the bag containing $t_i$.
	If $\card{\bdd{B}{A}} \ge \card{I}$ and $\card{\bdd{B}{C}} \ge \card{I}$, then reduce $(G, k, \terminals)$
	to $(G', k, \terminals)$, where $G'$ is obtained from $G$ by
	removing $B$ and making all vertices in $\bdd{A}{B}$ adjacent to all vertices in $\bdd{C}{B}$.
\end{reduction}

\begin{lemma}\label{lem:large:boundary}
	Reduction~\ref{red:large:boundary} is safe, i.e.\ under its stated conditions, $(G, k, \terminals)$ is a \yes-instance
	if and only if $(G', k, \terminals)$ is a \yes-instance.
\end{lemma}
\begin{proof}
	Suppose $(G, k, \terminals)$ is a \yes-instance and let $P_1, \ldots, P_k$ be one of its solutions.
	For each $i \in I$, we observe that $P_i$ contains a vertex from $\bdd{A}{B}$, say $a_i$, 
	and a vertex from $\bdd{C}{B}$, say $c_i$, 
	such that both $a_i$ and $c_i$ have a neighbor from $B$ on the path $P_i$.
	Let $P_i'$ be the path obtained from $P_i$ by removing the vertices $V(P_i) \cap B$, 
	and making $a_i$ and $c_i$ adjacent. For each $j \in [k] \setminus I$, let $P_j' \defeq P_j$.
	It is not difficult to verify that $P_1', \ldots, P_i'$ is a solution to $(G', k, \terminals)$.
	
	Conversely, suppose that $(G', k, \terminals)$ is a \yes-instance, and let $P_1', \ldots, P_k'$ be one of its solutions,
	and let $\calF'$ be the partition forest of $G'$.
	Denote by $A$ and $C$ the bags in $\calF'$ that correspond to the neighbors of $B$ in $\calF$.
	Note that $A$ and $C$ are adjacent in $\calT'$ by construction.
	There is a pair of vertices $a_i \in \mathrm{bd}_{\calF'}(A, C)$, $c_i \in \mathrm{bd}_{\calF'}(C, A)$ 
	such that $P_i'$ contains the edge $a_i c_i$.
	Let $b^a_i \in \mathrm{bd}_{\calF}(B, A)$ and $b^c_i \in \mathrm{bd}_{\calF}(B, C)$ (possibly $b^a_i = b^c_i$). 
	Then, the path $P_i$ obtained from $P_i'$ by replacing the edge $a_i c_i$ with the path $a_i b^a_i b^c_i c_i$
	is an $(s_i, t_i)$-path in $G$.
	Moreover, since $\card{\bdd{B}{A}} \ge \card{I}$ and $\card{\bdd{B}{C}} \ge \card{I}$,
	we can assign such vertices to each of the paths $P_i'$, $i \in I$.
	Finally, for $j \in [k] \setminus I$, let $P_j \defeq P_j'$. 
	Then, $P_1, \ldots, P_k$ is a solution to $(G, k, \terminals)$.
\end{proof}

We wrap up.
\begin{theorem}\label{thm:dp:kernel}
	\probDP{} on well-partitioned chordal graphs parameterized by $k$ 
	admits a kernel on $\calO(k^3)$ vertices.
\end{theorem}
\begin{proof}
	Let $(G, k, \terminals)$ be an instance of \probDP{} such that $G$ is well-partitioned chordal,
	and let $(G, k, \terminals)$ be as asserted by Observation~\ref{obs:kernel:dp:leaves}.
	Apply Reduction~\ref{red:small:boundary} exhaustively.
	If at some stage the reduction returned a trivial \no-instance, then we conclude that $(G, k, \terminals)$ is a \no-instance,
	which is correct by Lemma~\ref{lem:small:boundary}.
	Now apply Reduction~\ref{red:large:boundary} exhaustively, and reuse $(G, k, \terminals)$ to
	denote the resulting instance, and let $\calT$ be the partition tree of $G$.
	We count the number of bags in $\calT$.
	By Observation~\ref{obs:kernel:dp:leaves}, $\calT$ has at most $2k$ leaf nodes, therefore it also has at most $2k$
	internal nodes of degree at least $3$.
	Moreover, $\calT$ has at most $2k$ nodes of degree two -- by Reductions~\ref{red:small:boundary} and~\ref{red:large:boundary},
	the only degree two nodes in $\calT$ contain terminal vertices.
	Therefore, $\card{V(\calT)} = \calO(k)$, and since each bag contains at most $\calO(k^2)$ vertices
	by Lemma~\ref{lem:srdp:marking},
	$G$ contains at most $\calO(k^3)$ vertices.
\end{proof}

We now show that \probTDP{} also admits a polynomial kernel on well-partitioned chordal graphs.
Suppose we have an instance $(G, k, \terminals)$ of \probTDP{} such that $G$ is well-partitioned chordal with partition tree $\calT$.
We reduce $G$ as follows. 
We apply Lemma~\ref{lem:srtdp:marking}, and let $\calT'$ be the forest in $\calT$ that is described there.
(Note that the construction of $\calT'$ is based only on the \emph{light} indices.)
It is clear that in our kernel, it suffices to take $\calT'$, and the union of $M_i$ over all \emph{heavy} indices $i$.
We reduce $\calT'$ according to Reductions~\ref{red:small:boundary} and~\ref{red:large:boundary}.
If no trivial \no-instance was returned, then we return the resulting subgraph of $G$ as our kernelized instance.
We can apply roughly the same argument as given in the proof of Theorem~\ref{thm:dp:kernel}
to conclude that the subgraph of $G$ induced by vertices in the bags of $\calT'$ after reduction, has at most $\calO(k^3)$ vertices.
Since the number of vertices that have been marked for heavy edges is at most $\calO(k^2)$ in total,
we have the following theorem.
\begin{theorem}\label{thm:tdp:kernel}
	\probTDP{} on well-partitioned chordal graphs parameterized by $k$ 
	admits a kernel on $\calO(k^3)$ vertices.
\end{theorem}
Note that in resulting kernelized instance, the graph may not be connected. 
However, it is not difficult to see that with a slight modification of the above described procedure,
one can obtain a connected graph in the kernel.

\section{Conclusions}\label{sec7}

In this paper, we introduced the class of \emph{well-partitioned chordal graphs}, 
a subclass of chordal graphs that generalizes split graphs.
We provided a characterization by a set of forbidden induced subgraphs which also gave a polynomial-time recognition algorithm.
We showed that several parameterized variants and generalizations of the \probDP{} problem 
that on chordal graphs are only known to be in \XP{}, are in \FPT{} on well-partitioned chordal graphs.
These results in some cases implied polynomial kernels on split and well-partitioned chordal graphs.
It would be interesting to see for which problems well-partitioned chordal graphs can be used to narrow down complexity gaps
for problems that are hard on chordal and easy on split graphs,
or for which problems that are easy on split graphs and whose complexity is open on chordal graphs, 
well-partitioned chordal graphs can be used to obtain partial (positive) results.

Another typical characterization of (subclasses of) chordal graphs is via vertex orderings.
For instance, chordal graphs are famously characterized as the graphs admitting perfect elimination orderings~\cite{FulkersonGross1965}.
It would be interesting to see if well-partitioned chordal graphs admit a concise characterization in terms of vertex orderings as well.
While the degree of the polynomial in the runtime of our recognition algorithm is moderate,
our algorithm does not run in linear time. 
We therefore ask if it is possible to recognize well-partitioned chordal graphs in linear time;
and note that a characterization in terms of vertex orderings can be a promising step in this direction.

\bibliographystyle{plain}
\bibliography{ref}
\end{document}